\documentclass[12pt,a4paper,reqno,twoside]{amsart}

\usepackage{amsmath,amsfonts,amssymb,enumerate,stmaryrd,verbatim,latexsym}

\usepackage{iftex}
\ifpdftex
  \usepackage{newtxtext}
  \usepackage{newtxmath}
\else
  \usepackage{fontspec}
\fi

\DeclareMathAlphabet{\mathcal}{OMS}{cmsy}{m}{n}

\usepackage[english]{babel}
\usepackage{romannum,booktabs}
\usepackage[pdfencoding=unicode,psdextra,colorlinks=true,
            citecolor=blue,linkcolor=blue,urlcolor=black]{hyperref}
\usepackage{bookmark}
\usepackage{bbm}
\usepackage{pdflscape}
\usepackage{amsthm}
\usepackage{xr-hyper}
\usepackage{longtable}
\usepackage{graphicx}

\newcommand{\dotPlus}{
  \mathbin{\overset{\boldsymbol{\cdot}}{+}}
}

\newcommand{\myplus}{\mathbin{\scalebox{1}{$\bigoplus$}}}
\newcommand{\plus}{\mathop{\myplus}\limits}
\newcommand{\plusn}{\mathop{\myplus}\nolimits}

\newcommand{\mytensor}{\mathbin{\scalebox{1.1}{$\bigotimes$}}}

\newcommand{\tensorn}{\mathop{\textstyle\bigotimes}\nolimits}

\newcommand{\llongrightarrow}{\relbar\joinrel\longrightarrow}

\theoremstyle{plain}
 \newtheorem*{stabtheorem}{Stabilization Theorem}
 \newtheorem{theorem}{Theorem}[section]     
 \newtheorem{corr}[theorem]{Corollary}       
 \newtheorem{lmr}[theorem]{Lemma}
 \newtheorem{prr}[theorem]{Proposition}
 \newtheorem{hyp}[theorem]{Conjecture}

\theoremstyle{definition}
 \newtheorem{defnr}[theorem]{Definition}
 \newtheorem{remnr}[theorem]{Remark}
 
\def\brm{\begin{remnr}}
\def\erm{\end{remnr}}
\def\bdr{\begin{defnr}}
\def\edr{\end{defnr}}
\def\bp{\begin{proof}}
\def\ep{\end{proof}}
\def\btr{\begin{theorem}}
\def\etr{\end{theorem}}
\def\bpr{\begin{prr}}
\def\epr{\end{prr}}
\def\bcr{\begin{corr}}
\def\ecr{\end{corr}}
\def\blr{\begin{lmr}}
\def\elr{\end{lmr}}
\def\beq{\begin{equation}}
\def\eeq{\end{equation}}

\def\bcs{\begin{cases}}
\def\ecs{\end{cases}}
\def\G{\Gamma}
\def\Ker{\mathrm{Ker}}
\def\Hom{\mathrm{Hom}}
\def\Ext{\mathrm{Ext}}

\author{F.V.WEINSTEIN}

\title{integral cohomology of classical artin groups:\\
spherical types $A,C,D$ and affine types
$\widetilde B,\widetilde C,\widetilde D$}

\address
{Giacomettistrasse 4, CH-3006, Bern, Switzerland.}

\email{felix.weinstein46@gmail.com}

\begin{document}

\makeatletter
\@namedef{subjclassname@2020}{\textup{2020} Mathematics Subject Classification}
\makeatother

\pagenumbering{arabic}

\begin{abstract}
We determine the additive structure of the integral cohomology
groups with trivial coefficients of the Artin groups in the
spherical families $A_n,C_n,D_n$ and the affine families
$\widetilde B_n,\widetilde C_n,\widetilde D_n$.
We first compute the integral cohomology of braid groups
by algebraic methods and then use the Salvetti complexes
to express the cohomology of the remaining families
in terms of these computations.
The resulting descriptions apply in every admissible rank.
In particular, for the spherical families $A$, $C$, and $D$,
we recover the previously known integral cohomology groups.
For each family, we establish cohomological stabilization
and show that the integral cohomology groups contain
no elements of order $p^2$ for any prime $p$.
We also give uniform generating series formulas describing
the stable integral cohomology in terms of that of
the braid groups.
\end{abstract}

\maketitle

\markboth{F.V.WEINSTEIN}{INTEGRAL COHOMOLOGY OF CLASSICAL ARTIN GROUPS}

\pagenumbering{arabic}

\section*{Introduction}

In this paper, we determine the integral cohomology groups
of the Artin groups in the spherical families $A_n,C_n,D_n$
and the affine families
$\widetilde B_n,\widetilde C_n,\widetilde D_n$.
Our computations using the Salvetti complexes express
the cohomology groups of the remaining families in terms
of those of the braid groups.
To the best of my knowledge, a complete description
of the integral cohomology groups with trivial coefficients
in every admissible rank has not previously been given
for the affine families
$\widetilde B_n,\widetilde C_n,\widetilde D_n$.
The resulting descriptions determine the integral cohomology
groups in each admissible rank and provide a common framework
for studying their stabilization.
The affine series $\widetilde{A}_n$ is not considered here, since the
cyclic structure of its Coxeter graph leads to a substantially different
combinatorial problem, to which the methods used in this paper do not
apply directly.

\subsection*{Historical Background}
The study of the integral cohomology of the braid groups $\mathbf{B}_n$ on $n$ strands,
denoted by $H^*\bigl(\mathbf{B}_n\bigr)$,
began with the pioneering work of V.~I.~Arnold \cite{Arnold} in 1970.
In that work, he discovered the finiteness, repetition, and stabilization
of these cohomology groups and carried out several explicit calculations.

The first general result was obtained in 1970 by D.~B.~Fuchs,
who gave a comprehensive description of the algebra
$H^*\bigl(\mathbf{B}_n;\mathbb{F}_2\bigr)$ in \cite{Fuchs}.
Since a classifying space for $\mathbf{B}_n$ (the space of monic polynomials without multiple roots)
is an open manifold, he constructed a cellular decomposition of its one-point compactification
and extracted the cohomology of the braid group via Alexander duality.

In 1976, F.~R.~Cohen computed the homology of braid groups
with coefficients in $\mathbb F_p$, $\mathbb Q$, and
$\mathbb Z$ (see~\cite[Appendix, Theorem~A.1 and
Corollaries~A.3--A.4]{C3}).
These results also determine the additive structure
of their integral cohomology groups.
An independent description of the additive structure
of $H^*(\mathbf B_n)$ using the Fuchs complex
was given in my note \cite{FW}.

In connection with these results, it is worth highlighting
the remarkable work of G.~Segal~\cite{Seg},
which implies the existence of a homology equivalence
$\mathrm{B}\mathbf{B}_\infty \to \Omega^2S^3$. Here, $\Omega^2S^3$
denotes the double loop space of the $3$-sphere, and $\mathrm{B}\mathbf{B}_\infty$
is the classifying space of the group $\mathbf{B}_\infty:=\varinjlim\mathbf{B}_n$,
with the direct limit taken over the natural embeddings $\mathbf{B}_n\hookrightarrow\mathbf{B}_{n+1}$.
This yields a complete description of $H_*(\mathbf{B}_\infty)$, since the structure of the
Hopf algebra $H_*(\Omega^2S^3)$ is well known.

In 1971, E.~Brieskorn introduced the concept of the Artin group
$\mathbb{A}(W)$ associated with a Coxeter group $W$, thereby generalizing
the braid groups to the setting of arbitrary Coxeter groups.
Furthermore, for any finite Coxeter group $W$, Brieskorn canonically defined
a finite-dimensional manifold $\mathbf{X}(W)$ and proved that its fundamental group
is isomorphic to the group $\mathbb{A}(W)$ (see \cite{Brieskorn1}).
In addition, Brieskorn conjectured that $\mathbf{X}(W)$ is a~space
of homotopy type $K(\mathbb{A}(W);1)$ for all finite Coxeter groups $W$
and verified this for some of them (see \cite{Bries1}).
The Brieskorn hypothesis was proven for all finite Coxeter groups by P.~Deligne in 1972.

The definition of the Brieskorn space $\mathbf{X}(W)$ naturally generalizes to
all Coxeter groups thanks to the results of J.~Tits, who
defined the faithful canonical linear representation of any
Coxeter group of finite rank in a finite-dimensional
real vector space and described its properties.
In \cite{Lek}, H.~van der Lek established that in this
generalization, $\pi_1(\mathbf{X}(W))\cong\mathbb{A}(W)$.

The $K(\pi,1)$ conjecture, which generalizes Brieskorn's
conjecture, asserts that $\mathbf X(W)$ has homotopy type
$K(\mathbb A(W),1)$ for every Coxeter group $W$.
If this hypothesis holds, there is an isomorphism $H^*(\mathbb{A}(W))\cong H^*(\mathbf{X}(W))$.

In \cite{Salvetti} and \cite{CSA}, an approach was proposed
that reduces the description of $H^*(\mathbf{X}(W))$ to solving a~problem
related only to the algebraic structure of the group $W$.
Namely, in 1994, M.~Salvetti~\cite{Salvetti}
constructed a~$CW$-complex that is homotopy equivalent to the space
$\mathbf{X}(W)$ for a~finite Coxeter group $W$, and described the corresponding
algebraic complex $\mathbf{S}^*(W)$ (the \textit{Salvetti complex})
in terms of special parabolic subgroups of the group $W$.
With its help, Salvetti found the integral cohomology of the groups
$\mathbb{A}(W)$ for all exceptional finite Coxeter groups $W$.
In 1996, C.~De Concini and M.~Salvetti~\cite{CSA}
extended Salvetti's construction to all Coxeter groups.

The $K(\pi,1)$ conjecture for affine Artin groups
was proved in full generality by Paolini and
Salvetti~\cite{AKPi1}.
Consequently, the Salvetti complexes compute the group
cohomology for all families considered in this paper.

\subsection*{An Algebraic Approach for Classical Families}

The Salvetti complex encodes all the geometric information
needed for the cohomology computations considered here.
Once this complex is available, the computations proceed
entirely algebraically, without further geometric arguments.

One of the motivations for this paper was to revisit the
description of the additive structure of $H^*(\mathbf B_n)$
given in \cite{FW} and to make the algebraic
basis of those calculations more explicit.
The algebraic starting point is the Fuchs complex for braid
groups. Let
\[
\mathrm E[z]:=\Lambda[z^{(1)}]\otimes\Gamma[z],
\]
where $\Lambda[z^{(1)}]$ is the exterior algebra on $z^{(1)}$
and $\Gamma[z]$ is the usual divided power algebra.
We assign weight $1$ to $z^{(1)}$ and weight $2$ to $z$.
We identify the Fuchs complex for $\mathbf B_n$, with the
appropriate reindexing of degrees, with the weight-$n$
component of the normalized bar complex of
$\mathrm E[z]$ with trivial coefficients.
The homology $H_*(\mathrm E[z])$ can, in turn, be canonically
expressed in terms of the homology $H_*(\Gamma[z])$.\footnote{
The relationship between the cohomology of $\mathbf B_n$
with coefficients in a nontrivial $\mathbf B_n$-module
and the homology of the divided power algebra $\Gamma[z]$,
with $z$ assigned weight $1$, was first noted by N.~S.~Markaryan
in \cite{Mark}. This observation follows immediately
from the Fuchs complex and shed new light on
calculations in \cite{FW}.}

We begin by revisiting the braid groups.
Sections~\ref{FuchsC} through \ref{Z} are devoted to describing the
additive structure of the algebras $H^*(\mathbf{B}_n)$.
Throughout the paper, the grading on $\mathrm{E}[z]$ refers
to the internal grading determined by the indices of the divided powers,
whereas the degree in $H_*(\mathrm{E}[z])$ is homological. As we shall see,
the interplay between these two gradings is crucial for understanding
the stabilization phenomena in $H^*(\mathbf{B}_n)$.

As shown in Section~\ref{FuchsC}, the
Fuchs complex for $\mathbf B_n$ is the Salvetti complex
of type $A_{n-1}$ written in a particular basis.
The description of its cohomology in terms of the
homology of $\Gamma[z]$, together with the weight
decomposition and the stabilization operator, provides
the algebraic input for studying the Salvetti complexes
of the remaining families. Decompositions of these
complexes allow the braid-group calculations to be
used in different forms to determine the integral
cohomology of the other families considered here.

Although Sections~\ref{IsoA} to~\ref{Intcoh1} primarily
recover known results on the cohomology of braid groups,
they develop the algebraic tools used throughout the
subsequent computations. The results of Section~\ref{IsoA} recover, in particular, Arnold's
theorems on finiteness, repetition, and stabilization for the cohomology
of braid groups. To this end, we construct the canonical
injections
\[
i_n^*\colon H^q(\mathbf{B}_n)\hookrightarrow H^q(\mathbf{B}_{n+1})
\]
and prove that, for each fixed $q$, they are isomorphisms for all sufficiently large $n$.

We also show that the modules
$H^*(\mathbf{B}_\infty)$ and $H_*(\G[z];\mathbb{Z})$ are canonically
isomorphic, though not as graded modules, and give an explicit description
of the embedding
\[
H^*(\mathbf{B}_n)\subset H_*(\G[z];\mathbb{Z}).
\]
In addition, we introduce a canonical \emph{stabilization operator} $\mathrm{St}$
acting on the Fuchs complexes; this operator plays an important role in what follows.
Using the operator $\mathrm{St}$, one can show that
$H^q(\mathbf{B}_\infty)\cong\varinjlim\nolimits_{i_n^*} H^q(\mathbf{B}_n)$.
The stabilization $i_n^*$ implies, in particular, that each graded component of
$H^*(\mathbf{B}_\infty)$ is finitely generated.

In Section~\ref{Bpm}, we study the homology of the $\mathbb{Z}$-algebra $\G[z]$.
Since this algebra is commutative, its homology carries
a natural multiplicative structure known as the Pontryagin product.
We describe the corresponding Pontryagin algebra structures on $H_*(\G[z];\mathbb{F}_p)$
and $H_*(\G[z];\mathbb{Z})$.
In particular, we show that $H_*(\G[z];\mathbb{Z})$ is a \emph{squarefree-torsion group};
that is, an abelian group without elements of order $p^2$ for any prime $p$.
For each prime $p$, the graded dual of
$H_*(\G[z];\mathbb F_p)$ is an associative,
graded-commutative, primitively generated Hopf algebra.
It is endowed with a natural system of canonical generators
that provide an explicit description of both the multiplication
and comultiplication.

Section~\ref{Intcoh1} combines the results of
Sections~\ref{IsoA} and~\ref{Bpm} to provide a complete
description of the graded abelian group
$H_*(\mathbf B_\infty)$ and the $\mathbb Z$-modules
$H^*(\mathbf B_n)$, together with explicit expressions
for their canonical generators.
The additive identification of $H^*(\mathbf B_\infty)$
with $H_*(\G[z];\mathbb Z)$ also equips the former
with a multiplication inherited from the homology
of the commutative algebra $\G[z]$.
With coefficients in $\mathbb F_p$, we describe
an additional Hopf algebra structure explicitly.
The relation between the algebraic multiplication
and the usual cup product is discussed
in Remark~\ref{SegE}.

Since the groups $H^*(\mathbf{B}_n)$ are squarefree-torsion, it is sufficient to know the
generating polynomials for all primes $p<n$ as well as for $p=0$ to completely determine $H^q(\mathbf{B}_n)$.
For a given $p$, these polynomials for $n\geqslant 2$ are defined as
\[
\mathrm{G}_p(\mathbf{B}_n; t):=\sum_{q\geqslant 0}
\dim_{\mathbb{F}_p}\bigl(H^q(\mathbf{B}_n)\otimes\mathbb{F}_p\bigr)t^q \quad (\text{for } p>0),
\qquad\text{and}\qquad\mathrm{G}_0(\mathbf{B}_n; t):=1+t.
\]
In Section~\ref{Z}, we compute $\mathrm{G}_p(\mathbf{B}_n; t)$ for $p>0$,
thereby providing an algorithm to describe the (co)homology groups of $\mathbf{B}_n$.
This computation also shows that for $p>0$ we have:
\beq\label{GenB}
\mathrm{G}_p(\mathbf{B}_\infty; t)=
\sum_{q\geqslant 0}\dim_{\mathbb{F}_p}\big(H^q(\mathbf{B}_\infty)\otimes\mathbb{F}_p\big)t^q
=1+t\prod_{r\geqslant 1}\;\frac{1+t^{2p^r-1}}{1-t^{2p^r-2}}\,.
\eeq

\subsection*{Conventions and Notation} 
In this article, by \emph{classical Coxeter groups}
(and respectively, \emph{classical Artin groups}),
we mean the infinite spherical families $A,C,D$ and infinite
affine families $\widetilde{A}, \widetilde{B}, \widetilde{C}$, and $\widetilde{D}$.

Regarding notation, the standard braid group $\mathbb{A}(A_{n-1})$ is
traditionally denoted by $\mathbf{B}_n$. In what follows, for all other
classical Coxeter groups $W_n$, the corresponding Artin group $\mathbb{A}(W_n)$
is denoted by $\mathbf{W}_n$. For example, the group $\mathbb{A}(C_n)$ is denoted by $\mathbf{C}_n$.

\subsection*{Integral Cohomology of the Classical Artin Group Families}
Our subsequent description of the additive structure of the integral cohomology
of classical Artin groups relies on the results obtained for $H^*(\mathbf{B}_n)$,
demonstrating that the cohomology structure of the braid groups plays a central role throughout our argument.

The integral cohomology groups of $\mathbf{C}_n$ and
$\widetilde{\mathbf{C}}_n$ are described in
Section~\ref{Cn}. For the spherical family
$\mathbf{C}_n$, our computation recovers Goryunov's
result~\cite{Gor}.

Section~\ref{ADn} gives an independent algebraic
computation of the integral cohomology groups of
$\mathbf{D}_n$. For $q<n$, the resulting formula is
equivalent to Goryunov's formula~\cite{Gor}, while
the top-degree groups agree with those recorded
in his table in \cite{Gor}; see Remark~\ref{RmGor}.

In Section~\ref{BBDn}, we describe the integral cohomology of the groups $\widetilde{\mathbf{B}}_n$.

Relying on the techniques and results established in Section~\ref{ADn},
we describe the integral cohomology of the groups $\widetilde{\mathbf{D}}_n$ in Section~\ref{BBDDn}.

\subsection*{Cohomological Stabilization for Classical Artin Groups}
As a consequence of our results, we establish that, for every classical
Artin group family considered in this paper, the cohomology groups
$H^q(\mathbf W_n)$ canonically stabilize for fixed $q$ as $n$
increases.

Namely, we prove that for any fixed $q\geqslant 0$ and any admissible $n$
there exists a canonical homomorphism
$\mathrm{St}^*_W:H^q(\mathbf{W}_{n+1})\longrightarrow H^q(\mathbf{W}_n)$,
which is an isomorphism if $n\geqslant\max\{6,2q+1\}$.
We write $H^q(\mathbf W_\infty)$ for the stable value
of $H^q(\mathbf W_n)$ under the stabilization maps
constructed here. This notation does not presuppose
an identification with the cohomology of a direct limit group.

Our results also imply that the cohomology groups $H^*(\mathbf{W}_n)$
are squarefree-torsion for all classical families considered in this paper.
Therefore, to completely determine $H^q(\mathbf{W}_\infty)$,
it suffices---as in the case of the braid groups---to know the generating series of $\mathbf{W}_\infty$ modulo $p$
for all primes $p$, as well as for $p=0$ (under the convention $\mathbb{F}_0:=\mathbb{Q}$).
For a given $p$, this series is defined as
\[
\mathrm{G}_p(\mathbf{W}_\infty; t):=\sum_{q\geqslant 0}
\dim_{\mathbb{F}_p}\bigl(H^q(\mathbf{W}_\infty)\otimes\mathbb{F}_p\bigr) t^q.
\]
\begin{stabtheorem}
For every family $\mathbf W$ displayed in the table below,
\[
\mathrm{G}_p(\mathbf{W}_\infty; t)=\mathrm{P}_p(\mathbf{W},t)\;\mathrm{G}_p(\mathbf{B}_\infty; t)
-
\bcs
0&\text{if $p\neq 2$},\\
\mathrm{R}(\mathbf{W},t)&\text{if $p=2$},
\ecs
\]
where $\mathrm{P}_p(\mathbf{W},t)$ and $\mathrm{R}(\mathbf{W},t)$ are the
rational functions from the following table:
\begin{center}
\renewcommand{\arraystretch}{2}
\begin{tabular}{|c||c|c|c|c|c|} \hline
$\mathrm{P}\backslash\mathbf{W}$ & $\mathbf{C}_\infty$ & $\mathbf{D}_\infty$ & 
$\widetilde{\mathbf{B}}_\infty$ & $\widetilde{\mathbf{C}}_\infty$ & $\widetilde{\mathbf{D}}_\infty$ \\ \hline\hline
$\mathrm{P}_{p\neq 2}(\mathbf{W},t)$ & $\frac{1}{1-t}$ & $1$ & $\frac{1}{1-t}$ & $\frac{1}{(1-t)^2}$ & $1$ \\ \hline 
$\mathrm{P}_2(\mathbf{W},t)$ & $\frac{1}{1-t}$ & $\frac{1-t+t^2}{1-t}$ & $\frac{1-t+t^2}{(1-t)^2}$
& $\frac{1}{(1-t)^2}$ & $\left(\frac{1-t+t^2}{1-t}\right)^2$ \\ \hline 
$\mathrm{R}(\mathbf{W},t)$ & $0$ & $\frac{t^2}{1-t}$ & $\frac{t^2}{(1-t)^2}$
& $0$ & $\left(\frac{1-t+t^2}{1-t}\right)^2-1$ \\ \hline 
\end{tabular}
\end{center}
\end{stabtheorem}
This theorem synthesizes Corollaries
\ref{StabC}, \ref{StabCC}, \ref{StableD}, \ref{StabB}, and \ref{StableDD}
into a unified formula.

For the families $A$, $C$, and $\widetilde C$, the
cohomology groups at every finite rank embed canonically
into the corresponding stable cohomology groups.
For the remaining families, many nonzero cohomology
classes at finite rank are annihilated by the
stabilization maps.

\subsection*{Organization of the Paper.}
This paper is divided into three main parts:
\begin{enumerate}
    \item \emph{Cohomology of the braid groups }($\mathbf{B}_n$).
    This is the focus of Sections 2 through 6, where we establish our algebraic framework.
    \item \emph{Cohomology of other classical Artin families.} 
    Sections 7 through 10 are devoted to the description of the integral cohomology of the groups
    $\mathbf{C}_n, \widetilde{\mathbf{C}}_n, \mathbf{D}_n, \widetilde{\mathbf{B}}_n$, and $\widetilde{\mathbf{D}}_n$.
    \item \emph{Foundational material.}
    To make the paper maximally self-contained and avoid excessive
    cross-referencing, we have included Section~1 and the Appendix, which
    provide the necessary background and establish the notation used throughout
    the paper.
\end{enumerate}

Section~1 fixes the conventions and notation for Coxeter and Artin groups
and recalls the algebraic Salvetti complex.  The Appendix records the
conventions for the normalized bar complex and the
Eilenberg--Zilber and Alexander--Whitney maps used in the paper.

\subsection*{Use of AI tools}
During the preparation of this manuscript, the author used
ChatGPT (OpenAI) for language editing and for assistance
in reviewing mathematical arguments, identifying possible
errors, and refining parts of the exposition and proofs.
The author reviewed the suggestions incorporated into
the manuscript and takes full responsibility for its content.
I also used ChatGPT for assistance in locating and clarifying bibliographic references.

\section{Coxeter Groups, Artin Groups, and the Salvetti Complex}\label{sec0}

Let $S$ be a finite set. Suppose that for every pair $s,t\in S$, a number
$m_{s,t}=m_{t,s}\in \{1,2,\dots\}\cup\{\infty\}$ is given such that
$m_{s,t}=1$\; if and only if\; $s=t$.

The group defined by the presentation
\[
W=
\left\langle
S \ \middle| \ s^2=1 \text{ for } s\in S, \quad
(st)^{m_{s,t}}=1 \text{ for } s\neq t \text{ and } m_{s,t}<\infty
\right\rangle
\]
is called a \emph{Coxeter group}, and the pair $(W,S)$ is called a \emph{Coxeter system}.

To every Coxeter system $(W,S)$ one associates a \emph{Coxeter graph}
$\G(W,S)$, whose vertices correspond to the generators $s\in S$.
When the set $S$ is understood, we shall write simply $\G(W)$.
Distinct vertices $s$ and $t$ are joined by an edge whenever
$m_{s,t}\geqslant 3$ or $m_{s,t}=\infty$. If $m_{s,t}=3$, the edge is
left unlabeled; otherwise it is labeled by $m_{s,t}$. No edge is drawn
if $m_{s,t}=2$.
A Coxeter group is said to be \emph{irreducible} if its Coxeter graph is
connected.

If ${\G_1,\dots,\G_m}$ are the labeled connected components of $\G(W)$, then
each $\G_i$ corresponds to a~Coxeter group $W_i$, and
$W\cong W_1\times\dots\times W_m$ is a~direct product of groups
(see \cite{Humphreys}, 2.2).
The Coxeter group is called \emph{spherical} if it is finite.
It is called \emph{affine} if it is infinite and contains a normal abelian subgroup
$T$ such that the quotient $W\big/T$ is finite.

Irreducible spherical and affine Coxeter groups
are completely classified (see~\cite{Humphreys}).
The classical families consist of the spherical types
$A, C, D$ and the affine types
$\widetilde{A},\widetilde{B},\widetilde{C},\widetilde{D}$.
The Coxeter graphs corresponding to these groups have one of the
following forms, where the generators $s_i\in S$ are denoted simply by
their indices $i$:
\smallskip

\noindent\makebox[\textwidth][c]{\includegraphics{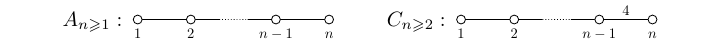}}
\noindent\makebox[\textwidth][c]{\includegraphics{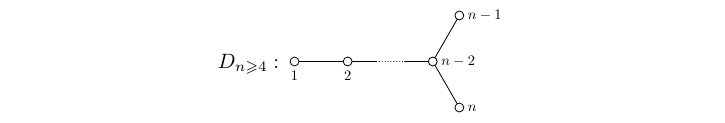}}
\noindent\makebox[\textwidth][c]{\includegraphics{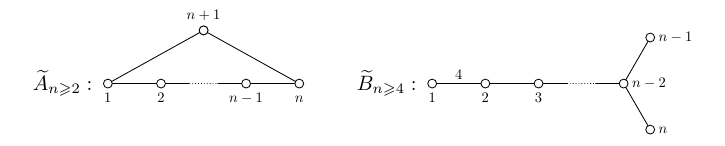}}
\noindent\makebox[\textwidth][c]{\includegraphics{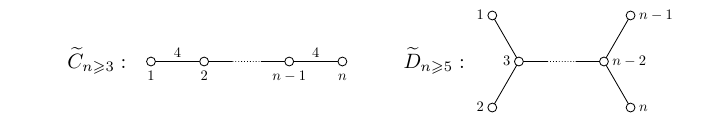}}

Any element $w$ of a Coxeter group $W$ can be written as a~product
$w=s_1\cdots s_a$ of elements from $S$.
The \textit{length $l(w)$ of $w$}
is defined as the minimal integer $a$ among all such expressions.
Since $|S|<\infty$, the power series
\[
W(t)=\sum_{w\in W}t^{l(w)}
\]
is well-defined. It is called the \textit{Poincaré series of $W$}.
Poincaré series are multiplicative with respect to direct products (see \cite{Humphreys}):
\[
\left(W_1\times W_2\times\dots\times W_r\right)(t)=W_1(t)\cdot W_2(t)\cdot{\dots}\cdot W_r(t).
\]
If $W$ is finite, then $W(t)$ is a polynomial, called the
\textit{Poincaré polynomial of $W$}.
For the groups $W=A_n,C_n,D_n$, these polynomials are given by
\[
W(t)=\prod_{i=1}^n\frac{t^{d_i}-1}{t-1}\;,
\]
where the values of $d_i$ are listed in the following table (see \cite{Humphreys}, p.59).
This table will be used repeatedly in the computations below:
\begin{center}\label{tabd}
\begin{tabular}{|l||l|}
\hline
$W$ & $d_1,d_2,d_3,\dots,d_{n-1},d_n$\\
\hline\hline
$A_n$ & $2,3,4,\dots,n,n+1$\\
\hline
$C_n$ & $2,4,6,\dots,2n-2,2n$\\
\hline
$D_n$ & $2,4,6,\dots,2n-2,n$\\
\hline
\end{tabular}
\end{center}
\smallskip

Let $\Pi(x,y:r)$ denote the alternating product $xyx\cdots$ of $x$ and $y$ of total length $r>0$.
\bdr[Brieskorn \cite{Brieskorn1}]
The \textit{Artin group} $\mathbb{A}(W)$ associated with a Coxeter system $(W,S)$
is the group generated by symbols $\sigma_s$, where $s\in S$, subject to the relations
\[
\Pi\left(\sigma_s,\sigma_{s^\prime}:m_{s,s^\prime}\right)=
\Pi\left(\sigma_{s^\prime},\sigma_s:m_{s,s^\prime}\right),\;
\text{where $s,s^\prime\in S,\;s\neq s^\prime,\;m_{s,s^\prime}\neq\infty$}.
\]
We say that the Artin group $\mathbb{A}(W)$ is \emph{spherical}
or \emph{affine} when the corresponding Coxeter group
$W$ is spherical or affine, respectively.
\edr

The geometric construction linking Artin groups to Coxeter groups
is essentially due to Brieskorn.
It relies on the canonical faithful
linear representations of Coxeter groups defined by
J.~Tits (see~\cite{Bourb1, Humphreys}). Let us recall his results.
A nontrivial linear involution of a~real vector space $V$
is called a reflection across a hyperplane if it fixes
pointwise a codimension-one subspace (the hyperplane)
and acts by $-1$ on a complementary normal direction.
For any Coxeter system $(W,S)$, Tits constructed a faithful geometric
representation
\[
\tau_W\colon W\longrightarrow GL(V),
\qquad V\cong\mathbb{R}^{|S|}.
\]
Under this representation, the generators $s\in S$
act as linear reflections in the walls of a closed
simplicial chamber $C\subset V$.
Set
\[
I_W:=\operatorname{int}\Bigl(\bigcup_{w\in W}w(C)\Bigr).
\]
Thus, $I_W$ is the interior of the Tits cone.
It is a nonempty $W$-invariant open convex cone,
and $W$ acts properly discontinuously on it.
Let $\mathcal H(W)$ denote the collection of reflecting
hyperplanes. The action of $W$ on
\[
I_W\setminus\bigcup_{H\in\mathcal H(W)}H
\]
is free.
For spherical Coxeter systems, $I_W$ is the entire space
$\mathbb R^{|S|}$; for irreducible affine Coxeter systems,
it is an open half-space.

\smallskip

The next statement is a key result that generalizes the results of \cite{FN, Brieskorn1, Nguyen}
and connects the Tits representation with Artin groups:
\smallskip

\btr[H. van der Lek \cite{Lek}]
For the Coxeter group $W$, let
\[
\mathbf{X}(W):=\Big((I_W\times I_W)\smallsetminus
\bigcup_{H\in\mathcal{H}(W)}(H\times H)\Big)\Big{/}W,
\]
where the action of $W$ on $I_W\times I_W$ is diagonal.
Then $\pi_1\bigl(\mathbf X(W)\bigr)\cong\mathbb A(W)$.
\etr
\smallskip

The authorship of the following fundamental $K(\pi,1)$-conjecture for Artin groups is attributed to
Arnold, Brieskorn, Pham, and Thom (see \cite{AFST}, p.362).
\smallskip

\begin{hyp}\label{Hyp_A}
For any Coxeter group $W$,
the space $\mathbf{X}(W)$ is homotopically equivalent to
the Eilenberg--MacLane space $K(\mathbb A(W),1)$.
\end{hyp}
\smallskip

The validity of this hypothesis implies the isomorphism
\[
H^*(\mathbb A(W);\mathbb Z)\cong H^*(\mathbf X(W);\mathbb Z),
\]
where $\mathbb{Z}$ is considered as a~trivial $\mathbb A(W)$-module.
Conjecture~\ref{Hyp_A} has been verified in many instances
(see~\cite{AFST}), but in the general case, it remains open.
We need only the cases in which $W$ is a spherical or an affine Coxeter group.
In these cases, the conjecture was established by P.~Deligne in the spherical case~\cite{Dl}
and by G.~Paolini and M.~Salvetti in the affine case~\cite{AKPi1}.

A fundamental result describing the cohomology of
the space $\mathbf{X}(W)$ was proved
by Salvetti~\cite{Salvetti} for spherical groups and subsequently
extended by De~Concini and Salvetti to all Coxeter groups~\cite{CSA}.
We shall require only a particular case of it, which we formulate as follows.

Let $(W,S)$ be a Coxeter system and let $P\subset S$.
Denote by $W_P$ the subgroup of $W$ generated by $P$.
It is known that $(W_P,P)$ is itself a Coxeter system
(see \cite{Humphreys}, p.~113). The subgroup $W_P$ is called
the \emph{(standard) parabolic subgroup} corresponding to $P$.

\btr\label{Sal}
Let $(W,S)$ be a Coxeter system.
Then the graded group $H^*(\mathbf{X}(W);\mathbb{Z})$
is isomorphic to the cohomology of the cochain
complex $\mathbf{S}(W)$ defined for $q\geqslant 0$ as follows:
\[
\mathbf{S}^q(W)=
\bigoplus_{P\subset S,\; |P|=q,\; |W_P|<\infty}
\mathbb Z\cdot P .
\]
After fixing a linear order on $S$, the differential of the complex $\mathbf{S}(W)$ is given by
\[
\partial_W(P)=
\sum_{s\in S\setminus P,\; |W_{P\cup\{s\}}|<\infty}
(-1)^{\epsilon(s,P)}
\left.\frac{W_{P\cup\{s\}}(t)}{W_P(t)}\right|_{t=-1}
\,(P\cup\{s\}),
\]
where
\[
\epsilon(s,P)=
\bigl|\{\,t\in P\mid t<s\,\}\bigr|,
\]
and $W_P(t)$ denotes the Poincaré polynomial of the finite Coxeter group $W_P$.
\etr

\bdr\label{DefSal}
The complex $\mathbf S(W)$, defined in Theorem~\ref{Sal},
is called the algebraic \textit{Salvetti complex} of the Coxeter group
$W$.
\edr

For every spherical or affine Coxeter group $W$,
the $K(\pi,1)$ results recalled above therefore give
\[
H^*(\mathbb A(W);\mathbb Z)\cong H^*(\mathbf S(W)).
\]

\noindent
\textsc{Conventions and List of Notations:}
\smallskip

Unless otherwise stated, the term \emph{module} refers to a
$\mathbb Z$-module. For (co)homology with coefficients in
$\mathbb Z$, the notation $H^*(X;\mathbb Z)$ or
$H_*(X;\mathbb Z)$ will generally be abbreviated to $H^*(X)$
and $H_*(X)$, respectively.

Throughout the main body of the paper, unless explicitly stated otherwise,
for an augmented graded $\mathbb Z$-algebra $A$ the notation
$C_*(A)$ denotes the normalized bar complex $C_*(A;\mathbb Z)$
with trivial coefficients, and $H_*(A)$ denotes its homology.
\medskip

\noindent
\begin{longtable}{@{} l c p{0.75\textwidth} @{}}
$|X|$ & --- & The cardinality of a finite set $X$. \\ \addlinespace[0.5ex]

$\llbracket n\rrbracket$ & --- & The ordered set
$\{1,2,\dots,n\}$ if $n\geqslant 1$, and $\llbracket 0\rrbracket:=\emptyset$. \\ \addlinespace[0.5ex]

$\mathbb{F}_p$ & --- & The field of $p$ elements if $p>0$ is prime,
and the field $\mathbb{Q}$ of rational numbers if $p=0$. \\ \addlinespace[0.5ex]

$\mathbb{Z}_p$ & --- & The cyclic group of order $p$ if $p>0$ is prime,
and the ring of integers $\mathbb{Z}$ if $p=0$. \\ \addlinespace[0.5ex]

$\chi(r)$ & --- & The parity factor defined by $\chi(r):=1+(-1)^r$. \\ \addlinespace[0.5ex]

$W(t)$ & --- & The Poincaré polynomial of a finite Coxeter group $W$.\\ \addlinespace[0.5ex]

$\langle P\rangle$ & --- & The parabolic group $W_P$, when the ambient group $W$ is clear.\\ \addlinespace[0.5ex]

$\mathrm{F}(n)$ & --- & The Salvetti complex of the Coxeter group $F_n$
belonging to the classical family $F\neq A$.\\ \addlinespace[0.5ex]

$\partial_F$ & --- & The differential of the complex $\mathrm{F}(n)$.\\ \addlinespace[0.5ex]

$\mathrm{St}_F$ & --- & The homomorphism of complexes $\mathrm{St}_F:\mathrm{F}(n)\longrightarrow\mathrm{F}(n-1)$
called the \emph{stabilization operator}. We set $\mathrm{St}:=\mathrm{St}_A$.\\ \addlinespace[0.5ex]
$G[m]$ & --- & For an abelian group $G$ and an integer $m\geqslant1$,
$G[m]:=\{g\in G\mid mg=0\}$, the subgroup of $m$-torsion elements.
\end{longtable}

\bdr\label{ElGr}
An abelian group is called \emph{squarefree-torsion} if it has no
elements of order $p^2$ for any prime $p$.
\edr

\bdr\label{DotPlus}
Let $M$ be a free module with a chosen basis $B$, and let
$\tau:M\to M$ be an involution preserving $B$. If
$u=\sum_{b\in B}\lambda_b b$,
we write
\[
u\dotPlus\tau(u)
\]
for the element obtained from $u+\tau(u)$ by counting each
$\tau$-fixed basis element only once.
\edr

For nonnegative integers $a$ and $b$, set
\[
\pi(a,b):=
\bcs
0&\text{if\; $a$ and $b$ are odd},\\[1mm]
\displaystyle\binom{\lfloor a/2\rfloor+\lfloor b/2\rfloor}{\lfloor a/2\rfloor}&\text{otherwise}.
\ecs
\]

\bdr\label{Ez}
Denote by
\[
\mathrm E[z]:=\plusn_{i\geqslant0}\mathbb Z\cdot z^{(i)},
\qquad w(z^{(i)})=i
\]
the graded $\mathbb{Z}$--algebra with multiplication defined by
$z^{(i)}\cdot z^{(j)}:=\pi(i,j)z^{(i+j)}$.

In addition, define its graded subalgebra $\G[z]:=\bigoplus_{i\geqslant 0}\mathbb{Z}\cdot z^{(2i)}$.
\edr

The identification $z^{[i]}\mapsto z^{(2i)}$
identifies $\G[z]$ with the divided power algebra
on a generator of weight $2$, where $z^{[i]}$
denotes the $i$th divided power.
 
\section{The Fuchs complex}\label{FuchsC}

\bdr
\textit{The braid group $\mathbf B_n$ on $n\geqslant1$ strands}
is the Artin group $\mathbb A(A_{n-1})$, where $A_0$ denotes the
trivial Coxeter group.
\edr

\bdr
Let $\mathrm{A}:=\plusn_{n\geqslant 0}\mathrm{A}(n)$ be the direct sum
of complexes, where $\mathrm{A}(n)=\mathbf{S}(A_{n-1})$ for $n>0$,
and $\mathrm{A}(0)$ is the complex $\mathbb{Z}$ concentrated in degree~$0$.
We denote the differential of the complex $\mathrm{A}$ by $\partial$.
We extend the notation by setting $\mathrm A(n)=0$ for $n<0$.
\edr

Any subgroup $\langle s_{i_1},\dots,s_{i_m}\rangle\subseteq A_{n-1}$,
where $1\leqslant i_1<\dots<i_m\leqslant n-1$,
defines the subset $\{i_1,\dots,i_m\}\subset\llbracket n-1\rrbracket$, and conversely.
In Lemma~\ref{CompB}(a), we will show that the subsets $P\subset\llbracket n-1\rrbracket$
are in canonical bijection with compositions of weight~$n$.
This identification provides a convenient description of the complex $\mathrm{A}(n)$ in terms of compositions.
\bdr\label{comp}
A \textit{composition} is a~sequence $\alpha=(n_1,\dots,n_l)$
of positive integers, which are called
the \textit{parts} of~$\alpha$,
the number $l(\alpha):=l$ is its \textit{length},
$w(\alpha):=n_1+\dots+n_l$ is its \textit{weight},
and $q(\alpha):=w(\alpha)-l(\alpha)$ is its \emph{degree}.
\edr

The set of compositions of weight $n$ and length $l$ is denoted by $R_l(n)$.
Define $R(n):=\bigcup_{1\leqslant l\leqslant n}R_l(n)$.
Set
$1_l:=(1,\dots,1)$, where 1 is repeated $l>0$ times.
\smallskip

\blr\label{CompB}
{\rm(a)} For $n\geqslant 1$, there is a~canonical bijection
\[
\varepsilon:\big\{P\subset\llbracket n-1\rrbracket\big\}\longrightarrow R(n)
\quad\text{with}\quad\varepsilon(P)\in R_{n-|P|}(n).
\]

{\rm(b)}
Let $n\geqslant 2,P\subset\llbracket n-1\rrbracket$, and let
$s\in\llbracket n-1\rrbracket\setminus P$.
If $\varepsilon(P)=(n_1,\dots,n_a)$, then
there exists a unique $i\in\{1,\dots,a-1\}$ with
\[
\varepsilon\bigl(P\cup\{s\}\bigr)=(n_1,\dots,n_{i-1},n_i+n_{i+1},n_{i+2},\dots,n_a).
\]
\elr

\bp
(a) Let $\llbracket n-1\rrbracket\setminus P=\{q_1<\dots<q_r\}$.
Set $q_0=0$ and $q_{r+1}=n$. Define the map $\varepsilon$ by
\[
\varepsilon(P):=(q_1-q_0,q_2-q_1,\dots,q_{r+1}-q_r)\in R_{n-|P|}(n).
\]
The inverse map $\varepsilon^{-1}$ on
a composition $\alpha=(n_1,\dots,n_a)$ of weight $n$ is given by
\[
\varepsilon^{-1}(\alpha)=P_\alpha:=\llbracket n-1\rrbracket\setminus\{t_1,\dots,t_{a-1}\},\qquad
t_i:=n_1+\dots+n_i\quad(i\in\llbracket a-1\rrbracket).
\]
The identities $\varepsilon(P_\alpha)=\alpha$ and $\varepsilon^{-1}(\varepsilon(P))=P$
show that $\varepsilon$ is a bijection.
\smallskip

(b) Since
$\llbracket n-1\rrbracket\setminus P=\{t_1,\dots,t_{a-1}\}$,
it follows that $s=t_i$ for a unique index $i$.
Hence, claim (b) follows directly from the construction of $\varepsilon(P)$.
\ep

\blr\label{WA}
Let $(W,\llbracket n-1\rrbracket)$ be the Coxeter system of type $A_{n-1}$,
and let $P\subseteq\llbracket n-1\rrbracket$.
Assume that $\varepsilon(P)=(n_1,\dots,n_a)$.
Then, for any $s\in\llbracket n-1\rrbracket\setminus P$,
let $i(s)$ be the unique index determined by
Lemma~\ref{CompB}(b). We have
\[
\left.\frac{W_{P\cup\{s\}}(t)}{W_P(t)}\right|_{t=-1}=\pi(n_{i(s)},n_{i(s)+1}).
\]
\elr

\bp
Let $\langle P\rangle=\langle P_1\rangle\times\dots\times\langle P_m\rangle$
be the decomposition of $\langle P\rangle\subseteq A_{n-1}$
into the direct product of its irreducible parabolic subgroups.
By the construction of $\varepsilon$, the subgroups
$\langle P_r\rangle$ for $1\le r\le m$ correspond bijectively
to the parts $|P_r|+1$ of the composition $\varepsilon(P)$.
The remaining parts of $\varepsilon(P)$ are equal to $1$.

By the multiplicativity of Poincaré polynomials under direct products of Coxeter groups,
the Poincaré polynomial $W_P(t)$, where $\varepsilon(P)=(n_1,\dots,n_a)$, satisfies
\[
W_P(t)=A_{n_1-1}(t)\cdot{\dots}\cdot A_{n_a-1}(t).
\]
Then Lemma \ref{CompB}(b) and a routine calculation imply
that for $s\in\llbracket n-1\rrbracket\setminus P$ and $i=i(s)$, we have
\[
\left.\frac{W_{P\cup\{s\}}(t)}{W_P(t)}\right|_{t=-1}=
\left.\frac{A_{n_i+n_{i+1}-1}(t)}{A_{n_i-1}(t)\cdot A_{n_{i+1}-1}(t)}\right|_{t=-1}=\pi(n_i,n_{i+1}).
\qedhere
\]
\ep

Define a function
\[
\nu:R(n)\longrightarrow \mathbb{Z}
\qquad\text{by}\qquad
\nu(n_1,\dots,n_a):=\sum\nolimits_{j=1}^{a-1}(a-j)n_j.
\]
\btr\label{FuchsD}
The complex $\mathrm{A}(n)$ is canonically isomorphic to the complex with basis
\[
e(n_1,\dots,n_a):=(-1)^{\nu(n_1,\dots,n_a)}(n_1,\dots,n_a),
\]
where $(n_1,\dots,n_a)$ runs over all compositions of weight~$n$,
and with the differential given by
\[
\partial\bigl(e(n_1,\dots,n_a)\bigr)=
\sum_{i=1}^{a-1}(-1)^i\pi(n_i,n_{i+1})\cdot
e(n_1,\dots,n_i+n_{i+1},\dots,n_a).
\]
Moreover, $\mathrm{A}^q(n)$ is the linear span of the elements $\{e(n_1,\dots,n_a)\mid n_1+\dots+n_a-a=q\}$.
\etr

\bp
Since, by linearity, the map $\varepsilon$ canonically identifies $\mathrm{A}(n)$
with the $\mathbb{Z}$-module generated by the set of all compositions of weight~$n$,
the elements $e(n_1,\dots,n_a)$ form a basis of $\mathrm{A}(n)$.
Combining the preceding discussion with Theorem~\ref{Sal},
we see that the differential $\partial$ acts on a composition
$c=(n_1,\dots,n_a)$ of weight $n$ by
\[
\partial(c)=\sum\nolimits_{i=1}^{a-1}
(-1)^{\sum_{k=1}^i n_k-i}\,\pi(n_i,n_{i+1})\cdot c_i,
\]
where $c_i=(n_1,\dots,n_i+n_{i+1},\dots,n_a)$.

In $\nu(c)$, the coefficient of $n_k$ is $a-k$.
In $\nu(c_i)$, the composition length is $a-1$.
Thus, for $k \leqslant i$, the coefficient of $n_k$ shifts down
to $(a-1)-k$, while for $k > i+1$ it remains $a-k$ because
its position index also shifts down by 1.
Therefore,
\[
\nu(c)-\nu(c_i)=\sum\nolimits_{k=1}^i n_k.
\]
Substituting $c_i=(-1)^{\nu(c_i)}e(c_i)$, the total exponent of
$-1$ for the $i$-th term is
\[
\nu(c)+\nu(c_i)+\sum_{k=1}^i n_k-i.
\]
Modulo 2, this exponent is congruent to:
\[
\nu(c)-\nu(c_i)+\sum\nolimits_{k=1}^i n_k-i=2\sum\nolimits_{k=1}^i n_k-i\equiv i \pmod 2.
\]
This yields the asserted expression for
$\partial\bigl(e(c)\bigr)$.
\ep

\bdr
The complex $\mathrm{A}(n)$ described in Theorem~\ref{FuchsD}
is called the \emph{Fuchs complex} of degree~$n$.
Its basis, consisting of the elements $e(n_1,\dots,n_a)$, is called the \emph{Fuchs basis}.
\edr

\emph{In what follows, we use the Fuchs basis as the default basis of $\mathrm{A}(n)$.}
\smallskip

\btr\label{SEF}
Define a homomorphism of abelian groups
\[
\Phi:\mathrm{A}:=\plusn_{n\geqslant 0}\mathrm{A}(n)
\longrightarrow C_*\bigl(\mathrm{E}[z]\bigr)
\]
by
\[
\Phi(e(0))=z^{(0)},\qquad
\Phi\bigl(e(n_1,\ldots,n_a)\bigr)
=
\bigl[z^{(n_1)}|\cdots|z^{(n_a)}\bigr].
\]
Then, for each $n\geqslant 0$ and each $q$, the map $\Phi$ restricts to
an isomorphism of abelian groups
\[
\Phi_n^q:\mathrm{A}^q(n)\xrightarrow{\;\cong\;}
C^{(n)}_{n-q}\bigl(\mathrm{E}[z]\bigr).
\]
Moreover, this isomorphism identifies the coboundary operator 
$\partial$ of the Fuchs complex with the differential $b$
in the normalized bar complex $C_*(E[z];\mathbb Z)$:
for every $x\in \mathrm{A}^q(n)$ one has
\[
\Phi_n^{q+1}(\partial x)=b\,\Phi_n^q(x).
\]
Define $H^q(\mathbf{B}_0):=H^q(\mathrm{A}(0))$.
Then, for $n\geqslant 0$,
\[
H^q(\mathbf{B}_n)\cong
H^{(n)}_{n-q}\bigl(\mathrm{E}[z]\bigr).
\]
\etr

\begin{proof}
The map $\Phi$ is clearly a bijection on the basis elements of weight
$n$, and it sends $\mathrm{A}^q(n)$ onto
$C^{(n)}_{n-q}(\mathrm{E}[z])$. It remains only to compare the
differentials. By Theorem~\ref{FuchsD},
\[
\partial e(n_1,\ldots,n_a)
=\sum_{i=1}^{a-1}(-1)^i\pi(n_i,n_{i+1})\,
e(n_1,\ldots,n_i+n_{i+1},\ldots,n_a).
\]
On the other hand, the bar differential on $C_*(\mathrm{E}[z])$
gives
\[
b\bigl[z^{(n_1)}|\cdots|z^{(n_a)}\bigr]
=
\sum_{i=1}^{a-1}
(-1)^i
\bigl[z^{(n_1)}|\cdots|z^{(n_i)}z^{(n_{i+1})}|\cdots|z^{(n_a)}\bigr].
\]
Since
$z^{(n_i)}z^{(n_{i+1})}=\pi(n_i,n_{i+1})\,z^{(n_i+n_{i+1})}$,
we obtain $\Phi_n^{q+1}(\partial x)=b\,\Phi_n^q(x)$
for every $x\in\mathrm{A}^q(n)$. Hence $\Phi$ induces the stated
isomorphism in cohomology.
\end{proof}

\brm\label{rmFuchs}
Fuchs introduced the complex $\mathrm{A}(n)\otimes\mathbb{F}_2$ in \cite{Fuchs}
and used it to compute $H^q(\mathbf{B}_n;\mathbb{F}_2)$,
long before the Salvetti complex became known.
In contrast to the algebraic approach adopted in this section,
Fuchs defined the elements $e(m_1,\dots,m_a)$ geometrically, as a 
cellular decomposition of the one-point compactification of the space of monic polynomials
\[
t^n+a_{n-1}t^{n-1}+\dots+a_0\in\mathbb C[t]
\]
without multiple roots. The latter space, before compactification,
is well known to be an open manifold that serves as a classifying
space for the braid group $\mathbf B_n$.
Fuchs extracted the cohomology of the group from the cellular cohomology 
of this compactification by invoking Alexander duality (see~\cite{Fuchs}).
Later, Fuchs computed the action of the coboundary operator on his cellular decomposition
over the integers, though this result remained unpublished.
It served as the starting point for my note \cite{FW}.
\erm 

\section{Cohomology of the Braid Groups and Homology of the Algebra 
\texorpdfstring{$\G[z]$}{Gamma[z]}}\label{IsoA}

Theorem~\ref{SEF} reduces the description of the cohomology of braid groups
to the description of the homology of the algebra $\mathrm{E}[z]$.
This section presents several corollaries of this observation.

The formula
\[
\varphi(z^{(a)}):=
\bcs
z^{(a)}\otimes z^{(0)}&\text{\rm if\; $a$ is even},\\
z^{(a-1)}\otimes z^{(1)}&\text{\rm if\; $a$ is odd}
\ecs
\]
defines an isomorphism of graded algebras:
$\varphi:\mathrm{E}[z]\longrightarrow\G[z]\otimes\Lambda[e(1)]$.

Under the identification of Theorem~\ref{SEF}, the one-part
basis element $e(i)\in \mathrm{A}(i)$ corresponds to the one-letter
chain $[z^{(i)}]\in C^{(i)}_1(\mathrm{E}[z])$. In formulas, when no
confusion is possible, we shall write simply $e(i)$ or $z^{(i)}$ for
this corresponding generator.

Since $C_*(\Lambda[e(1)])\cong\bigoplus_{a\geqslant 0}\mathbb{Z}\cdot e(1_a)$,\footnote{By definition, $e(1_0)=1\in\mathbb{Z}$.}
the Alexander--Whitney morphism (see~\ref{AW}),
for $n\geqslant 0$, induces, by Theorem~\ref{AWH}, a quasi-isomorphism of complexes
\[
\mathrm{AW}:C^{(n)}_*\bigl(\mathrm{E}[z]\bigr)\to
\plusn_{0\leqslant a\leqslant n}
C_{*-a}^{(n-a)}\bigl(\G[z]\bigr)\otimes e(1_a).
\]
Since $C^{(m)}_*\bigl(\G[z]\bigr)=0$ for odd $m$, setting $n-a=2r$ 
allows us to rewrite this for any $0\leqslant q\leqslant n$ in the form:
\[
\mathrm{AW}:C^{(n)}_q\bigl(\mathrm{E}[z]\bigr)\to
\plusn_{\left\lceil\frac{n-q}{2}\right\rceil\leqslant r\leqslant\left\lfloor\frac{n}{2}\right\rfloor} 
C_{q-n+2r}^{(2r)}\bigl(\G[z]\bigr)\otimes e(1_{n-2r}).
\]
Combining this quasi-isomorphism with Theorem~\ref{SEF}, whose
degree $q$ component identifies $\mathrm A^q(n)$ with
$C^{(n)}_{n-q}(\mathrm E[z])$, we obtain that the corresponding
component in cohomological degree $q$ is:
\beq\label{Qasi}
\mathrm{AW}:\mathrm{A}^q(n)\to
\plusn_{\left\lceil\frac{q}{2}\right\rceil\leqslant r\leqslant\left\lfloor\frac{n}{2}\right\rfloor}
C_{2r-q}^{(2r)}\bigl(\G[z]\bigr)\otimes e(1_{n-2r}).
\eeq
Consequently, we obtain the canonical isomorphism
\beq\label{HAn}
H^q(\mathbf{B}_n)\cong
\plusn_{\left\lceil\frac{q}{2}\right\rceil\leqslant r\leqslant\left\lfloor\frac{n}{2}\right\rfloor}
H_{2r-q}^{(2r)}\bigl(\G[z]\bigr)\otimes e(1_{n-2r}).
\eeq
In what follows, we identify the rank-one factor
$\mathbb Z\cdot e(1_{n-2r})$ with $\mathbb Z$ by means of its
canonical generator $e(1_{n-2r})$.

\btr\label{HGIso}
There is a canonical isomorphism
\beq\label{HAn1}
H^q(\mathbf{B}_n)\cong
\bcs
\plusn_{\left\lceil\frac{q}{2}\right\rceil\leqslant r\leqslant\min\left\{q-1,\left\lfloor\frac{n}{2}\right\rfloor\right\}}
H_{2r-q}^{(2r)}\bigl(\G[z]\bigr)&\text{if $q\geqslant 2$ and $n\geqslant 2$},\\
\mathbb{Z}&\text{if $q\in\{0,1\}$ and $n\geqslant 2$},\\
\mathbb{Z}&\text{if $q=0$ and $n\in\{0,1\}$},\\
0&\text{otherwise}.
\ecs
\eeq
\etr

\bp
The cases $n=0,1$ follow directly from the definitions.
For $n\geqslant2$, the cases $q=0,1$ follow from
\eqref{HAn}. Suppose now that $n\geqslant2$ and $q\geqslant2$.

It is clear that $C_{2r-q}^{(2r)}\bigl(\G[z]\bigr)={0}$ for $r>q$.
If $r=q$, then
$C_q^{(2q)}\bigl(\G[z]\bigr)=\mathbb{Z}\cdot\big[z^{(2)}|\dots|z^{(2)}\big]$,
where $z^{(2)}$ is repeated $q$ times.
Since $b\big[z^{(2)}|\dots|z^{(2)}\big]\neq 0$,
it follows that $H_q^{(2q)}\bigl(\G[z]\bigr)\cong 0$.
Hence, $r\leqslant q-1$, and the required statement follows from
isomorphism \eqref{HAn}.
\ep

The following corollary subsumes \emph{Arnold's repetition and stabilization theorems}~\cite{Arnold}.

\bcr\label{ArnRS}
There are canonical isomorphisms:
\begin{itemize}
    \item[$\mathrm{(a)}$] $H^q(\mathbf{B}_n)\cong H^q(\mathbf{B}_{n+1})$ if $n$ is even.
    \item[$\mathrm{(b)}$] If $q\geqslant 2$ and $n\geqslant 2q-2$, then $H^q(\mathbf{B}_n)\cong H^q(\mathbf{B}_{2q-2})$.
    \item[$\mathrm{(c)}$] $H^2(\mathbf{B}_n)=0$.
\end{itemize}
\ecr

\bp
Claims $\mathrm{(a)}$ and $\mathrm{(b)}$ follow directly from the isomorphism \eqref{HAn1}.
Claim $\mathrm{(c)}$ also follows, since for $q=2$,
the only possible summand on the right-hand side of the isomorphism \eqref{HAn1}
is $H^{(2)}_0(\G[z])$, which is zero since $C^{(2)}_0(\G[z])=0$.
\ep

\bcr[Arnold's finiteness theorem, \cite{Arnold}]\label{Arn1}
For $n\geqslant 2$, there are isomorphisms:
\[
H^0(\mathbf{B}_n;\mathbb{Q})\cong H^1(\mathbf{B}_n;\mathbb{Q})\cong\mathbb{Q},
\qquad H^q(\mathbf{B}_n;\mathbb{Q})\cong 0\quad\text{if\; $q\geqslant 2$}.
\]
In particular, for $q\geqslant 2$, the groups $H^q\bigl(\mathbf{B}_n\bigr)$ are finite.
\ecr

\bp
The map
$\G[z]\otimes\mathbb{Q}\to\mathbb{Q}[y]$, $\,z^{(2r)}\to\frac{y^{r}}{r!}$, where $w(y)=2$,
is a graded $\mathbb{Q}$-algebra isomorphism.
Since $H_*(\mathbb{Q}[y];\mathbb{Q})\cong\mathbb{Q}\plus\mathbb{Q}\cdot y$, we have:
\[
H_q^{(2r)}(\G[z];\mathbb{Q})\cong H_q^{(2r)}(\mathbb{Q}[y];\mathbb{Q})\cong
\bcs
\mathbb{Q}&\text{if $(r,q)=(0,0),(1,1)$},\\
0&\text{otherwise}.
\ecs
\]
Applying the functor $-\otimes_{\mathbb{Z}}\mathbb{Q}$ to
isomorphism~\eqref{HAn1}, we obtain the desired result.
\ep

For $n\geqslant 0$, the isomorphism \eqref{HAn1} yields an obvious canonical embedding
\[
i_n^*:H^q(\mathbf{B}_n)\hookrightarrow H^q(\mathbf{B}_{n+1}).
\]
In what follows, we will identify $H^q(\mathbf B_n)$ with its image in
$H^q(\mathbf B_{n+1})$ under this embedding. In particular, \emph{any quotient of the form
$H^q(\mathbf B_{n+1})\big/H^q(\mathbf B_n)$ is strictly understood with respect to $i_n^*$}.

\bdr
For $n\geqslant0$, define the \emph{stabilization operator}
\[
\mathrm{St}:\mathrm A(n)\longrightarrow \mathrm A(n-1)
\]
as follows. For $n=0$, this is the zero map
$\mathrm A(0)\to \mathrm A(-1)=0$. For $n\geqslant1$, put
\[
\mathrm{St}\bigl(e(n_1,\ldots,n_a)\bigr)=
\begin{cases}
e(n_1,\ldots,n_{a-1}), & \text{if } n_a=1,\\
0, & \text{if } n_a>1.
\end{cases}
\]

Here, for $a=1$ and $n_1=1$, the expression
$e(n_1,\ldots,n_{a-1})$ means the empty composition
$e(0)=1\in \mathrm A(0)$.
\edr

The map $\mathrm{St}$ is a morphism of complexes, as follows directly
from the definitions.

\btr\label{SeqS}
For $n\geqslant 0$ and $q\geqslant 0$, there is a split short exact sequence
\[
0\longrightarrow
H^{(n)}_{n-q}\bigl(\G[z]\bigr)
\longrightarrow H^q\bigl(\mathrm{A}(n)\bigr)
\stackrel{\mathrm{St}^*}\llongrightarrow H^q\bigl(\mathrm{A}(n-1)\bigr)
\longrightarrow 0.
\]
Thus, there exists a canonical isomorphism
\[
H^q\bigl(\mathrm{A}(n)\bigr)\cong
H^{(n)}_{n-q}\bigl(\G[z]\bigr)\plusn
H^q\bigl(\mathrm{A}(n-1)\bigr).
\]
\etr

\bp
For $n=0,1$, the assertion follows immediately from the
definition of $\mathrm{St}$ in these cases.

Let $n\geqslant 2$.
For $0\leqslant 2r\leqslant n$, define
\[
\Psi_{n,r}:
H_{2r-q}^{(2r)}(\G[z])
\longrightarrow H^q(\mathrm A(n)),
\qquad
[c]\longmapsto[c*e(1_{n-2r})].
\]
Here the shuffle product is followed by the multiplication
isomorphism
$\G[z]\otimes\Lambda[e(1)]\cong\mathrm E[z]$
and the identification of Theorem~\ref{SEF}.
By the Eilenberg--Zilber theorem, these maps give an isomorphism
\[
\plusn_{0\leqslant 2r\leqslant n}
H_{2r-q}^{(2r)}(\G[z])
\stackrel{\cong}{\longrightarrow}
H^q(\mathrm A(n)),
\]
where groups in negative homological degrees are understood
to be zero.

Every letter of a chain $c\in C_*(\G[z])$ has positive
even weight. Hence, for $a\geqslant1$, the shuffle terms
in $c*e(1_a)$ ending in $e(1)$ correspond, with the same
signs, to the terms in $c*e(1_{a-1})$. Consequently,
\[
\mathrm{St}\bigl(c*e(1_a)\bigr)=
\begin{cases}
c*e(1_{a-1}),&a\geqslant1,\\
0,&a=0.
\end{cases}
\]
Thus,
\[
\mathrm{St}^*\Psi_{n,r}=
\begin{cases}
\Psi_{n-1,r},&2r<n,\\
0,&2r=n.
\end{cases}
\]
It follows that $\mathrm{St}^*$ is surjective and its kernel is
the summand $H_{n-q}^{(n)}(\G[z])$.

The canonical embedding
\[
i_{n-1}^*:H^q(\mathrm A(n-1))\hookrightarrow H^q(\mathrm A(n))
\]
introduced earlier maps the $r$-th summand of $H^q(\mathrm A(n-1))$
identically to the $r$-th summand of $H^q(\mathrm A(n))$. In terms of our basis, this means:
\[
i_{n-1}^*\Psi_{n-1,r}([c])=\Psi_{n,r}([c]).
\]
Because $\mathrm{St}^*\Psi_{n,r}=\Psi_{n-1,r}$, it immediately
follows that $\mathrm{St}^*\circ i_{n-1}^*=\mathrm{id}$.
Thus, $i_{n-1}^*$ is a canonical section for $\mathrm{St}^*$,
which proves that the short exact sequence splits.
\ep

Now, we can strengthen Arnold's stabilization theorem (Corollary \ref{ArnRS}, $\mathrm{(b)}$) as follows:
\smallskip

\btr\label{StabSH}
For $n\geqslant 0$ the canonical map
\[
\mathrm{St}^*:H^q\bigl(\mathrm{A}(n)\bigr)\longrightarrow H^q\bigl(\mathrm{A}(n-1)\bigr)
\]
is a surjection, and is an isomorphism if $n\geqslant\max\{3,2q-1\}$.
In particular, the cohomology of the braid family $\mathbf B$ stabilizes.
\etr

\bp
Surjectivity follows from Theorem~\ref{SeqS}.
Suppose that $n\geqslant\max\{3,2q-1\}$.
By the same theorem, it suffices to show that
$H_{n-q}^{(n)}(\G[z])=0$.
For odd $n$, this follows from the weight grading.
For even $n$, put $r=n/2$. Then $r\geqslant q$ and $r\geqslant2$.
If $r>q$, the corresponding chain group is zero.
If $r=q$, then $q\geqslant2$, and the vanishing follows
from the proof of Theorem~\ref{HGIso}.
\ep
\smallskip

\bcr\label{A2}
Define
\[
\mathrm A_2(n):=
\begin{cases}
\mathrm A(0),&\text{if }n=0,\\
0,&\text{if }n=1,\\
\Ker\bigl(\mathrm{St}:\mathrm A(n)\to\mathrm A(n-1)\bigr),
&\text{if }n\geqslant2.
\end{cases}
\]
For all $n,q\geqslant0$, there is a canonical isomorphism
\[
H^q(\mathrm A_2(n))
\cong H_{n-q}^{(n)}(\G[z]).
\]
For $n\geqslant2$, this gives
\[
H^q(\mathrm A_2(n))
\cong\Ker(\mathrm{St}^*)
\cong
H^q(\mathbf B_n)\big/H^q(\mathbf B_{n-1}).
\]
In particular, for $0\leqslant n\leqslant2$,
\[
H^q(\mathrm A_2(n))\cong
\begin{cases}
\mathbb Z,&(n,q)=(0,0)\text{ or }(n,q)=(2,1),\\
0,&\text{otherwise}.
\end{cases}
\]
\ecr

\bp
For $n\geqslant2$, the identifications follow from the
long exact cohomology sequence associated with
\[
0\longrightarrow\mathrm A_2(n)
\longrightarrow\mathrm A(n)
\xrightarrow{\mathrm{St}}\mathrm A(n-1)
\longrightarrow0
\]
and the surjectivity in Theorem~\ref{SeqS}.
The cases $n=0,1$ follow directly from the definitions.
\ep

\section{(Co)homology of the Algebra $\G[z]$}\label{Bpm}

Theorem~\ref{HGIso} shows that, to describe the additive structure
of the groups $H^*(\mathbf{B}_n)$ for all $n\geqslant 1$,
it suffices to determine the homology of the algebra $\G[z]$.
For greater generality, and with a slight abuse of notation,
in this section we regard $\G[z]$ simply as the standard divided power algebra 
generated by an element $z$, equipped with an arbitrary positive integer weight $w(z)$ 
such that $w(z^{(k)}) = k \cdot w(z)$. This convention is used only in the present section.

We begin by considering
$H_*\bigl(\G[z];\mathbb{F}_p\bigr)\cong H_*\bigl(\G_p[z];\mathbb{F}_p\bigr)$ for all primes $p$,
where $\G_p[z]:=\G[z]\otimes\mathbb{F}_p$.
It is well known (see \cite[Ex.~3C.5]{Hatcher}) that
for any prime $p$, there is an isomorphism of graded $\mathbb{F}_p$-algebras
\[
\G_p[z]\;\cong\;\tensorn_{r\geqslant 0}\;\mathbb{F}_p[z_{p^r}]\big/(z^p_{p^r}),
\qquad\text{where}\quad w(z_{p^r})=w(z)p^r,
\]
where $z_{p^r}$ denotes the image of $z^{(p^r)}$ in $\Gamma_p[z]$.
By Theorem~\ref{AWH}, we obtain a
canonical isomorphism of $\mathbb{F}_p$-vector spaces:
\beq\label{lm10}
H_*\bigl(\G[z];\mathbb{F}_p\bigr)\;\cong\;
\tensorn_{r\geqslant 0}\, H_*\bigl(\mathbb{F}_p[z_{p^r}]\big/(z^p_{p^r});\mathbb{F}_p\bigr).
\eeq

According to Corollary~\ref{ttth2}, the space $H_*\bigl(\mathbb{F}_p[t]\big/(t^p);\mathbb{F}_p\bigr)$
admits the structure of a~graded-commutative algebra with respect to the Pontryagin product (Def.~\ref{dPon}).
Moreover, the algebra $\mathbb{F}_p[t]/(t^p)$ has a natural structure
of a commutative Hopf algebra, where the generator $t$ is primitive.
By Corollary~\ref{HopfHomology}, its homology automatically inherits the structure of a graded Hopf algebra.
\smallskip

\blr\label{Zp}
For any prime $p$, there is an isomorphism of graded algebras
\[
H_*\bigl(\mathbb{F}_p[t]\big/(t^p);\mathbb{F}_p\bigr)\cong\mathrm{E}_p[t],\quad
\text{where}\quad\mathrm{E}_p[t]:=\mathrm{E}[t]\otimes\mathbb{F}_p,
\]
and
$H_q\bigl(\mathbb{F}_p[t]\big/(t^p);\mathbb{F}_p\bigr)$ is generated by the cycle
$t^{(q)}:=\Pi[t,t^{p-1}:q]$, the alternating product $\bigl[t\mid t^{p-1}\mid t\mid\dots\bigr]$
of $t$ and $t^{p-1}$ of total length $q>0$.

The coproduct $\Delta$ of the Hopf algebra $H_*\bigl(\mathbb{F}_p[t]\big/(t^p);\mathbb{F}_p\bigr)$
is expressed by
\[
\Delta(t^{(q)})=\sum\nolimits_{i=0}^q t^{(i)}\otimes t^{(q-i)}\qquad\text{if $p=2$}.
\]
\[
\bcs
\Delta(t^{(2m)})&=\sum\nolimits_{i=0}^m t^{(2i)}\otimes t^{(2m-2i)},\\[2mm]
\Delta(t^{(2m+1)})&=\sum\nolimits_{i=0}^m\left(t^{(2i)}\otimes t^{(2m+1-2i)}+t^{(2i+1)}\otimes t^{(2m-2i)}\right)
\ecs
\qquad\text{if $p>2$}.
\]
In these formulas, $t^{(r)}$ denotes the homology
class of the corresponding cycle for every $r\geqslant0$,
with $t^{(0)}=1$ represented by the empty chain.
\elr

\bp
A basis of $C_*\bigl(\mathbb{F}_p[t]\big/(t^p);\mathbb{F}_p\bigr)$ consists of chains
$\bigl[t^{i_1}|\dots|t^{i_q}\bigr]$, where $1\leqslant i_j\leqslant p-1$ for all $j$,
and its differential is given by
\[
d\bigl[t^{i_1}|\dots|t^{i_q}\bigr]=
\sum\nolimits_{a=1}^{q-1}(-1)^a
\bigl[t^{i_1}|\dots|t^{i_a+i_{a+1}}|\dots|t^{i_q}\bigr].
\]
Here a term with $i_a+i_{a+1}\geqslant p$ is understood to be zero.

The isomorphism $H_q\bigl(\mathbb{F}_p[t]\big/(t^p);\mathbb{F}_p\bigr)\cong\mathbb{F}_p\cdot t^{(q)}$
is evident for $p=2$, or when $q=0,1$.
To prove it for $p>2$ and $q\geqslant 2$, we proceed by induction on $q$.

Define an operator
$h:C_q\bigl(\mathbb{F}_p[t]\big/(t^p);\mathbb{F}_p\bigr)\to
C_{q+1}\bigl(\mathbb{F}_p[t]\big/(t^p);\mathbb{F}_p\bigr)$
by setting $h=0$ on $C_0$, and
\[
h\bigl[t^{i_1}|t^{i_2}|\dots|t^{i_q}\bigr]:=
\bcs
0&\text{if\; $i_1=1$},\\
\bigl[t|t^{i_1-1}|t^{i_2}|\dots|t^{i_q}\bigl]&\text{if\; $i_1>1$}.
\ecs
\]
A straightforward calculation gives the following identities:
\begin{align*}
(hd+dh)\big[t^{i_1}|\dots|t^{i_q}\big]&=-\big[t^{i_1}|\dots|t^{i_q}\big]
+(\text{terms starting with } t) \quad \text{for } i_1 > 1,\\
(hd+dh)\big[t|t^{i_2}|\dots|t^{i_q}\big] &= 
\bcs
-\big[t|t^{i_2}|\dots|t^{i_q}\big] & \text{if } i_2 \neq p-1,\\
0 & \text{if } i_2=p-1.
\ecs
\end{align*}
Set $P:=\mathrm{id}+hd+dh$. A direct calculation shows that, for $q\geqslant2$,
\[
P[t^{i_1}|\dots|t^{i_q}]
=
\begin{cases}
\big[t|t^{p-1}|t^{i_3}|\dots|t^{i_q}\big],&i_1+i_2=p,\\
0,&i_1+i_2\ne p.
\end{cases}
\]
Since $P-\mathrm{id}=hd+dh$, the map $P$ is
chain-homotopic to the identity. Its image consists
of $\mathbb F_p$ in degree $0$, $\mathbb F_p\cdot[t]$ in degree $1$, and the chains
$\big[t|t^{p-1}|v\big]$ in degrees at least $2$.
For the latter chains,
\[
d\big[t|t^{p-1}|v\big]=\big[t|t^{p-1}|dv\big],
\]
since both $t\cdot t^{p-1}$ and the product of
$t^{p-1}$ with the first letter of $v$ vanish.
Thus, the part of $\operatorname{Im}P$ in degrees
at least $2$ is a copy of the original complex
shifted by two degrees.
Consequently, $H_q\cong H_{q-2}$ for $q\geqslant2$.

Together with the direct computation in degrees $0$ and $1$, this
proves that $H_q$ is one-dimensional, generated by the alternating
cycle $t^{(q)}$.
The Pontryagin product of the homology classes
of the cycles $t^{(q_1)}$ and $t^{(q_2)}$
corresponds to the homology class of the cycle $t^{(q_1)}*t^{(q_2)}$,
where $*$ is the shuffle product (see \ref{EZ} and \ref{ttth2}).
The equality
\[
t^{(2a)}*t^{(2b)}=\pi(2a,2b)t^{(2a+2b)}
\]
is proved by a straightforward induction on $a$.
Moreover, $t^{(2a+1)}*t^{(2b+1)}=0$.
Indeed, this product contains the factor $t^{(1)}*t^{(1)}$, and
$t^{(1)}*t^{(1)}=[t]*[t]=[t|t]-[t|t]=0$ for every prime $p$.
Finally, by associativity,
\[
t^{(2a+1)}*t^{(2b)}=
t^{(2b)}*t^{(2a+1)}=
\pi(2a+1,2b)t^{(2a+2b+1)}.
\]
This follows since $t^{(2a+1)}=t^{(2a)}*t^{(1)}$.

Thus the Pontryagin product is given by
$t^{(i)}*t^{(j)}=\pi(i,j)t^{(i+j)}$.

To compute the coproduct $\Delta(t^{(q)})$, we apply Corollary~\ref{HopfHomology}.
For any $c\in A_+$, we have $\Delta_A(c)=c\otimes 1+1 \otimes c+\sum c^\prime\otimes c^{\prime\prime}$,
where $c^\prime,c^{\prime\prime}\in A_+$. 
In the definition of $\mathrm{AW}$ (Definition~\ref{AW}),
the augmentation $\varepsilon$ annihilates $A_+$.
Thus, when evaluating $\mathrm{AW}\circ C_*(\Delta_A)$ on the chain $\big[c_1|\dots|c_q\big]$,
the only terms from $\Delta_A(c_i)$ that survive are $c_i\otimes 1$
for positions $i\leqslant r$ and $1 \otimes c_i$ for positions $i > r$. 
Therefore, the induced coproduct on the normalized bar
complex is given by the sum over all splittings:
\[
\Delta\big[c_1|\dots|c_q\big]
=\sum_{r=0}^q
\big[c_1|\dots|c_r\big]\otimes
\big[c_{r+1}|\dots|c_q\big].
\]
We call the left and right factors of each summand
the prefix and suffix of the splitting, respectively.

Applying this to $t^{(q)}$, we split the alternating sequence of $t$ and $t^{p-1}$.
For $p=2$, we obtain the asserted formula. Assume $p>2$.

If the split position $r=2i$ is even, the prefix is exactly $t^{(2i)}$
and the suffix is exactly $t^{(q-2i)}$, yielding the term $t^{(2i)}\otimes t^{(q-2i)}$.

If $r=2i+1$ is odd, the prefix is $t^{(2i+1)}$,
but the suffix is $S=\big[t^{p-1}|t|\dots\big]$ of length $q-2i-1$.

If the length of $S$ is odd, say $2k+1$, its weight is $(kp+p-1)w(t)$,
which is strictly greater than the weight $(kp+1)w(t)$ of the standard generator $t^{(2k+1)}$.
Since $H_*$ is concentrated only in the weights of $t^{(m)}$,
$S$ must be a boundary. Thus, for $q=2m$, all odd splits vanish in homology.
This proves the asserted formula for $\Delta(t^{(2m)})$.

It remains to compute $\Delta(t^{(2m+1)})$.
Set $S_k=\big[t^{p-1}|t|\cdots|t^{p-1}|t\big]$
with $k$ pairs, and put $S_0=1$.
For $k\geqslant1$, direct calculation gives
\[
d(h(S_k))
=-S_k+\big[t|t^{p-1}|S_{k-1}\big].
\]
Prepending $[t|t^{p-1}]$ commutes with the
differential. Induction on $k$ therefore gives
$S_k\sim t^{(2k)}$.
Hence each odd split $r=2i+1$ contributes
$t^{(2i+1)}\otimes t^{(2m-2i)}$,
which proves the asserted formula.
\ep

Isomorphism \eqref{lm10} and Lemma \ref{Zp} imply the following result:
\smallskip

\btr\label{th1}
For any prime $p$, there is a~canonical isomorphism of algebras
\[
H_*\bigl(\G[z];\mathbb{F}_p\bigr)\;\cong\;\mytensor_{r\geqslant 0}\;\mathrm{E}_p[z_{p^r}],
\]
compatible with the bigraded Hopf algebra structures.
The bigrading of the right-hand side is defined as follows:
\begin{align*}
\deg\big(z_{p^{r_1}}^{(q_1)}\otimes{\dots}\otimes z_{p^{r_a}}^{(q_a)}\big)&:=q_1+\dots+q_a,\\[1mm]
 w\big(z_{p^{r_1}}^{(q_1)}\otimes{\dots}\otimes z_{p^{r_a}}^{(q_a)}\big)&:=
 w\big(z_{p^{r_1}}^{(q_1)}\big)+\dots+ w\big(z_{p^{r_a}}^{(q_a)}\big),
\end{align*}
with
\[
w\bigl(z_{p^r}^{(q)}\bigr):=
\bcs
\frac{q}{2}\, p^{r+1} w(z)&\text{if}\quad q\equiv 0\mod 2,\\[2mm]
\frac{q-1}{2}\, p^{r+1} w(z)+p^r w(z)&\text{if}\quad q\equiv 1\mod 2.
\ecs
\]
Under the above isomorphism, the monomials
$x=z_{p^{r_1}}^{(q_1)}\otimes{\dots}\otimes z_{p^{r_a}}^{(q_a)}$
with $\deg(x)=m$ and $ w(x)=n$ form a basis of the
$\mathbb F_p$-vector space
$H^{(n)}_m\bigl(\G[z];\mathbb F_p\bigr)$.
\etr
\smallskip

All duals below are taken separately in each weight.
The cochain complexes and their cohomology are then
assembled as direct sums of their weight components.

The cohomology algebra $H^*\bigl(\G[z];\mathbb{F}_p\bigr)$ can be expressed more transparently,
since for the Hopf algebra $\mathrm{E}^*_p[t]$, dual to $\mathrm{E}_p[t]$, there is an isomorphism
\[
\mathrm{E}^*_p[t]\cong
\bcs
\Lambda_p[x]\otimes\mathbb{F}_p[y]&\text{if\;
$p>2$},\\[1mm]
\mathbb{F}_2[x]&\text{if\; $p=2$},
\ecs
\]
where $x,y\in\mathrm{E}^*_p[t]$
are dual to $t^{(1)},t^{(2)}\in\mathrm{E}_p[t]$, respectively (\cite[Ex.~3C.11]{Hatcher}).

Let $x_r,y_r\in\mathrm{E}^*_p[z_{p^r}]$ denote the elements dual to
$z^{(1)}_{p^r},z^{(2)}_{p^r}\in\mathrm{E}_p[z_{p^r}]$, respectively.
Then Theorem~\ref{th1} implies the following:
\smallskip

\btr\label{corHP}
For every prime $p$, there is an isomorphism of bigraded primitively generated Hopf algebras
\[
H^*\bigl(\G[z];\mathbb{F}_p\bigr)\cong
\bcs
\Lambda_p[x_0,x_1,x_2,\dots]\otimes\mathbb{F}_p[y_0,y_1,y_2,\dots]
&\text{if\; $p>2$},\\[1mm]
\mathbb{F}_2[x_0,x_1,x_2,\dots]&\text{if\;
$p=2$},
\ecs
\]
where $\deg(x_r)=1,\,\deg(y_r)=2,\, w(x_r)=w(z)p^r,\, w(y_r)=w(z)p^{r+1}$,
and the bigrading on cohomology
is dual to the bigrading on homology from Theorem~\ref{th1}.
\etr
\smallskip

The integral cohomology of the algebra $\G[z]$ is given by the following result.
\smallskip

\btr\label{CHGZ}
There is an isomorphism of groups
\[
H^q\bigl(\G[z]\bigr)\cong
\begin{cases}
\mathbb Z, & q=0,1,\\[2mm]
\displaystyle\plusn_p
\beta_p\bigl(H^{q-1}(\G[z];\mathbb F_p)\bigr),
& q\geqslant 2,
\end{cases}
\]
where the sum runs over all primes $p$, and
$\beta_p:H^{q-1}\bigl(\G[z];\mathbb F_p\bigr)
\longrightarrow H^q\bigl(\G[z]\bigr)$
is the Bockstein homomorphism associated with the exact sequence of
coefficients
$0\longrightarrow \mathbb Z\xrightarrow{\times p}
\mathbb Z\longrightarrow\mathbb F_p\longrightarrow 0$.

Let
$b_p^*:H^{*-1}\bigl(\G[z];\mathbb F_p\bigr)
\longrightarrow H^*\bigl(\G[z];\mathbb F_p\bigr)$
be the mod-$p$ Bockstein differential dual to the homological
Bockstein associated with the exact sequence
$0\longrightarrow \mathbb F_p\longrightarrow \mathbb Z_{p^2}
\longrightarrow \mathbb F_p\longrightarrow 0$.
Then $b_p^*$ is the graded derivation given by
\beq\label{Bockp}
\begin{split}
b^*_p(x_0)=0,&\qquad b^*_p(x_r)=y_{r-1}\quad (r\geqslant 1),
\qquad b^*_p(y_r)=0\quad (r\geqslant 0),\quad\text{if $p>2$,}{}\\
&b^*_2(x_0)=0,\qquad b^*_2(x_r)=x_{r-1}^2\quad (r>0).
\end{split}
\eeq
Moreover, every torsion element of $H^*\bigl(\G[z]\bigr)$ has squarefree order.
\etr

\bp
Let $b_p:H_*\bigl(\G[z];\mathbb{F}_p\bigr)\to H_{*-1}\bigl(\G[z];\mathbb{F}_p\bigr)$
denote the Bockstein homomorphism associated with the exact sequence of coefficients
$0\to\mathbb{F}_p\to\mathbb{Z}_{p^2}\to\mathbb{F}_p\to 0$.

Put $N=p^r$. The cycle $z_{p^r}^{(2)}$ admits
the integral lift $[z^{(N)}|(z^{(N)})^{p-1}]$.
With the convention $d[a|b]=-[ab]$, we have
\[
d[z^{(N)}|(z^{(N)})^{p-1}]
=
-\frac{(pN)!}{(N!)^p}[z^{(pN)}].
\]
Since
\[
\frac{(p^{r+1})!}{p(p^r!)^p}\equiv-1\pmod p,
\]
the definition of the Bockstein gives
\[
b_p\bigl(z_{p^r}^{(0)}\bigr)=0,\qquad
b_p\bigl(z_{p^r}^{(1)}\bigr)=0,\qquad
b_p\bigl(z_{p^r}^{(2)}\bigr)
=z_{p^{r+1}}^{(1)}.
\]
Consequently,
\[
\left\langle b_p^*(x_r),z_{p^k}^{(2)}\right\rangle
=
\left\langle x_r,z_{p^{k+1}}^{(1)}\right\rangle.
\]
For $p>2$, weight considerations exclude products
of two degree-one generators from $b_p^*(x_r)$.
It follows that
\[
b_p^*(x_0)=0,\qquad
b_p^*(x_r)=y_{r-1}\quad(r\geqslant1).
\]
Since $(b_p^*)^2=0$, we also obtain
$b_p^*(y_r)=0$ for every $r\geqslant0$.

For $p=2$, the coproduct formula shows that
the class dual to $z_{2^k}^{(2)}$ is $x_k^2$.
Together with weight considerations, this gives
\[
b_2^*(x_0)=0,\qquad
b_2^*(x_r)=x_{r-1}^2\quad(r\geqslant1).
\]
Thus formulas~\eqref{Bockp} hold.

Recall that the Bockstein homomorphism $b_p^*$ acts as a graded derivation on the cohomology algebra.
Thus, these formulas determine its action on $H^*\bigl(\G[z];\mathbb{F}_p\bigr)$
by virtue of Theorem~\ref{corHP}.
Let $BH^*\bigl(\G[z];\mathbb{F}_p\bigr)$ denote the
cohomology of the complex $H^*\bigl(\G[z];\mathbb{F}_p\bigr)$
with differential~$b_p^*$.
The formulas \eqref{Bockp} easily imply that for each $p$,
\[
BH^*\bigl(\G[z];\mathbb{F}_p\bigr)\cong\mathbb{F}_p\oplus\mathbb{F}_px_0.
\]
All these calculations respect weights. For each fixed
weight, the integral chain complex is a bounded complex
of finitely generated free abelian groups. We therefore
apply the Bockstein criterion separately in each weight.

In each weight and cohomological degree, the Bockstein
cohomology has the same dimension over $\mathbb F_p$
as the rational cohomology over $\mathbb Q$.
The remaining claims therefore follow from the Bockstein
criterion (see \cite[Cor.~3.E.4]{Hatcher}).
\ep 

\section{(Co)homology of the Braid Groups}\label{Intcoh1}

By the results of Section~\ref{IsoA}, and in particular by isomorphism~\eqref{HAn1}, 
the cohomological degree $q$ for the braid groups is related to the internal degree 
and weight of the algebra $\G[z]$ by the formula $q=w-\deg$. 
Recall that under the identification of the Fuchs complex with the bar complex, 
the generator $z$ corresponds to the geometric weight $2$, so we now explicitly set $w(z)=2$.

To complete the description of the integral (co)homology of braid groups, it remains 
to restate the material of Section~\ref{Bpm} using this specific weight and the shifted 
cohomological grading. Combining isomorphism~\eqref{HAn1} with Theorem~\ref{th1}, we obtain:

\btr\label{corth1}
For any prime $p$, there is an isomorphism of bigraded
$\mathbb F_p$-vector spaces
\[
H^*\bigl(\mathbf{B}_\infty;\mathbb{F}_p\bigr)
\cong
\tensorn_{r=0}^\infty \mathrm{E}_p[z_{p^r}].
\]
The right-hand side carries the Hopf algebra structure described in
Theorem~\ref{th1}; under the above identification this gives the
algebraic structure used below.
The monomials
\[
h=z_{p^{r_1}}^{(q_1)}\otimes{\dots}\otimes z_{p^{r_m}}^{(q_m)},
\qquad
0\leqslant r_1<\cdots<r_m,\quad q_i>0,
\]
where $w(h)-\deg(h)=q$ and $w(h)\leqslant n$,
form a basis of
$H^q(\mathbf{B}_n;\mathbb F_p)\subset
H^q(\mathbf{B}_\infty;\mathbb F_p)$.\footnote{Here and below, empty products are allowed and represent the unit.}
\etr
\smallskip

Equip $H_*(\mathbf B_\infty;\mathbb F_p)$ with the
Hopf algebra structure dual to the algebraic structure
on $H^*(\mathbf B_\infty;\mathbb F_p)$ introduced in
Theorem~\ref{corth1}.
With this structure, Theorem~\ref{corHP} gives:

\btr\label{corHP00}
For any prime $p$, there is an isomorphism of bigraded
primitively generated Hopf algebras
\[
H_*\bigl(\mathbf{B}_\infty;\mathbb{F}_p\bigr)\cong
\bcs
\Lambda_p[x_0,x_1,x_2,\dots]\otimes\mathbb{F}_p[y_0,y_1,y_2,\dots]
&\quad\mbox{if}\quad p>2,\\[1mm]
\mathbb{F}_2[x_0,x_1,x_2,\dots]&\quad\mbox{if}\quad p=2,
\ecs
\]
where $\deg(x_r)=2p^r-1,\;\deg(y_r)=2p^{r+1}-2,\; w(x_r)=2p^r,\; w(y_r)=2p^{r+1}$.
A basis of $H_q\bigl(\mathbf{B}_n;\mathbb{F}_p\bigr)\subset H_q\bigl(\mathbf{B}_\infty;\mathbb{F}_p\bigr)$
is given by the set of monomials
\[
h=
\bcs
x_{i_1}\wedge\dots\wedge x_{i_a}\otimes y_{j_1}^{r_1}\cdots y_{j_b}^{r_b}&\text{if $p>2$},\\[1mm]
x_{i_1}^{r_1}\cdots x_{i_a}^{r_a}&\text{if $p=2$},
\ecs
\quad\text{where\;\; $i_1<\cdots<i_a,\;\; j_1<\cdots<j_b,\;\; r_\ell>0$},
\]
and $\deg(h)=q$, $ w(h)\leqslant n$.
\etr

Translating Theorem~\ref{CHGZ} through the identifications of
Theorems~\ref{corth1} and~\ref{corHP00}, and using the homological
Bockstein sequence, we obtain the following description of the integral
homology of the braid groups.
\smallskip

\btr\label{HGZ0}
There is an isomorphism of $\mathbb{Z}$-modules
\[
H_q\bigl(\mathbf{B}_\infty\bigr)\cong
\bcs
\mathbb{Z}&\text{if\; $q=0,1$},\\[1mm]
\bigoplus_p\beta_p\bigl(H_{q+1}(\mathbf{B}_\infty;\mathbb{F}_p)\bigr)&\text{if\;
$q\geqslant 2$},
\ecs
\]
where the sum runs over all primes $p$, and $\beta_p$ denotes the
homological Bockstein homomorphism associated with the exact sequence
of coefficients
$0\longrightarrow\mathbb Z\xrightarrow{\times p}\mathbb Z
\longrightarrow\mathbb F_p\longrightarrow 0$.
Under the isomorphism of Theorem~\ref{corHP00}, the corresponding
mod-$p$ Bockstein differential on $H_*(\mathbf{B}_\infty;\mathbb F_p)$
is given by formulas~\eqref{Bockp}.

For $n\geqslant2$, one also has
$H_0(\mathbf{B}_n)\cong H_1(\mathbf{B}_n)\cong\mathbb Z$,
and for $q\geqslant 2$,
\[
H_q\bigl(\mathbf{B}_n\bigr)\cong\plusn_p\beta_p\bigl(H_{q+1}(\mathbf{B}_n;\mathbb{F}_p)\bigr).
\] 
\etr
\smallskip

By the universal coefficient theorem and the
cohomological Bockstein sequence,
Theorem~\ref{HGZ0} implies:

\btr\label{HGZB}
There is an isomorphism
\[
H^q\bigl(\mathbf{B}_\infty\bigr)\cong
\bcs
\mathbb{Z}&\text{if\;
$q=0,1$},\\[1mm]
\bigoplus_p\beta_p^*\bigl(H^{q-1}(\mathbf{B}_\infty;\mathbb{F}_p)\bigr)&\text{if\; $q\geqslant 2$},
\ecs
\]
where the sum runs over all primes $p$, and $\beta_p^*$
denotes the Bockstein homomorphism on cohomologies associated with the exact sequence
of coefficients
$0\longrightarrow\mathbb Z\xrightarrow{\times p}\mathbb Z
\longrightarrow\mathbb F_p\longrightarrow 0$.
Moreover, the group $H^q\bigl(\mathbf{B}_\infty\bigr)$ is squarefree-torsion. For $q\geqslant 2$, we have:
\[
H^q\bigl(\mathbf{B}_n\bigr)\cong\plusn_p\beta_p^*\bigl(H^{q-1}(\mathbf{B}_n;\mathbb{F}_p)\bigr).
\] 
\etr

\brm\label{SegE}
Segal's results \cite{Seg} identify the stable cohomology
of braid groups, with its cup product, with the cohomology
of $\Omega^2S^3$. For an explicit description of
$H_*(\Omega^2S^3)$, see
\cite[Cor.~10.26.4]{AlgMeth}.

The agreement of the additive groups and Bockstein
operators supports the expectation that the multiplication
obtained here from the algebraic Pontryagin product agrees
with this cup product. However, compatibility of our
identifications with multiplication is not established
by the present argument.
It would be desirable to recover the cup product directly
from the Fuchs complex and thereby establish this
compatibility within the algebraic approach developed here.
\erm

\section{Computation of the (Co)homology of Braid Groups}\label{Z}

Although Theorems~\ref{HGZ0} and \ref{HGZB} provide comprehensive information
about the (co)homology of braid groups, they are difficult to apply directly.
In this section, we develop a practical algorithm for concrete calculations.

Since $H^0(\mathbf{B}_n)\cong H^1(\mathbf{B}_n)\cong\mathbb Z$
and $H^q(\mathbf{B}_n)$ is finite and squarefree-torsion, for $q\geqslant2$,
the isomorphism type of $H^q(\mathbf B_n)$ is determined by
\[
H^q(\mathbf{B}_n)\cong
\plusn_{p>0}
\mathbb Z_p^{\oplus
\dim_{\mathbb F_p}\bigl(H^q(\mathbf{B}_n)\otimes\mathbb F_p\bigr)},
\]
where the direct sum is taken over all prime numbers $p>0$.
Therefore, knowing the polynomial
\[
\mathrm{G}_p(\mathbf{B}_n;t):=\sum_{q\geqslant 0}
\dim_{\mathbb{F}_p}\bigl(H^q(\mathbf{B}_n)\otimes\mathbb{F}_p\bigr)t^q
\]
is sufficient to completely determine $H^q(\mathbf{B}_n)$.
The same approach can be used to determine the homology $H_q(\mathbf{B}_n)$. In both cases,
the corresponding polynomials are described below in Lemma~\ref{HoB}.
This description is based on the following corollary of Theorem~\ref{corHP00}:

\bcr\label{CoBn}
Denote by $H_q^{(w)}\bigl(\mathbf{B}_\infty;\mathbb{F}_p\bigr)$ the $\mathbb{F}_p$-vector subspace
of $H_q\bigl(\mathbf{B}_\infty;\mathbb{F}_p\bigr)$ spanned by all homology classes of weight $w$.
Then for any prime $p$, we have:
\[
\mathcal{H}_p(t, u):=\sum_{q\geqslant 0}\sum_{w\geqslant 0}\dim_{\mathbb{F}_p} 
H_q^{(w)}\bigl(\mathbf{B}_\infty;\mathbb{F}_p\bigr)\,u^wt^q=
\prod_{r\geqslant 0}\frac{1+u^{2p^r}t^{2p^r-1}}{1-u^{2p^{r+1}}t^{2p^{r+1}-2}}.
\]
Moreover, for any $n\geqslant 1$, let\; $\mathcal{H}_p(t, u)=1+\sum_{q=1}^\infty f_q(u)t^q$. Then
\[
\mathrm{H}_{p,n}(t):=\sum_{q\geqslant 0}
\dim_{\mathbb{F}_p}H_q(\mathbf{B}_n;\mathbb{F}_p)t^q=
1+\sum_{0<q\leqslant n}\bigl(f_q(u)\bmod u^{n+1}\bigr)\Big|_{u=1}t^q.
\]
In particular, by Arnold's stability theorem,
\[
\mathrm{H}_p(t):=\sum_{q\geqslant 0}
\dim_{\mathbb{F}_p} H_q\bigl(\mathbf{B}_\infty;\mathbb{F}_p\bigr)\,t^q=
\prod_{r\geqslant 0}\;\frac{1+t^{2p^r-1}}{1-t^{2p^{r+1}-2}}.
\]
\ecr

\blr\label{HoB}
For any prime $p$ and $n\geqslant 2$, we have the following identities:
\begin{align}
\sum_{q\geqslant 0}
\dim_{\mathbb{F}_p}\big(H_q(\mathbf{B}_n)\otimes\mathbb{F}_p\big)t^q
&=
t+\frac{1}{1+t}\mathrm{H}_{p,n}(t),\label{eqh1}\\
\sum_{q\geqslant 0}\dim_{\mathbb{F}_p}\big(H^q(\mathbf{B}_n)\otimes\mathbb{F}_p\big)t^q
&=
1+\frac{t}{1+t}\mathrm{H}_{p,n}(t).\label{eqh2}
\end{align}
\elr

\bp
Set
$
h_p(q,n):=\dim_{\mathbb{F}_p} H_q(\mathbf{B}_n;\mathbb{F}_p),\;
\overline{h}_p(q,n):=\dim_{\mathbb{F}_p}\bigl(H_q(\mathbf{B}_n)\otimes\mathbb{F}_p\bigr).
$
For brevity, for fixed $p$ and $n$, we use the shorthand
$a_q:=h_p(q,n)$ and $b_q:=\overline{h}_p(q,n)$,
and let $\overline{\mathrm{H}}_{p,n}(t)$ be the left-hand side of identity \eqref{eqh1}.

For any prime $p$ and $q\geqslant 0$, the universal coefficient theorem yields the short exact sequence
\[
0\to H_q(\mathbf{B}_n)\otimes\mathbb{F}_p\to H_q(\mathbf{B}_n;\mathbb{F}_p)
\to\mathrm{Tor}\bigl(H_{q-1}(\mathbf{B}_n),\mathbb{F}_p\bigr)\to 0.
\]
It shows that
$
a_q=b_q+\dim_{\mathbb{F}_p}\mathrm{Tor}\bigl(H_{q-1}(\mathbf{B}_n),\mathbb{F}_p\bigr).
$
Since $H_0(\mathbf{B}_n)\cong H_1(\mathbf{B}_n)\cong\mathbb{Z}$ and the groups $H_q(\mathbf{B}_n)$
are squarefree-torsion for $q\geqslant 2$, this expression implies:
\begin{gather*}
a_0=b_0=1,\quad a_1=b_1=1,\quad a_2=b_2,\quad
a_q=b_{q-1}+b_q\quad\text{for $q\geqslant 3$}.
\end{gather*}
Thus,
\begin{multline*}
\mathrm{H}_{p,n}(t)=1+t+b_2t^2+\sum_{q=3}^{\infty}(b_{q-1}+b_q)t^q\\
=t\Big(\sum_{q=2}^{\infty}b_qt^q\Big)+\overline{\mathrm{H}}_{p,n}(t)
=t\big(\overline{\mathrm{H}}_{p,n}(t)-1-t\big)+\overline{\mathrm{H}}_{p,n}(t)=
(1+t)\overline{\mathrm{H}}_{p,n}(t)-t(1+t).
\end{multline*}
From this expression, identity \eqref{eqh1} follows.

Since the groups $H_q(\mathbf{B}_n)$ are squarefree-torsion,
from the universal coefficient theorem we obtain a (non-canonical) splitting
\[
H^q\bigl(\mathbf{B}_n\bigr)\cong\Hom\bigl(H_q(\mathbf{B}_n),\mathbb{Z}\bigr)\plus\Ext\bigl(H_{q-1}(\mathbf{B}_n),\mathbb{Z}\bigr)
\cong
\bcs
\mathbb{Z}&\text{if } q=0,1,\\
0&\text{if } q=2,\\
H_{q-1}(\mathbf{B}_n)&\text{if } q\geqslant 3.
\ecs
\]
Tensoring this isomorphism with $\mathbb{F}_p$ yields that
$\dim_{\mathbb{F}_p}\bigl(H^q(\mathbf{B}_n)\otimes\mathbb{F}_p\bigr)=b_{q-1}$ for $q\geqslant 3$.
For $q\in\{0,1,2\}$, the dimensions are $1,1$, and $0$, respectively.
Multiplying by $t^q$ and summing over all $q \geqslant 0$ gives:
\begin{multline*}
\sum_{q\geqslant 0}\dim_{\mathbb{F}_p}\big(H^q(\mathbf{B}_n)\otimes\mathbb{F}_p\big)t^q
=1+t+\sum_{q=3}^{\infty} b_{q-1}t^q \\
=1+t+t\Big(\sum_{j=2}^{\infty}b_jt^j\Big)
=1+t+t\big(\overline{\mathrm{H}}_{p,n}(t)-1-t\big)
=1-t^2+t\overline{\mathrm{H}}_{p,n}(t).
\end{multline*}
Substituting $\overline{\mathrm{H}}_{p,n}(t)=t+\frac{1}{1+t}\mathrm{H}_{p,n}(t)$
from \eqref{eqh1} yields identity \eqref{eqh2}.
\ep

Lemma \ref{HoB} implies formula \eqref{GenB} for $p>0$.
For $p=0$, Arnold's theorem gives
\[
\mathrm G_0(\mathbf B_\infty;t)=1+t,
\]
since the rational cohomology of the braid groups $\mathbf B_n$,
$n\geqslant2$, is non-trivial only in degrees $0$ and $1$.

We already know that $H^0(\mathbf{B}_n)\cong H^1(\mathbf{B}_n)\cong\mathbb{Z}$, $H^2(\mathbf{B}_n)=0$,
$H^q(\mathbf{B}_n)=0$ for $q \geqslant n$, and
the groups $H^q(\mathbf{B}_n)$ are finite for $q\geqslant 3$.
The first few groups $H^q(\mathbf{B}_n)$ with $q\geqslant 3$
are presented in the following table, taking into account Arnold's repetition theorem.
Empty cells, as well as entries equal to $0$, correspond to trivial groups;
$\mathbb{Z}_p^m:=\mathbb{Z}_p^{\oplus m}$, and stable groups are highlighted in bold.
\begin{center}
\renewcommand{\arraystretch}{1.2}
{\scriptsize
\begin{tabular}{|c||c|c|c|c|c|c|c|c|c|c|c|c|c|c|c|c|}
\multicolumn{15}{c}{}\\
\hline
$n\setminus q$ &3&4&5&6&7&8&9&10&11&12&13&14&15&16 \\
\hline
\hline
4,5 &$\mathbf{Z}_2$ &&&&&&&&&&&&&\\
\hline
6,7 &$\mathbf{Z}_2$&$\mathbf{Z}_2$&$\mathbb{Z}_3$ &&&&&&&&&&&\\
\hline
8,9 &$\mathbf{Z}_2$&$\mathbf{Z}_2$&$\mathbf{Z}_6$&$\mathbb{Z}_3$&$\mathbb{Z}_2$ &&&&&&&&&\\
\hline
10,11 & $\mathbf{Z}_2$&$\mathbf{Z}_2$&$\mathbf{Z}_6$&$\mathbf{Z}_6$&$\mathbb{Z}_2$&$\mathbb{Z}_2$&$\mathbb{Z}_5$&&&&&&&\\
\hline
12,13 & $\mathbf{Z}_2$&$\mathbf{Z}_2$&$\mathbf{Z}_6$&$\mathbf{Z}_6$&$\mathbf{Z}^2_2$
&$\mathbb{Z}_2$&$\mathbb{Z}_{30}$&$\mathbb{Z}_{10}$ &&&&&&\\
\hline
14,15 &$\mathbf{Z}_2$&$\mathbf{Z}_2$&$\mathbf{Z}_6$&$\mathbf{Z}_6$&$\mathbf{Z}^2_2$
&$\mathbf{Z}^2_2$&$\mathbb{Z}_{30}$&$\mathbb{Z}_2\oplus\mathbb{Z}_{30}$&$\mathbb{Z}_2$
&0&$\mathbb{Z}_7$&&&\\
\hline
16,17 &$\mathbf{Z}_2$&$\mathbf{Z}_2$&$\mathbf{Z}_6$&$\mathbf{Z}_6$&$\mathbf{Z}^2_2$
&$\mathbf{Z}^2_2$&$\mathbf{Z}_2\oplus\mathbf{Z}_{30}$&$\mathbb{Z}_2\oplus\mathbb{Z}_{30}
$&$\mathbb{Z}^2_2$&$\mathbb{Z}_2$&$\mathbb{Z}_{14}$&$\mathbb{Z}_7$&$\mathbb{Z}_2$ & \\
\hline
18,19 &$\mathbf{Z}_2$&$\mathbf{Z}_2$&$\mathbf{Z}_6$&$\mathbf{Z}_6$&$\mathbf{Z}^2_2$
&$\mathbf{Z}^2_2$&$\mathbf{Z}_2\oplus\mathbf{Z}_{30}$&$\mathbf{Z}^2_2\oplus\mathbf{Z}_{30}
$&$\mathbb{Z}^2_2$&$\mathbb{Z}^2_2$&$\mathbb{Z}^2_2\oplus\mathbb{Z}_{21}$&$\mathbb{Z}_2\oplus\mathbb{Z}_7$
&$\mathbb{Z}_2$&$\mathbb{Z}_2$ \\ \hline
20,21 &$\mathbf{Z}_2$&$\mathbf{Z}_2$&$\mathbf{Z}_6$&$\mathbf{Z}_6$&$\mathbf{Z}^2_2$
&$\mathbf{Z}^2_2$&$\mathbf{Z}_2\oplus\mathbf{Z}_{30}$&$\mathbf{Z}^2_2\oplus\mathbf{Z}_{30}
$&$\mathbf{Z}^3_2$&$\mathbb{Z}^2_2$&$\mathbb{Z}^3_2\oplus\mathbb{Z}_{21}$&$\mathbb{Z}^2_2\oplus\mathbb{Z}_{21}$
&$\mathbb{Z}^2_2$&$\mathbb{Z}^2_2$ \\ \hline 
\end{tabular}
}
\end{center}

\brm\label{Steenrod}
Interestingly, for any prime $p$, the generating function $\mathrm{H}_p(t)$ for the sequence
$\dim_{\mathbb{F}_p}H_q(\mathbf{B}_\infty;\mathbb{F}_p)$
coincides with the graded dimension of the Steenrod algebra
$\mathcal{A}_p$. However, $H_*(\mathbf{B}_\infty;\mathbb{F}_p)$
and $\mathcal{A}_p$ are not isomorphic as Hopf algebras. Indeed,
$H_*(\mathbf{B}_\infty;\mathbb{F}_p)$ is a primitively generated
Hopf algebra, whereas neither $\mathcal{A}_p$ nor its dual has this property;
see~\cite{Milnor}.
\erm

\section{Cohomology of the Groups
\texorpdfstring{$\mathbf{C}_n$ and $\widetilde{\mathbf{C}}_n$}{C_n and \textasciitilde C_n}}
\label{Cn}

First, we describe the cohomology of the complex $\mathrm{C}(n)$ for $n\geqslant 1$, where $C_1:=A_1$.

Lemma~\ref{CompB} allows us to identify $\mathrm{C}(n)$
with the module spanned by the compositions
\[
(n_1,\dots,n_a,r+1)
\]
of weight $n+1$,
where $0\leqslant r\leqslant n$.
We write the corresponding cochain as
\[
e(n_1,\ldots,n_a)\times y^r.
\]

\blr\label{ACnT}
For $n\geqslant 1$, there is a canonical isomorphism of complexes
\[
\mathrm{C}(n)\cong\plusn_{r\geqslant 0}\mathrm{A}(n-r)\times y^r,\qquad\text{where}\qquad
\mathrm{C}^q(n)\cong\plusn_{r\geqslant 0}\mathrm{A}^{q-r}(n-r)\times y^r,
\]
such that $\partial_C(c\times y^r)=\partial(c)\times y^r$ for $c\in\mathrm{A}(n-r)$.
\elr

\bp
The claimed isomorphism of modules follows directly from the definition above.
The formula $\partial_C(c\times y^r)=\partial(c)\times y^r$
for the basis cochains follows from the definition of the Salvetti complex,
since for basis elements this definition shows that
\[
\partial_C\bigl(e(n_1,\dots,n_a)\times y^r\bigr)=
\partial\bigl(e(n_1,\dots,n_a)\bigr)\times y^r.
\]
This follows from the formula
\beq\label{CC}
\left.\frac{C_{n_a+r}(t)}{A_{n_a-1}(t)\cdot C_r(t)}\right|_{t=-1}=0,
\eeq
where $C_0(t)=A_0(t):=1$. This completes the proof.
\ep

Lemma~\ref{ACnT} implies the following result:
\smallskip

\btr[\cite{Gor}]\label{CohC}
For $n\geqslant 1$, we have:
\beq\label{HCn}
H^q\bigl(\mathbf{C}_n\bigr)\cong\plusn_{r\geqslant 0}H^{q-r}\bigl(\mathbf{B}_{n-r}\bigr).
\eeq
\etr

\btr\label{StabS_C}
For $n\geqslant 2$, there is a canonical morphism of complexes
\[
\mathrm{St}_C:\mathrm{C}(n)\longrightarrow\mathrm{C}(n-1)
\]
such that, for any $q\geqslant 1$, the induced homomorphism
$
\mathrm{St}^*_C:H^q(\mathbf{C}_n)\longrightarrow H^q(\mathbf{C}_{n-1})
$
is surjective, and is an isomorphism if $n\geqslant\max\{q+2,2q-1\}$.
Thus, the cohomology of the family $\mathbf{C}$ stabilizes.
\etr

\bp
For $0\leqslant r\leqslant n$ and $c\in\mathrm{A}(n-r)$, set 
$\mathrm{St}_C(c\times y^r)=\mathrm{St}(c)\times y^r$. Then Lemma~\ref{ACnT} implies that
$\mathrm{St}_C$ defines the declared morphism of complexes.
For $q\geqslant 0$, the map
\[
H^q(\mathrm{C}(n))\stackrel{\mathrm{St}^*_C}\llongrightarrow H^q(\mathrm{C}(n-1))
\]
is induced by the set of maps
\beq\label{r}
H^{q-r}(\mathrm{A}(n-r))\stackrel{\mathrm{St}^*}\llongrightarrow H^{q-r}(\mathrm{A}(n-r-1)),
\eeq
where $r\geqslant 0$.
Since Theorem~\ref{StabSH} imply that $\mathrm{St}^*$ is surjective for all $r\geqslant0$,
the map $\mathrm{St}^*_C$ is surjective as well.
If $q\geqslant 2$ and $n\geqslant 2q-1$, then $n-r\geqslant 2(q-r)-1$ for any $r\geqslant 0$.
Theorem \ref{StabSH} shows that the map \eqref{r} is an isomorphism if $q-r\geqslant 2$.
For $q-r\in\{0,1\}$ the maps
\[
H^1(\mathrm{A}(n-q+1))\stackrel{\mathrm{St}^*}\longrightarrow H^1(\mathrm{A}(n-q))\qquad\text{and}
\qquad H^0(\mathrm{A}(n-q))\stackrel{\mathrm{St}^*}\longrightarrow H^0(\mathrm{A}(n-q-1))
\]
are both isomorphisms only if $n\geqslant q+2$. The proof is completed.
\ep

\bcr\label{StabC}
The stable cohomology of the family $\mathbf{C}$ is given by the isomorphism:
\[
H^q\bigl(\mathbf{C}_\infty\bigr)\cong\plusn_{r\geqslant 0}H^{q-r}(\mathbf{B}_\infty).
\]
Thus, for any prime $p$, as well as for $p=0$, we have:
\[
\mathrm{G}_p\bigl(\mathbf{C}_\infty;t\bigr)=
\frac{1}{1-t}\;\mathrm{G}_p(\mathbf{B}_\infty;t).
\]
\ecr

In addition, from Theorem~\ref{CohC} and Corollary~\ref{Arn1}, we obtain:
\smallskip

\bcr\label{DC}
For $n\geqslant 2$, we have
\[
\dim_{\mathbb{Q}} H^q\bigl(\mathbf{C}_n;\mathbb{Q}\bigr)=
\bcs
1&\text{if $q=0$ or $q=n$},\\
2&\text{if $1\leqslant q<n$},\\
0&\text{otherwise}.
\ecs
\]
\ecr

Similarly, we describe the cohomology of the complex
$\widetilde{\mathrm{C}}(n)$ for $n\geqslant3$.

Using Lemma~\ref{CompB}, we identify $\widetilde{\mathrm{C}}(n)$
with the module spanned by the compositions
\[
(l+1,n_1,\ldots,n_a,r+1)
\]
of weight $n+1$, where $l,r\geqslant0$ and $l+r\leqslant n-1$.
We write the corresponding cochain as
\[
x^l\times e(n_1,\ldots,n_a)\times y^r,\qquad\text{where}\qquad n_1+\cdots+n_a=n-l-r-1.
\]
Thus the middle part gives the complex
$\mathrm A(n-l-r-1)$. If $l+r=n-1$, the middle part is empty,
and we write the corresponding basis element simply as $x^l\times y^r$.

\blr\label{HCA}
For $n\geqslant3$, there is an isomorphism of complexes
\[
\widetilde{\mathrm{C}}(n)\cong
\plusn_{m\geqslant 1}
\mathrm A(n-m)^{\oplus m},\qquad\text{where}
\qquad\widetilde{\mathrm{C}}^q(n)\cong
\plusn_{m\geqslant 1}
\mathrm A^{q-m+1}(n-m)^{\oplus m}.
\]
\elr

\begin{proof}
Let $l,r\geqslant0$ be fixed. Consider the submodule
\[
D_{l,r}(n):=
x^l\times \mathrm A(n-l-r-1)\times y^r
\subset \widetilde{\mathrm{C}}(n).
\]
Here $0\leqslant l+r\leqslant n-1$.

The possible boundary terms which join the middle $A$-part to the left
or to the right terminal block vanish by the same calculation as in
formula~\eqref{CC}. Hence the differential preserves $D_{l,r}(n)$.
Before changing signs, it is given by
\[
\partial_{\widetilde C}
\bigl(x^l\times c\times y^r\bigr)
=
(-1)^l x^l\times \partial(c)\times y^r,
\qquad
c\in \mathrm A(n-l-r-1).\footnote{We shall use the following elementary observation. If $K$ is a
cochain complex, then the complexes $(K,d)$ and $(K,-d)$ are
isomorphic, via $c\mapsto (-1)^{\deg(c)}c$. More generally, a
constant factor $(-1)^l$ in the differential on a fixed summand can
be removed by the automorphism $c\mapsto (-1)^{l\deg(c)}c$.}
\]
The factor $(-1)^l$ appears because the differential on the middle
$A$-part is shifted past the $l$ vertices in the left terminal block.
By the elementary observation in the footnote, this constant factor
$(-1)^l$ can be removed by a degree-wise automorphism. Thus
$D_{l,r}(n)$ is identified with the usual complex
$\mathrm A(n-l-r-1)$, shifted in degree by $l+r$.

Set $m=l+r+1$. For a fixed $m$, where $1\leqslant m\leqslant n$, there are exactly
$m$ pairs $(l,r)$ with $l+r=m-1$, namely
\[
(l,r)=(0,m-1),(1,m-2),\ldots,(m-1,0).
\]
For each of these pairs we have $n-l-r-1=n-m$.
Moreover, multiplication by $x^l$ and $y^r$ shifts the cohomological degree by $l+r=m-1$.

Therefore a cochain $c\in \mathrm A^j(n-m)$
contributes to total degree $q=j+m-1$, or equivalently $j=q-m+1$.
Hence
\[
\widetilde{\mathrm{C}}^q(n)\cong
\plusn_{m\geqslant1}
\mathrm A^{q-m+1}(n-m)^{\oplus m}.
\]
With respect to these identifications, the differential is the usual
differential $\partial$ on each copy of $\mathrm A(n-m)$.
This proves the claimed isomorphism of complexes.
\end{proof}

\btr\label{CohCC}
For $n\geqslant3$, we have
\[
H^q\bigl(\widetilde{\mathbf C}_n\bigr)\cong
\plusn_{m\geqslant 1}
H^{q-m+1}\bigl(\mathbf B_{n-m}\bigr)^{\oplus m}.
\]
\etr

\btr\label{StabS_CC}
For $n\geqslant 4$, there is a canonical morphism of complexes
\[
\mathrm{St}_{\widetilde C}:\widetilde{\mathrm{C}}(n)
\longrightarrow
\widetilde{\mathrm{C}}(n-1),
\]
such that for any $q\geqslant1$, the induced homomorphism
$\mathrm{St}^*_{\widetilde C}:
H^q(\widetilde{\mathbf C}_n)\longrightarrow
H^q(\widetilde{\mathbf C}_{n-1})$
is surjective, and is an isomorphism if $n\geqslant\max\{q+4,2q\}$.
Thus, the cohomology of the family $\widetilde{\mathbf C}$ stabilizes.
\etr

\bp
The proof is similar to that of Theorem~\ref{StabS_C},
with the following modifications.
Under the decomposition of Lemma~\ref{HCA}, the morphism
$\mathrm{St}_{\widetilde C}$ is induced on the summand
$\mathrm A(n-m)$ by the stabilization map
$\mathrm{St}:\mathrm A(n-m)\longrightarrow \mathrm A(n-m-1)$.
Hence on cohomology it is induced by
\[
H^{q-m+1}(\mathrm A(n-m))
\longrightarrow
H^{q-m+1}(\mathrm A(n-m-1)).
\]
Put $d=q-m+1$ and $k=n-m$.
Surjectivity follows from Theorem~\ref{StabSH}.
For $d\geqslant2$, the inequality $n\geqslant2q$
implies $k\geqslant2d-1$, so the map is an
isomorphism by the same theorem.
For $d=0,1$, the assertion follows directly from
the cohomology of braid groups in these degrees,
using $n\geqslant q+4$.
\ep

\bcr\label{StabCC}
The stable cohomology of the family $\widetilde{\mathbf{C}}$ is given by the isomorphism
\[
H^q\bigl(\widetilde{\mathbf C}_\infty\bigr)\cong
\plusn_{m\geqslant1}
\bigl(H^{q-m+1}(\mathbf B_\infty)\bigr)^{\oplus m}.
\]
Thus, for any prime $p$, as well as for $p=0$, we have
\[
\mathrm G_p\bigl(\widetilde{\mathbf C}_\infty;t\bigr)=
\frac{1}{(1-t)^2}\,
\mathrm G_p(\mathbf B_\infty;t).
\]
\ecr

\bcr\label{DCC}
For $n\geqslant3$, we have
\[
\dim_{\mathbb Q}H^q(\widetilde{\mathbf C}_n;\mathbb Q)=
\begin{cases}
2q+1, & \text{if $0\leqslant q\leqslant n-2$},\\
n, & \text{if $q=n-1$},\\
0, & \text{otherwise}.
\end{cases}
\]
\ecr

\brm
Corollary~\ref{StabC} shows that, as graded abelian groups,
\[
H^*(\mathbf{C}_\infty)\cong
H^*(\mathbf{B}_\infty)\otimes H^*(\Omega S^2).
\]
Similarly, Corollary~\ref{StabCC} shows that, as graded abelian groups,
\[
H^*(\widetilde{\mathbf{C}}_\infty)\cong
H^*(\mathbf{B}_\infty)\otimes H^*(\Omega S^2\times\Omega S^2).
\]

Note that in his 1974 paper \cite{FuchsQ}, Fuchs announced the existence of a homological
equivalence between the spaces $K(\mathbf{C}_\infty,1)$ and $\Omega^2 S^3\times\Omega S^2$
(see Remark~\ref{SegE}).
\erm

\section{Cohomology of the Groups $\mathbf{D}_n$}\label{ADn}

In this section, we describe the cohomology
of the complex $\mathrm{D}(n)$, where $n\geqslant 4$.
Define $D_1:=A_1$, $D_2:=A_1\times A_1$, and $D_3:=A_3$.

The assignments 
$\sigma(s_{n-1})=s_n,\;\sigma(s_n)=s_{n-1}$ and $\sigma(s_i)=s_i$ if $i\leqslant n-2$,
extend to an involution $\sigma$ of the complex $\mathrm{D}(n)$.

For $l<n$, identify the group 
$\langle s_1,s_2,\dots,s_l\rangle\subset D_n$
with the Coxeter group $A_l$, and regard the complex
$\mathrm{A}(l+1)$ as a submodule of $\mathrm{D}(n)$.

For an integer $r$ with $2\leqslant r\leqslant n$, set
\[
P_r:=\langle s_{n-r+1},\dots,s_n\rangle\subseteq D_n.
\]
For a parabolic subgroup $P\subseteq D_n$, let $h(P)$ be the maximal integer
$r\geqslant 2$ such that $P_r\subseteq P$.
If no such subgroup $P_r$ exists, set $h(P)=0$.

\bdr\label{DefDr}
Let $\mathrm{D}_r(n)$ denote the submodule of $\mathrm{D}(n)$ generated by the
parabolic subgroups $P\subset D_n$ with $h(P)=r\geqslant 2$.
We use the identification
\[
\mathrm{D}_0(n)=\mathrm{A}(n),\qquad
\mathrm{D}_r(n)=\mathrm{A}(n-r)\times y^r\quad\text{for } 2\leqslant r\leqslant n,
\qquad
\mathrm{D}_r(n)=0 \quad \text{for } r>n.
\]
For $i\geqslant 0$, define the submodule $\mathrm{T}_i(n)$ of $\mathrm{D}(n)$ by
\[
\mathrm{T}_i(n):=
\bcs
\mathrm{A}(n)+\sigma(\mathrm{A}(n))&\text{if $i=0$},\\
\mathrm{D}_{2i}(n)\oplus\mathrm{D}_{2i+1}(n)&\text{if $i\geqslant 1$}.
\ecs
\]
\edr
It is clear that
\beq\label{DecD}
\mathrm{D}(n)=\mathrm{T}_0(n)\plus\mathrm{T}(n),\qquad\text{where}\qquad
\mathrm{T}(n)=\plusn_{1\leqslant i\leqslant\left\lfloor\frac{n}{2}\right\rfloor}\mathrm{T}_i(n).
\eeq

Let $\mathrm{A}_1(n)$ be the submodule of $\mathrm{A}(n)$
generated by the cochains $e(n_1,\dots,n_{a-1},1)$.
Then $\mathrm{A}(n)=\mathrm{A}_1(n)\oplus\mathrm{A}_2(n)$.

\blr\label{d00}
For $n\geqslant 4$, the action of $\partial_D$ on
$\mathrm{T}_0(n)$ is determined by the commutation relation
$\partial_D\circ\sigma=\sigma\circ\partial_D$ and the formula
\[
\partial_D(c)=
\bcs
\partial(c)\dotPlus\sigma\bigl(\partial(c)\bigr)&\text{if $c\in\mathrm{A}_1(n)$},\\
\partial(c)&\text{if $c\in\mathrm{A}_2(n)$}.
\ecs
\]
In particular, $\mathrm{T}_0(n)$ is a subcomplex of\, $\mathrm{D}(n)$. 
\elr

\bp
For $c\in\mathrm{A}_1(n)$ the formula follows from the definition of $\partial_D$.

For a basis cochain $c=e(n_1,\dots,n_a)\in\mathrm{A}_2(n)$, the definition of $\partial_D$ yields:
\[
\partial_D(c)=\partial(c)
+(-1)^a\left.\frac{D_{n_a}(t)}{A_{n_a-1}(t)}\right|_{t=-1}
e(n_1,\dots,n_{a-1})\times y^{n_a}
=\partial(c).
\]
This proves the formula for $c\in\mathrm{A}_2(n)$.
\ep

\blr\label{d0}
For $n\geqslant 4$, there is a canonical isomorphism of complexes
\[
\mathrm{T}_0(n)\cong\mathrm{A}(n)\oplus\mathrm{A}_2(n).
\]
Thus, for $n\geqslant 4$, we have the isomorphism
\[
H^q\bigl(\mathrm{T}_0(n)\bigr)\cong
H^q(\mathbf{B}_n)\oplus\left[H^q(\mathbf{B}_{n})\big/H^q(\mathbf{B}_{n-1})\right].
\]
\elr

\bp
Define
\[
\Delta\bigl(\mathrm{A}_2(n)\bigr):=\{\,c+\sigma(c)\mid c\in\mathrm{A}_2(n)\,\}.
\]
Then
\[
\mathrm{T}_0(n)
=\mathrm{A}_2(n)\oplus\mathrm{A}_1(n)\oplus\sigma\bigl(\mathrm{A}_2(n)\bigr)
=\mathrm{A}_1(n)\oplus\Delta\bigl(\mathrm{A}_2(n)\bigr)\oplus\mathrm{A}_2(n).
\]
The formula for $\partial_D$ given in Lemma~\ref{d00}
shows that $\mathrm{A}_1(n)\oplus\Delta\bigl(\mathrm{A}_2(n)\bigr)$ is a subcomplex of
$\mathrm{T}_0(n)$ isomorphic to $\mathrm{A}(n)$.
Therefore, this decomposition of $\mathrm{T}_0(n)$ yields the claimed isomorphism.
The final isomorphism follows from Corollary~\ref{A2}.
\ep

Let us now turn to the study of the modules $\mathrm{T}_i(n)\subset\mathrm{D}(n)$, where $i>0$.

\blr\label{di}
For a cochain $c\times y^r\in\mathrm{D}_r(n)$ with $c\in\mathrm{A}(n-r)$ and $r\geqslant 2$, we have
\[
\partial_D(c\times y^r)=
\partial(c)\times y^r+(-1)^{\deg(c)}\chi(r)\cdot\mathrm{St}(c)\times y^{r+1},
\]
where $\chi(r):=1+(-1)^r$.
Hence, for $i\geqslant 1$, $\mathrm{T}_i(n)$ is a subcomplex of $\mathrm{D}(n)$.
\elr

\bp
It suffices to verify the lemma for the basis cochains $c=e(n_1,\dots,n_a)\in\mathrm{A}(n-r)$.
The result then follows from the definition of $\partial_D$, together with the easily verified identity
\[
\left.\frac{D_{n_a+r}(t)}{A_{n_a-1}(t)\cdot D_r(t)}\right|_{t=-1}=
\bcs
0 &\text{if\; $n_a>1$},\\
\chi(r)\,&\text{if\; $n_a=1$}.
\ecs
\qedhere
\]
\ep

\bcr\label{Sum}
For $n\geqslant 4$, there is an isomorphism:
\[
H^q(\mathrm{T}(n))\cong
\bcs
\plusn_{1\leqslant i\leqslant\left\lfloor\frac{n-1}{2}\right\rfloor}H^q(\mathrm{T}_i(n))
&\text{if $q<n$},\\[2mm]
\mathbb{Z}\big/\chi(n-1)\mathbb{Z}&\text{if $q=n$}.
\ecs
\]
\ecr

To describe $H^*\bigl(\mathrm{T}_i(n)\bigr)$ for $i>0$,
we present the complex $\mathrm{T}_i(n)$ as the mapping cone of a morphism of complexes.

Recall the definition and the main properties of the
mapping cone in the cohomological setting (see \cite{Rot}, p.350).

\blr\label{cone}
Let $(K_0, d_0)$ and $(K_1, d_1)$ be cochain complexes,
and let $V:K_0\longrightarrow K_1$ be a morphism of complexes.
Consider the graded module $\mathrm{Cn}:=\bigoplus_m \mathrm{Cn}^m$,
where $\mathrm{Cn}^m:=K_0^m\oplus K_1^{m-1}$. We equip $\mathrm{Cn}$ with a linear map
$\partial_\mathrm{Cn}:\mathrm{Cn}\to \mathrm{Cn}$ whose restriction $\partial_m:\mathrm{Cn}^m\to \mathrm{Cn}^{m+1}$ is defined by
\[
\partial_m(a,b):=\bigl(d_0(a),\,(-1)^m V(a)+d_1(b)\bigr).
\]
Then $\partial_\mathrm{Cn}$ defines a cochain complex structure on $\mathrm{Cn}$.
Moreover, $\mathrm{Cn}$ fits into a long exact sequence in cohomology:
\[
\cdots\longrightarrow H^{m-1}(K_0)\xrightarrow{\,V^*\,} H^{m-1}(K_1)\longrightarrow H^m(\mathrm{Cn})
\longrightarrow H^m(K_0)\xrightarrow{\,V^*\,} H^m(K_1)\longrightarrow\cdots.
\]
\elr

\bp
Checking that $\partial^2_m=0$ is straightforward.

Introduce the shifted subcomplex $K_1(-1)\subset \mathrm{Cn}$, defined by
\[
K_1(-1)^m:=\{(0,c)\mid c\in K_1^{m-1}\},
\]
with the differential induced from $\mathrm{Cn}$, that is,
$d(0,c)=(0,d_1c)$.

The natural inclusion and projection maps yield a short exact sequence of cochain complexes:
\[
0\longrightarrow K_1(-1)\longrightarrow \mathrm{Cn}\longrightarrow K_0\longrightarrow 0.
\]
This sequence induces the declared long exact sequence in cohomology.
\ep

\bdr
The cochain complex $\mathrm{Cn}$ defined in Lemma~\ref{cone} is called the \emph{mapping cone}
of the morphism of complexes $V:K_0\longrightarrow K_1$.
\edr

\bpr\label{HTi00}
For $k\geqslant 2$, let $\mathrm{Cn}(k)$ be the mapping cone of the morphism
$2\,\mathrm{St}:\mathrm{A}(k)\to\mathrm{A}(k-1)$. Then, for every $m\geqslant0$, there is an isomorphism:
\[
H^m\bigl(\mathrm{Cn}(k)\bigr)\cong H^m\bigl(\mathrm A_2(k)\bigr)
\plus H^{m-1}\bigl(\mathrm A(k-1)\bigr)\otimes\mathbb Z_2
\plus H^m\bigl(\mathrm A(k-1)\bigr)[2].
\]
\epr

\bp
For $k=2$, the complexes $\mathrm A(2)$ and
$\mathrm A(1)$ have zero differentials, with
\[
\mathrm A^0(2)=\mathbb Z\cdot e(1,1),\qquad
\mathrm A^1(2)=\mathbb Z\cdot e(2),\qquad
\mathrm A^0(1)=\mathbb Z\cdot e(1).
\]
Since $\mathrm{St}(e(1,1))=e(1)$ and $\mathrm{St}(e(2))=0$,
the cone $\mathrm{Cn}(2)$ is
\[
\mathbb Z\xrightarrow{\,a\mapsto(0,2a)\,}
\mathbb Z\oplus\mathbb Z,
\]
concentrated in degrees $0$ and $1$. Hence
\[
H^m(\mathrm{Cn}(2))\cong
\begin{cases}
\mathbb Z\oplus\mathbb Z_2,&m=1,\\
0,&m\ne1.
\end{cases}
\]
This agrees with the asserted formula, since
$\mathrm A_2(2)=\mathbb Z\cdot e(2)$ is concentrated
in degree $1$ and $\mathrm A(1)$ in degree $0$.

For $k\geqslant 3$, the cone exact sequence for $\mathrm{Cn}(k)$,
\begin{multline*}
\cdots\longrightarrow
H^{m-1}\bigl(\mathrm{A}(k)\bigr)\xrightarrow{2\,\mathrm{St}^*\,}
H^{m-1}\bigl(\mathrm{A}(k-1)\bigr)\longrightarrow H^m\bigl(\mathrm{Cn}(k)\bigr)\\
\longrightarrow
H^{m}\bigl(\mathrm{A}(k)\bigr)\xrightarrow{2\,\mathrm{St}^*\,}
H^{m}\bigl(\mathrm{A}(k-1)\bigr)\longrightarrow\cdots,
\end{multline*}
yields a short exact sequence
\[
0\longrightarrow
H^{m-1}\bigl(\mathrm{A}(k-1)\bigr)\big/\operatorname{\mathrm{Im}}(2\mathrm{St}^*)
\longrightarrow
H^m\bigl(\mathrm{Cn}(k)\bigr)
\longrightarrow
\Ker(2\,\mathrm{St}^*)
\longrightarrow 0.
\]
Since $\mathrm{St}^*$ is a surjection, we have
$\operatorname{\mathrm{Im}}(2\mathrm{St}^*)=2H^{m-1}\bigl(\mathrm{A}(k-1)\bigr)$,
and hence there is a short exact sequence
\[
0\longrightarrow
H^{m-1}\bigl(\mathrm{A}(k-1)\bigr)\otimes\mathbb{Z}_2
\longrightarrow
H^m\bigl(\mathrm{Cn}(k)\bigr)
\longrightarrow
\Ker(2\,\mathrm{St}^*)
\longrightarrow 0.
\]

For $m=0,1$, the map
\[
\mathrm{St}^*:H^m(\mathrm A(k))
\longrightarrow H^m(\mathrm A(k-1))
\]
is an isomorphism between infinite cyclic groups, since
$k\geqslant3$. Hence $\Ker(2\,\mathrm{St}^*)=0$ in these
degrees, and the asserted formula follows from the short
exact sequence above and Theorem~\ref{SeqS}.

We may therefore assume that $m\geqslant2$.
In this case, $H^m(\mathrm A(k))$ is finite, and the
short exact sequence above consists of finite groups.

For an abelian group $G$ and a prime $p$, denote by $G_{(p)}$ its $p$-primary component.
The last exact sequence is a direct sum (over all primes $p$) of short exact sequences
\[
0\longrightarrow \bigl(H^{m-1}(\mathrm{A}(k-1))\otimes\mathbb{Z}_2\bigr)_{(p)}
\longrightarrow H^m\bigl(\mathrm{Cn}(k)\bigr)_{(p)}
\xrightarrow{\ \theta\ }
\Ker(2\,\mathrm{St}^*)_{(p)}
\longrightarrow 0.
\]
For $p\neq 2$, the first term vanishes.
We now show that for $p=2$, this short exact sequence splits; since Theorem~\ref{SeqS} implies that
\[
\mathrm{Ker}\bigl(2\,\mathrm{St}^*:H^m(\mathrm{A}(k))\to H^m(\mathrm{A}(k-1))\bigr)\cong
H^m\bigl(\mathrm{A}_2(k)\bigr)\plus
H^m\bigl(\mathrm{A}(k-1)\bigr)[2],
\]
establishing this splitting will complete the proof.

By the splitting lemma for abelian groups (see \cite{Hatcher}, p.147),
it suffices to construct a section of the map $\theta$; that is, a homomorphism
\[
\eta:\mathrm{Ker}(2\,\mathrm{St}^*)_{(2)}\longrightarrow H^m\bigl(\mathrm{Cn}(k)\bigr)_{(2)}
\]
such that $\theta\circ\eta=\mathrm{id}_{\Ker(2\,\mathrm{St}^*)_{(2)}}$.

The group $\Ker(2\,\mathrm{St}^*)_{(2)}$ is a
squarefree-torsion $2$-group, so it can be viewed as a vector space over
$\mathbb{F}_2$. Choose a basis $\{[a_i]\}_{i \in I}$ for
$\Ker(2\,\mathrm{St}^*)_{(2)}$. For each basis element $[a_i]$,
choose a cocycle representative $a_i\in\mathrm{A}^m(k)$.
Consider the class $\mathrm{St}^*[a_i]\in H^m(\mathrm{A}(k-1))_{(2)}$.
Since $2\,\mathrm{St}^*[a_i]=0$, there exists a cochain $b_i\in\mathrm{A}^{m-1}(k-1)$
with $\partial(b_i)=(-1)^m 2\,\mathrm{St}(a_i)$.
For each basis element, set
\[
u_i:=[(a_i,-b_i)]\in H^m(\mathrm{Cn}(k)).
\]
This yields a well-defined cohomology class for each basis element because
\[
\partial(a_i,-b_i)=(\partial(a_i),(-1)^m 2\,\mathrm{St}(a_i)-\partial(b_i))=(0,0).
\]
It remains to verify that these lifts have order dividing $2$,
so that the assignment extends to a homomorphism on
$\Ker(2\,\mathrm{St}^*)_{(2)}$.
Since $2[a_i]=0$, choose
$c_i\in\mathrm A^{m-1}(k)$ with $\partial c_i=2a_i$.
Set
\[
v_i:=(-1)^m \mathrm{St}(c_i)-b_i.
\]
Then $\partial v_i=0$, and
\[
2(a_i,-b_i)-\partial_{\mathrm{Cn}}(c_i,0)=(0,2v_i).
\]
Since $\mathrm{St}^*$ is surjective, the class $2[v_i]$ belongs to
$\operatorname{Im}(2\,\mathrm{St}^*)$ and hence maps to zero in
$H^m(\mathrm{Cn}(k))$ by the cone exact sequence.
Thus $2[(a_i,-b_i)]=0$.
The assignment $[a_i]\mapsto[(a_i,-b_i)]$ therefore
extends to a homomorphism
\[
\eta:\Ker(2\,\mathrm{St}^*)_{(2)}
\longrightarrow H^m(\mathrm{Cn}(k))_{(2)}.
\]
Since $\theta(\eta([a_i]))=[a_i]$ for every basis element,
we have $\theta\circ\eta=\mathrm{id}$.
Thus the short exact sequence splits, completing the proof.
\ep

\bcr\label{HTn}
For $n\geqslant 4$, the following isomorphism holds:
\[
H^q(\mathrm T(n))\cong
\bcs
\left[
\plusn_{i\geqslant1}
H^{q-2i}\bigl(\mathrm A_2(n-2i)\bigr)
\right]&\\[2mm]
\qquad\oplus
\left[
\plusn_{i\geqslant1}
H^{q-2i-1}\bigl(\mathbf B_{n-2i-1}\bigr)\otimes\mathbb Z_2
\right]&\\[2mm]
\qquad\qquad\oplus
\left[
\plusn_{i\geqslant1}
H^{q-2i}\bigl(\mathbf B_{n-2i-1}\bigr)[2]
\right]
&\text{if $q<n$},\\[4mm]
\mathbb{Z}\big/\chi(n-1)\mathbb{Z}&\text{if $q=n$}.
\ecs
\]
\ecr

\bp
For $1\leqslant i<\left\lfloor\frac{n}{2}\right\rfloor$, Lemma~\ref{di} yields an isomorphism of complexes $\mathrm{T}_i(n)\cong\mathrm{Cn}(n-2i)$,
where $\mathrm{T}^q_i(n)\cong\mathrm{Cn}^{q-2i}(n-2i)$, since
\[
\mathrm{T}^q_i(n)=
\mathrm{A}^{q-2i}(n-2i)\times y^{2i}
\plus
\mathrm{A}^{q-2i-1}(n-2i-1)\times y^{2i+1}.
\]
Therefore, Proposition~\ref{HTi00} implies that if $1\leqslant i<\left\lfloor\frac{n}{2}\right\rfloor$, then
\[
H^q\bigl(\mathrm{T}_i(n)\bigr)\cong
H^{q-2i}\bigl(\mathrm{A}_2(n-2i)\bigr)
\plus H^{q-2i-1}(\mathbf{B}_{n-2i-1})\otimes\mathbb{Z}_2
\plus H^{q-2i}(\mathbf{B}_{n-2i-1})[2].
\]
Summing over $i$ and applying Corollary~\ref{Sum} completes the proof.
\ep

Now, we can obtain the final result.
\smallskip

\btr\label{HDn00}
For $n\geqslant 4$, the following isomorphism holds:
\[
H^q(\mathbf{D}_n)\cong
\bcs
H^q(\mathbf{B}_n)\oplus\plusn_{i\geqslant 0}H^{q-2i}(\mathbf{B}_{n-2i})\big/H^{q-2i}(\mathbf{B}_{n-2i-1})\\[2mm]
\hspace{3.5em}\oplus\plusn_{i\geqslant 1}H^{q-2i-1}(\mathbf{B}_{n-2i-1})\otimes\mathbb{Z}_2&\\[2mm]
\hspace{6em}\oplus\plusn_{i\geqslant 1}H^{q-2i}(\mathbf{B}_{n-2i-1})[2]&\text{if $q<n$},\\[3mm]
\mathbb{Z}\big/\chi(n-1)\mathbb{Z}&\text{if $q=n$}.
\ecs
\]
\etr

\bp
By Lemma~\ref{A2}, we know that
\[
H^{q-2i}\bigl(\mathrm{A}_2(n-2i)\bigr)
\cong
H^{q-2i}\bigl(\mathbf{B}_{n-2i}\bigr)\big/ H^{q-2i}\bigl(\mathbf{B}_{n-2i-1}\bigr).
\]
Therefore, Lemma~\ref{d0} and Corollaries~\ref{Sum} and \ref{HTn} give the declared isomorphism.
\ep

\btr\label{StabS_D}
For $n\geqslant4$, there is a canonical morphism of complexes
\[
\mathrm{St}_D:\mathrm{D}(n)\longrightarrow\mathrm{D}(n-1)
\]
such that for any $q\geqslant 1$, the induced homomorphism
\[
\mathrm{St}^*_D:H^q(\mathbf{D}_n)\longrightarrow H^q(\mathbf{D}_{n-1})
\]
is an isomorphism if $n\geqslant 2q$.
Thus, the cohomology of the family $\mathbf{D}$ stabilizes.
\etr

\bp
For $n=4$, define $\mathrm{St}_D$ to be the canonical projection
\[
\mathrm D(4)\cong
\mathrm A(4)\oplus\mathrm A_2(4)\oplus\mathrm T(4)
\longrightarrow \mathrm A(4)=\mathrm D(3).
\]
For $q=1,2$, Corollaries~\ref{A2} and~\ref{HTn} give
$H^q(\mathrm A_2(4))=H^q(\mathrm T(4))=0$.
Thus the assertion holds for $n=4$.
Henceforth assume $n\geqslant5$.

For $n\geqslant 5$, the target complex $\mathrm{D}(n-1)$ contains the subcomplex $\mathrm{T}(n-1)$. In this case, for elements of $\mathrm T(n)$, we set
\[
\mathrm{St}_D(c\times y^r)=\mathrm{St}(c)\times y^r,
\qquad c\in \mathrm A(n-r).
\]
Lemma~\ref{di} implies that $\mathrm{St}_D:\mathrm{T}(n)\to\mathrm{T}(n-1)$ is a morphism of complexes.

Using the canonical decomposition $\mathrm{T}_0(n)\cong\mathrm{A}(n)\oplus\mathrm{A}_2(n)$ from Lemma~\ref{d0}, we define
\[
\mathrm{St}_D(c_1\oplus c_2):=\mathrm{St}(c_1),
\qquad c_1\in\mathrm{A}(n),\quad c_2\in\mathrm{A}_2(n).
\]
The right-hand side is taken in the $\mathrm{A}(n-1)$
summand of the canonical decomposition
$\mathrm{T}_0(n-1)\cong\mathrm{A}(n-1)\oplus\mathrm{A}_2(n-1)$.
Thus $\mathrm{St}_D$ is a morphism of complexes.

It remains to prove that $\mathrm{St}^*_D$ induces isomorphisms on cohomology in the range $n\geqslant 2q$.

\smallskip
\noindent
\textbf{1.} \emph{The subcomplex $\mathrm{T}(n)$.}
For $1\leqslant i<\lfloor n/2\rfloor$, recall from the proof
of Corollary~\ref{HTn} that we have an identification
$\mathrm{T}_i^q(n)\cong\mathrm{Cn}^{q-2i}(n-2i)$,
where $\mathrm{Cn}(n-2i)$ is the mapping cone of the chain map
\[
2\,\mathrm{St}:\mathrm{A}(n-2i)\longrightarrow\mathrm{A}(n-2i-1).
\]
Under this identification, the chain map $\mathrm{St}_D$ acts on the cochain space
\[
\mathrm{Cn}^m(n-2i)=\mathrm{A}^m(n-2i)\oplus\mathrm{A}^{m-1}(n-2i-1)
\]
by applying the operator $\mathrm{St}$ to both cochain factors simultaneously:
\[
\mathrm{St}_{\mathrm{Cn}}(a, b) = \bigl(\mathrm{St}(a), \mathrm{St}(b)\bigr).
\]
Since $\mathrm{St}$ commutes with the differential $\partial$
and strictly satisfies $\mathrm{St}\circ(2\,\mathrm{St})=(2\,\mathrm{St})\circ\mathrm{St}$,
the map $\mathrm{St}_{\mathrm{Cn}}$ is a canonical morphism of mapping cones:
\[
\begin{array}{ccc}
\mathrm{A}(n-2i) & \xrightarrow{\;2\mathrm{St}\;} & \mathrm{A}(n-2i-1) \\
\Big\downarrow\rlap{$\scriptstyle\mathrm{St}$} & & \Big\downarrow\rlap{$\scriptstyle\mathrm{St}$} \\
\mathrm{A}(n-2i-1) & \xrightarrow{\;2\mathrm{St}\;} & \mathrm{A}(n-2i-2)
\end{array}
\]
Put $m=q-2i$ and $k=n-2i$.
Only indices with $2i\leqslant q$ need to be considered,
since $\mathrm T_i(n)$ and $\mathrm T_i(n-1)$ have no
cochains in degrees below $2i$.

For $m\geqslant1$, the assumption $n\geqslant2q$ gives
\[
k-1=n-2i-1\geqslant2m+2i-1\geqslant2m+1.
\]
Hence Theorem~\ref{StabSH} implies that both vertical maps
in the diagram induce isomorphisms in degrees $m-1$ and $m$.
For $m=0$, we have $k\geqslant3$; the maps in degree $0$
are the identity on $\mathbb Z$, and the groups in degree
$-1$ vanish.

The Five Lemma, applied to the induced diagram of long exact
cohomology sequences, therefore gives an isomorphism
\[
H^m(\mathrm{Cn}(k))
\longrightarrow H^m(\mathrm{Cn}(k-1)).
\]
Equivalently,
\[
\mathrm{St}_D^*:H^q(\mathrm T_i(n))
\longrightarrow H^q(\mathrm T_i(n-1))
\]
is an isomorphism. Taking the direct sum over $i$ proves
the assertion for $\mathrm T(n)$.

\smallskip
\noindent
\textbf{2.} \emph{The subcomplex $\mathrm{T}_0(n)$.}
By Lemma~\ref{d0}, we have $H^q(\mathrm{T}_0(n))\cong H^q(\mathrm{A}(n))\oplus H^q(\mathrm{A}_2(n))$.
The map $\mathrm{St}_D^*$ acts as $\mathrm{St}^*$ on the first summand and is identically zero on the second.
By Theorem~\ref{StabSH}, $\mathrm{St}^*:H^q(\mathrm{A}(n))\to H^q(\mathrm{A}(n-1))$ is an isomorphism for $n\geqslant 2q-1$. 
Furthermore, for $n\geqslant 2q$, Corollary~\ref{A2} shows that both $H^q(\mathrm{A}_2(n))=0$ and $H^q(\mathrm{A}_2(n-1))=0$.
Thus, $\mathrm{St}_D^*:H^q(\mathrm{T}_0(n))\to H^q(\mathrm{T}_0(n-1))$ is an isomorphism for $n\geqslant 2q$.
\ep

\brm\label{StabS_D_Small}
In order to define the stabilization morphisms for the affine
families in subsequent sections across all dimensions, we will specify the chain map
$\mathrm{St}_D:\mathrm{D}(m)\longrightarrow\mathrm{D}(m-1)$
for the cases $m \in \{1,2,3\}$. 

For these small values, we define $\mathrm{St}_D$ 
as the algebraic restriction to the standard parabolic subgroup
$D_{m-1}\subset D_m$ (where $D_0$ is the trivial group).
Explicitly, on the basis elements of the Salvetti complex,
$\mathrm{St}_D(P) = P$ if $P \subset D_{m-1}$, and $0$ otherwise.
Because this restriction map inherently preserves the empty
composition $P=\emptyset$, it induces the identity isomorphism on
$H^0(\mathbf{D}_m)\cong\mathbb{Z}$ for all $m \geqslant 1$,
which is the only homological property required from these 
low-dimensional components in the stable ranges.
\erm

\bcr\label{StableD}
The stable cohomology of the family $\mathbf{D}$ is given by the isomorphism:
\[
H^q\bigl(\mathbf{D}_\infty\bigr)\cong H^q(\mathbf{B}_\infty)\oplus
\bigl[\plusn_{i\geqslant 1}H^{q-2i-1}(\mathbf{B}_\infty)\otimes\mathbb{Z}_2\bigr]\oplus
\bigl[\plusn_{i\geqslant 1}H^{q-2i}(\mathbf{B}_\infty)[2]\bigr].
\]
Thus, for any prime $p$, as well as for $p=0$, we have:
\[
\mathrm{G}_p\bigl(\mathbf{D}_\infty;t\bigr)=
\bcs
\mathrm{G}_p(\mathbf{B}_\infty;t)&\text{if $p\neq 2$},\\[1mm]
\frac{1-t+t^2}{1-t}\;\mathrm{G}_2(\mathbf{B}_\infty;t)-\frac{t^2}{1-t}&\text{if $p=2$}.
\ecs
\]
\ecr

\bp
If $n\geqslant 2q$, then by Theorem~\ref{StabSH} the quotient terms
$H^{q-2i}(\mathbf B_{n-2i})/H^{q-2i}(\mathbf B_{n-2i-1})$
vanish for all $i\geqslant0$. Moreover, the stabilization maps identify
\[
H^{q-2i-1}(\mathbf B_{n-2i-1})
\cong
H^{q-2i-1}(\mathbf B_\infty)\qquad\text{and}\qquad H^{q-2i}(\mathbf B_{n-2i-1})[2]
\cong
H^{q-2i}(\mathbf B_\infty)[2]
\]
for all $i\geqslant1$ for which the corresponding summands are non-zero.
Applying these identifications to Theorem~\ref{HDn00} gives the stated
stable decomposition.
\ep

In addition, from Theorem \ref{HDn00} we obtain:
\smallskip

\bcr\label{DQ}
For $n\geqslant 4$, we have
\[
\dim_{\mathbb{Q}} H^q\!\left(\mathbf{D}_n;\mathbb{Q}\right)=
\bcs
1&\text{if $q=0,1$},\\
1&\text{if $q=n-1$ and $n$ is even},\\
1&\text{if $q=n$ and $n$ is even},\\
0&\text{otherwise.}
\ecs
\]
\ecr

\bp
By Arnold's theorem (Corollary~\ref{Arn1}), $H^k(\mathbf{B}_m;\mathbb{Q})\cong\mathbb{Q}$
for $k\in\{0,1\}$ and $m \geqslant 2$, and vanishes for $k \geqslant 2$. 
Applying the functor $-\otimes \mathbb{Q}$ to the isomorphism in Theorem~\ref{HDn00}
eliminates all $\mathbb{Z}_2$-torsion terms.

For $q=n$, the term $\mathbb{Z}\big/\chi(n-1)\mathbb{Z}\otimes\mathbb{Q}$
yields $\mathbb{Q}$ if $n$ is even (since $\chi(n-1)=0$) and $0$ if $n$ is odd.
For $q < n$, the base term $H^q(\mathbf{B}_n;\mathbb{Q})$ yields $\mathbb{Q}$ for $q=0, 1$. 

The remaining rational contribution comes from the quotient terms
\[
H^{q-2i}(\mathbf{B}_{n-2i};\mathbb{Q}) \big/ H^{q-2i}(\mathbf{B}_{n-2i-1};\mathbb{Q}).
\]
For this quotient to be non-zero (i.e., isomorphic to $\mathbb{Q}$),
we must have $H^{q-2i}(\mathbf{B}_{n-2i};\mathbb{Q})\cong\mathbb{Q}$
and $H^{q-2i}(\mathbf{B}_{n-2i-1};\mathbb{Q})=0$.
Since rational cohomology of braid groups is non-trivial only in dimensions 0 and 1,
and the dimension 0 yields $\mathbb{Q}/\mathbb{Q}=0$, we must strictly have $q-2i=1$. 
Furthermore, $H^1(\mathbf{B}_m;\mathbb{Q})$ vanishes only for $m < 2$.
Thus, we need $n-2i \geqslant 2$ and $n-2i-1 < 2$, which uniquely forces $n-2i=2$.
This relation implies $2i=n-2$, meaning $n$ must be even. 
Consequently, the degree is $q=1+2i=n-1$. In this case, exactly one quotient term survives, yielding $\mathbb{Q}$.
\ep

\brm\label{RmGor}
The integral cohomology of the groups $\mathbf D_n$ was previously
studied by Goryunov~\cite{Gor} using topological methods.
For $q < n$, he established the isomorphism
\[
H^q(\mathbf{D}_n)\cong
H^q(\mathbf{B}_n)\oplus\plusn_{i\geqslant 0}H^{q-2i}(\mathbf{B}_{n-2i})\big/H^{q-2i}(\mathbf{B}_{n-2i-1})
\oplus\plusn_{j\geqslant 1}H^{q-2j-1}(\mathbf{B}_{n-2j-1};\mathbb{Z}_2).
\]
By the Universal Coefficient Theorem, there is a
noncanonical isomorphism
\[
H^k(X;\mathbb Z_2)\cong
\bigl(H^k(X;\mathbb Z)\otimes\mathbb Z_2\bigr)
\oplus H^{k+1}(X;\mathbb Z)[2].
\]
Applying this decomposition to Goryunov's formula directly yields
the right-hand side of our Theorem~\ref{HDn00} in the range $q < n$. 

In the top degree $q=n$, our Theorem~\ref{HDn00}
gives $\mathbb Z/\chi(n-1)\mathbb Z$, in agreement
with the entries in Goryunov's table given in \cite{Gor}.
Thus, in the range $q<n$, our formula is equivalent
to Goryunov's by the Universal Coefficient Theorem,
while our argument treats the top degree explicitly
within the same algebraic framework.
\erm

\section{Cohomology of the Groups
\texorpdfstring{$\widetilde{\mathbf{B}}_n$}{\textasciitilde B_n}}
\label{BBDn}

In this section, we describe the cohomology of the complex
$\widetilde{\mathrm{B}}(n)$, where $n\geqslant 4$.

For an integer $l\geqslant 0$, we regard the subgroup generated by the 
elements $\{s_{l+2},s_{l+3},\dots,s_n\}$ in $\widetilde{B}_n$ 
as the Coxeter group $D_{n-l-1}\subset\widetilde{B}_n$.

For an integer $l$ with $1\leqslant l\leqslant n-1$, set
$L_l:=\langle s_1,\dots,s_l\rangle\subset\widetilde{B}_n$.
For a parabolic subgroup $P\subset\widetilde{B}_n$, let $h(P)$ be the maximal integer
$l\geqslant 1$ such that $L_l\subset P$.
If no such subgroup $L_l$ exists, set $h(P)=0$.

For the submodule of $\widetilde{\mathrm{B}}(n)$ generated by parabolic subgroups
$P\subset\widetilde{B}_n$ with $h(P)=l$, we use the identification $x^l\times\mathrm{D}(n-l-1)$.
Then
\beq\label{SBm}
\widetilde{\mathrm{B}}(n)=\plusn_{0\leqslant l\leqslant n-1} x^l\times\mathrm{D}(n-l-1).
\eeq
(Here and below we use the convention that $\mathbf D_0$ is the
trivial group and that $\mathrm{D}(0)=\mathbb Z$, concentrated in
degree $0$.)
Since the basis of the complex $\widetilde{\mathrm{B}}(n)$
includes only proper subgroups of $\widetilde{B}_n$, 
formula \eqref{CC} shows that for $m=n-l-1$ and $c\in\mathrm{D}_r(m)$, we have:
\[
\partial_{\widetilde{B}}(x^l\times c)=
\bcs
(-1)^lx^l\times\partial_D(c)&\text{if $r<m$},\\
0&\text{if $r=m$}.
\ecs\]
This shows that \eqref{SBm} is a direct sum of complexes.\footnote{The factor $(-1)^l$
does not affect cohomology, by the degree-wise automorphism described in Section \ref{Cn}.}
Thus, we obtain the following result:
\smallskip

\btr\label{BStab}
For $n\geqslant4$, there is an isomorphism
\[
H^q\bigl(\widetilde{\mathbf B}_n\bigr)
\cong
\plusn_{0\leqslant l\leqslant n-1}
H^{q-l}\bigl(\mathbf D_{n-l-1}\bigr).
\]
\etr

\btr\label{StabS_B}
For $n\geqslant 5$, there is a canonical morphism of complexes
\[
\mathrm{St}_{\widetilde B}:\widetilde{\mathrm B}(n)
\longrightarrow
\widetilde{\mathrm B}(n-1)
\]
such that the induced homomorphism
$
\mathrm{St}_{\widetilde B}^*:
H^q(\widetilde{\mathbf B}_n)
\longrightarrow
H^q(\widetilde{\mathbf B}_{n-1})
$\
is an isomorphism if $n\geqslant\max\{q+4,2q+1\}$.
\etr

\bp
On the summand $x^l\times\mathrm{D}(n-l-1)$, set
$\mathrm{St}_{\widetilde B}(x^l\times c)=x^l\times \mathrm{St}_D(c)$.
The last summand $x^{n-1}\times\mathrm{D}(0)$ is sent to zero.

For a non-zero contribution to $H^q$, put
$m=n-l-1$ and $d=q-l$.
Under the stated assumptions, whenever $d>0$ we have $m\geqslant4$.
Moreover, if $n\geqslant2q+1$, then
\[
m=n-l-1\geqslant 2q-l\geqslant 2(q-l)=2d.
\]
If $d=0$, the assertion is clear. If $d>0$, then
Theorem~\ref{StabS_D} implies that
$H^d(\mathbf D_m)\longrightarrow H^d(\mathbf D_{m-1})$
is an isomorphism. Summing over $l$ gives the claim.
\ep

\bcr\label{StabB}
For any $q\geqslant3$ and $n\geqslant2q$,\footnote{The shift by one in the stated stable range comes from applying
Theorem~\ref{StabS_B} to the stabilization map at level $n+1$.} as well as for
$q=2$ and $n\geqslant5$, the cohomology group
$H^q(\widetilde{\mathbf B}_n)$ stabilizes. The stable cohomology is
given by
\[
H^q\bigl(\widetilde{\mathbf{B}}_\infty\bigr)\cong
\plusn_{m\geqslant 0}H^{q-m}\bigl(\mathbf{D}_\infty\bigr). 
\]
Thus, for any prime $p$, as well as for $p=0$, we have:
\beq\label{GB}
\mathrm{G}_p\bigl(\widetilde{\mathbf{B}}_\infty;t\bigr)=
\bcs
\frac{1}{1-t}\;\mathrm{G}_p(\mathbf{B}_\infty;t)&\text{if $p\neq 2$},\\[2mm]
\frac{1-t+t^2}{(1-t)^2}\;\mathrm{G}_2(\mathbf{B}_\infty;t)-\frac{t^2}{(1-t)^2}&\text{if $p=2$}.
\ecs
\eeq
\ecr

Applying Corollary~\ref{DQ} to Theorem \ref{BStab} and taking into account our conventions 
$\mathbf{D}_1:=\mathbf{A}_1$, $\mathbf{D}_2:=\mathbf{A}_1\times \mathbf{A}_1$, 
and $\mathbf{D}_3:=\mathbf{A}_3$ yields the following statement:
\smallskip

\bcr[\cite{CMS}]
For $n\geqslant 4$, we have:
\[
\dim_{\mathbb{Q}} H^q\bigl(\widetilde{\mathbf{B}}_n;\mathbb{Q}\bigr)=
\bcs
1&\text{if $q=0$},\\[1mm]
2&\text{if $1\leqslant q<n-2$},\\[1mm]
\left\lfloor\frac{n+3}{2}\right\rfloor&\text{if $q=n-2$ or $q=n-1$},\\
0&\text{otherwise}.
\ecs
\]
\ecr 

\section{Cohomology of the Groups 
\texorpdfstring{$\widetilde{\mathbf{D}}_n$}{\textasciitilde D_n}}\label{BBDDn}

In this section, we describe the cohomology of the complex
$\widetilde{\mathrm{D}}(n)$, where $n\geqslant 5$.

The assignments
\begin{gather*}
\sigma(s_i)=s_i\quad\text{if $i\leqslant n-2$},\quad\sigma(s_{n-1})=s_n,\;\sigma(s_n)=s_{n-1},\quad
\tau(s_i)=s_{n-i+1}
\end{gather*}
extend to involutions $\sigma$ and $\tau$ of the complex
$\widetilde{\mathrm D}(n)$.
Set $\overline{\sigma}:=\tau\circ\sigma\circ\tau$.
The involutions $\sigma$ and $\overline{\sigma}$ commute.

We begin by showing that the complex $\widetilde{\mathrm{D}}(n)$
splits as a direct sum of two subcomplexes:
\beq\label{K1K2}
\widetilde{\mathrm{D}}(n)=\mathrm{K}_1(n)\oplus\mathrm{K}_2(n).
\eeq

\subsection*{Complex $\mathrm{K}_1(n)$}
We identify the subgroup
$\langle s_2,\dots,s_n\rangle\subset \widetilde{D}_n$
with the Coxeter group $D_{n-1}$, and regard the corresponding complex
$\mathrm{D}(n-1)$ as a submodule of $\widetilde{\mathrm{D}}(n)$.
The definition of the differential shows that
$\mathrm D(n-1)+\overline\sigma(\mathrm D(n-1))$
and its image under $\tau$ are subcomplexes.

Define $\mathrm{K}_1(n)$ as the sum of subcomplexes:
\[
\mathrm{K}_1(n):=\bigl[\mathrm{D}(n-1)+\overline{\sigma}\bigl(\mathrm{D}(n-1)\bigr)\bigr]+
\tau\bigl[\mathrm{D}(n-1)+\overline{\sigma}\bigl(\mathrm{D}(n-1)\bigr)\bigr].
\]
According to Section~\ref{ADn}, there is a decomposition \eqref{DecD}
\[
\mathrm{D}(n-1)=\mathrm{T}_0(n-1)\oplus\mathrm{T}(n-1)
\]
into a direct sum of subcomplexes of $\mathrm{D}(n-1)$.
Therefore, the modules
\[
\overline{\mathrm{T}}_0(n-1):=\mathrm{T}_0(n-1)+\overline{\sigma}\bigl(\mathrm{T}_0(n-1)\bigr),\qquad
\overline{\mathrm{T}}(n-1):=\mathrm{T}(n-1)+\overline{\sigma}\bigl(\mathrm{T}(n-1)\bigr)
\]
are subcomplexes of $\widetilde{\mathrm{D}}(n)$, and
\[
\mathrm{K}_1(n)=
\bigl[\,\overline{\mathrm{T}}_0(n-1)\oplus\overline{\mathrm{T}}(n-1)\bigr]+
\tau\bigl[\,\overline{\mathrm{T}}_0(n-1)\oplus\overline{\mathrm{T}}(n-1)\bigr].
\]
It is clear that
\[
\overline{\mathrm{T}}_0(n-1)=\tau\bigl(\overline{\mathrm{T}}_0(n-1)\bigr)
\qquad\text{and}\qquad
\overline{\mathrm{T}}(n-1)\cap\tau\bigl(\overline{\mathrm{T}}(n-1)\bigr)=\{0\}.
\]
Thus, we obtain the decomposition of $\mathrm{K}_1(n)$ into a direct sum of complexes:
\beq\label{K}
\mathrm{K}_1(n)\cong\overline{\mathrm{T}}_0(n-1)
\oplus
\bigl[\overline{\mathrm{T}}(n-1)\bigr]^{\oplus 2}.
\eeq

\subsection*{Complex $\mathrm{K}_2(n)$}
For integers $l$ and $r$ such that $2\leqslant l,r\leqslant n-1$, set
\[
\overline{P}_l:=\langle s_1,\dots,s_l\rangle
\qquad\text{and}\qquad
P_r:=\langle s_{n-r+1},\dots,s_n\rangle.
\]
For a parabolic subgroup $P\subset\widetilde{D}_n$, let $\overline{h}(P)$ and $h(P)$
denote the maximal integers $l$ and $r$, respectively, such that
\[
\overline{P}_l\subseteq P\quad\text{and}\quad P_r\subseteq P.
\]

Let $\mathrm{D}_{l,r}(n)$ denote the submodule of $\widetilde{\mathrm{D}}(n)$ generated by the
parabolic subgroups $P\subset\widetilde{D}_n$ with $\overline{h}(P)=l$ and $h(P)=r$.
Define
\[
\mathrm{K}_2(n):=\plus_{l\geqslant 2,\;r\geqslant 2,\;l+r\leqslant n-1}
x^l\times\mathrm{A}(n-l-r-1)\times y^r.
\]
\emph{Let $\overline{\mathrm{St}}:\mathrm{A}(m)\longrightarrow\mathrm{A}(m-1)$ be a morphism of complexes
defined by $\overline{\mathrm{St}}:=\tau\circ \mathrm{St}\circ\tau$.}

Using arguments analogous to those used in the proofs of Lemmas \ref{d00} and \ref{di},
we see that
\begin{multline}\label{BoundK2}
\partial_{\widetilde{D}}\bigl(x^l\times c\times y^r\bigr)\\
=
\begin{cases}
\chi(l)\,x^{l+1}\times\overline{\mathrm{St}}(c)\times y^r
+(-1)^l x^l\times \partial(c)\times y^r \\[2pt]
\qquad\qquad\qquad +\;
(-1)^{l+\deg(c)}\chi(r)\,x^l\times \mathrm{St}(c)\times y^{r+1},
& \text{if $l,r\geqslant 2$ and $l+r<n-1$,}\\[4pt]
0 & \text{otherwise},
\end{cases}
\end{multline}
where $c\in\mathrm{A}(n-l-r-1)$.

Thus, $\mathrm{K}_2(n)$ is a subcomplex of $\widetilde{\mathrm{D}}(n)$.
Since $\widetilde{\mathrm{D}}(n)=\mathrm{K}_1(n)+\mathrm{K}_2(n)$ and
$\mathrm{K}_1(n)\cap\mathrm{K}_2(n)=\{0\}$, the decomposition \eqref{K1K2} follows.
\smallskip

\subsection*{Calculation of $H^*(\mathrm{K}_1(n))$}
In view of the isomorphism~\eqref{K}, we first describe
$H^*\bigl(\overline{\mathrm{T}}(n-1)\bigr)$.

Let $\overline{\mathrm{A}}_1(n-1)$ and $\overline{\mathrm{A}}_2(n-1)$
denote the submodules of $\mathrm{A}(n-1)$ generated by the cochains $e(n_1,n_2,\dots,n_a)$
such that $n_1=1$ and $n_1\geqslant 2$ respectively.
For $m=0$, set
$\overline{\mathrm A}_1(0)=0$ and
$\overline{\mathrm A}_2(0)=\mathrm A(0)$.
Define
\[
\mathrm{T}^{(1)}(n-1):=\plusn_{r\geqslant 2}\overline{\mathrm{A}}_1(n-r-1)\times y^r,
\qquad\mathrm{T}^{(2)}(n-1):=\plusn_{r\geqslant 2}\overline{\mathrm{A}}_2(n-r-1)\times y^r.
\]
Obviously,
\[
\mathrm{T}(n-1)=\mathrm{T}^{(1)}(n-1)\plus\mathrm{T}^{(2)}(n-1).
\]

\blr\label{D00}
For $n\geqslant 5$, the action of $\partial_{\widetilde{D}}$ on
$\overline{\mathrm{T}}(n-1)$ is determined by the commutation relation
$\partial_{\widetilde{D}}\circ\overline{\sigma}=\overline{\sigma}\circ\partial_{\widetilde{D}}$ and the formula
\[
\partial_{\widetilde{D}}(c)=
\bcs
\partial_D(c)\dotPlus\overline{\sigma}\bigl(\partial_D(c)\bigr)&\text{if $c\in\mathrm{T}^{(1)}(n-1)$},\\
\partial_D(c)&\text{if $c\in\mathrm{T}^{(2)}(n-1)$}.
\ecs
\]
In particular, $\mathrm{T}^{(2)}(n-1)$ is a subcomplex of\, $\overline{\mathrm{T}}(n-1)$. 
\elr

\bp
For $c\in\mathrm{T}^{(1)}(n-1)$ the formula follows from the definition of $\partial_{\widetilde{D}}$.

For a basis cochain $c=(-1)^{\frac{a(a+1)}{2}}e(n_1,\dots,n_a)\times y^r\in\mathrm{T}^{(2)}(n-1)$
the definition of $\partial_{\widetilde{D}}$ yields:
\[
\partial_{\widetilde{D}}(c)=
\left.\frac{D_{n_1}(t)}{A_{n_1-1}(t)}\right|_{t=-1}
e(n_2,\dots,n_a)\times y^r+\partial_D(c)
=\partial_D(c).
\]
This proves the formula for $c\in\mathrm{T}^{(2)}(n-1)$.
\ep

\blr\label{Dd0}
For $n\geqslant 5$, there is a canonical isomorphism of complexes
\[
\overline{\mathrm{T}}(n-1)
\cong
\mathrm{T}(n-1)\oplus\mathrm{T}^{(2)}(n-1).
\]
\elr

\bp
Define
\[
\Delta\bigl(\mathrm{T}^{(2)}(n-1)\bigr):=\{\,c+\overline{\sigma}(c)\mid c\in\mathrm{T}^{(2)}(n-1)\,\}.
\]
Then
\begin{multline*}
\qquad\overline{\mathrm{T}}(n-1)
=\mathrm{T}^{(2)}(n-1)\oplus\mathrm{T}^{(1)}(n-1)\oplus\overline{\sigma}\bigl(\mathrm{T}^{(2)}(n-1)\bigr)\\
=\mathrm{T}^{(1)}(n-1)\oplus\Delta\bigl(\mathrm{T}^{(2)}(n-1)\bigr)\oplus\mathrm{T}^{(2)}(n-1).\qquad
\end{multline*}
The formula for $\partial_{\widetilde{D}}$ given in Lemma~\ref{D00}
shows that $\mathrm{T}^{(1)}(n-1)\oplus\Delta\bigl(\mathrm{T}^{(2)}(n-1)\bigr)$ is a subcomplex of
$\overline{\mathrm{T}}(n-1)$ isomorphic to $\mathrm{T}(n-1)$.
Therefore, this decomposition of $\overline{\mathrm{T}}(n-1)$ yields the claimed isomorphism.
\ep

\blr\label{HTT}
For $n\geqslant 5$, there is an isomorphism
\[
H^q\bigl(\mathrm{T}^{(2)}(n-1)\bigr)\cong\plusn_{r\geqslant 2} H^{q-r}\bigl(\mathrm{A}_2(n-r-1)\bigr).
\]
\elr

\bp
For every $r\geqslant2$, put $m=n-r-1$ and consider the column
\[
C_r:=\overline{\mathrm A}_2(m)\times y^r
\]
with the part of the differential which preserves the exponent of
$y$. This column complex is just the complex
$\overline{\mathrm A}_2(m)$, shifted in degree by $r$.

Define
\[
\Phi_r:C_r\longrightarrow \mathrm A_2(m)\times y^r\qquad\text{by}
\qquad
\Phi_r(c\times y^r):=\tau(c)\times y^r .
\]
Since
\[
\overline{\mathrm A}_2(m)=
\Ker\bigl(\overline{\mathrm{St}}:\mathrm{A}(m)\to\mathrm{A}(m-1)\bigr),
\qquad
\overline{\mathrm{St}}=\tau\mathrm{St}\tau ,
\]
we have
$\overline{\mathrm{St}}(c)=0\Longleftrightarrow\mathrm{St}(\tau c)=0$.
Thus $\Phi_r$ maps $C_r$ isomorphically onto
$\mathrm A_2(m)\times y^r$. Since $\tau$ commutes with the
differential of $\mathrm A(m)$, $\Phi_r$ is an isomorphism of
column complexes. Hence
\[
H^q(C_r)
\cong
H^{q-r}\bigl(\mathrm A_2(n-r-1)\bigr).
\]

It remains to take into account the part of the total differential
which changes the exponent of $y$. This part contains the factor
$\chi(r)$, and therefore it is zero for odd $r$. Hence
$\mathrm T^{(2)}(n-1)$ decomposes into the direct sum, over
$a\geqslant1$, of two-column complexes
\[
K_a:\;
C_{2a}\longrightarrow C_{2a+1}.
\]
Here
\[
C_{2a}=\overline{\mathrm A}_2(n-2a-1)\times y^{2a},
\qquad
C_{2a+1}=\overline{\mathrm A}_2(n-2a-2)\times y^{2a+1}.
\]

Since the weights $n-2a-1$ and $n-2a-2$ have opposite parities, and
$H^*(\mathrm A_2(k))=0$ for odd $k$, one of the two column
complexes $C_{2a}$, $C_{2a+1}$ is acyclic. Therefore the short exact
sequence
\[
0\longrightarrow C_{2a+1}\longrightarrow K_a\longrightarrow C_{2a}
\longrightarrow0
\]
gives, by the associated long exact cohomology sequence,
\[
H^q(K_a)
\cong
H^q(C_{2a})\oplus H^q(C_{2a+1}).
\]
This is canonical, because one of the two summands is zero.

Taking the direct sum over all $a\geqslant1$, we obtain the required isomorphism.
\ep

To finish the calculation of  $H^*(\mathrm{K}_1(n))$ it remains
to find the cohomology of $\overline{\mathrm{T}}_0(n-1)$.

\blr\label{HDDT0}
For $n\geqslant 5$, the following isomorphism holds:
\begin{multline*}
\qquad H^q\bigl(\overline{\mathrm T}_0(n-1)\bigr)
\cong
H^q\bigl(\mathrm T_0(n-1)\bigr)
\plus
H^q\bigl(\mathrm A_2(n-1)\bigr)\\
\plus
\frac{\chi(n-1)}{2}\,
H^q\bigl(\mathrm A_2(n-1)\bigr)
\plus
\frac{\chi(n)}{2}\,
H^{q-1}\bigl(\mathrm A_2(n-2)\bigr).\qquad
\end{multline*}
\elr

\bp
Similarly to Lemma~\ref{Dd0}, we obtain a canonical isomorphism of complexes
\[
\overline{\mathrm{T}}_0(n-1)\cong\mathrm{T}_0(n-1)\oplus\mathrm{T}^{(0)}_0(n-1),
\qquad\text{where}\qquad
\mathrm{T}^{(0)}_0(n-1):=\overline{\mathrm{A}}_2(n-1)+\sigma\bigl(\overline{\mathrm{A}}_2(n-1)\bigr)
\]
and, similarly to Lemma~\ref{d0}, we obtain a canonical isomorphism of complexes
\[
\mathrm{T}^{(0)}_0(n-1)\cong\overline{\mathrm{A}}_2(n-1)\oplus\mathrm{A}_{2,2}(n-1),
\]
where $\mathrm{A}_{2,2}(n-1)=\Ker(\mathrm{St})\cap\Ker(\overline{\mathrm{St}})$.

It remains to compute $H^*(\mathrm{A}_{2,2}(n-1))$.
The morphism of complexes $\overline{\mathrm{St}}:\mathrm{A}_2(m)\to\mathrm{A}_2(m-1)$
is surjective and induces a short exact sequence of complexes
\[
0\longrightarrow \mathrm{A}_{2,2}(m)
\longrightarrow \mathrm{A}_2(m)\xrightarrow{\;\overline{\mathrm{St}}\;}\mathrm{A}_2(m-1)\longrightarrow 0.
\]
This yields a long exact sequence in cohomology.
Since $H^*(\mathrm{A}_2(k))$ vanishes when $k$ is odd, this sequence
shows that
\[
H^q\bigl(\mathrm{A}_{2,2}(m)\bigr)\cong
\bcs 
H^q\bigl(\mathrm{A}_2(m)\bigr)&\text{if $m$ is even},\\
H^{q-1}\bigl(\mathrm{A}_2(m-1)\bigr)&\text{if $m$ is odd}.
\ecs
\]
This isomorphism, for $m=n-1$, implies the isomorphism of the lemma.
\ep

\btr\label{HK1}
For $n\geqslant 5$, the following isomorphism holds:

If $q<n-1$, then
\begin{align*}
H^q\bigl(\mathrm{K}_1(n)\bigr) &\cong H^q\bigl(\mathbf{D}_{n-1}\bigr)\plus
\bigl[H^q(\mathbf{B}_{n-1})\big/H^q(\mathbf{B}_{n-2})\bigl]
\\[2mm]
&\hspace{2em}\plus\frac{\chi(n-1)}{2}H^q(\mathbf{B}_{n-1})\big/H^q(\mathbf{B}_{n-2})
\\[2mm]
&\hspace{2em}\plus\frac{\chi(n)}{2}H^{q-1}(\mathbf{B}_{n-2})\big/H^{q-1}(\mathbf{B}_{n-3})
\\[2mm]
&\hspace{2em}\plusn_{i\geqslant 1}\bigl[H^{q-2i}(\mathbf{B}_{n-2i-1})\big/H^{q-2i}(\mathbf{B}_{n-2i-2})\bigr]
\\[2mm]
&\hspace{2em}\plusn_{i\geqslant 1} \bigl[H^{q-2i-1}(\mathbf{B}_{n-2i-2})\otimes\mathbb{Z}_2
\oplus H^{q-2i}(\mathbf{B}_{n-2i-2})[2]\bigr]
\\[2mm]
&\hspace{2em}\plusn_{i\geqslant 2}
\bigl[H^{q-i}(\mathbf{B}_{n-i-1})\big/H^{q-i}(\mathbf{B}_{n-i-2})\bigr]^{\oplus 2}.
\\[4mm]
H^{n-1}\bigl(\mathrm{K}_1(n)\bigr)&\cong\mathbb{Z}^{\oplus 2}\plusn\bigl[\mathbb{Z}\big/\chi(n)\mathbb{Z}\bigr]^{\oplus 2}.
\end{align*}
\etr

\bp
From the isomorphism~\eqref{K}, we assemble the cohomology $H^q\bigl(\mathrm{K}_1(n)\bigr)$
by summing the components of $\overline{\mathrm{T}}_0(n-1)$ and $\bigl[\overline{\mathrm{T}}(n-1)\bigr]^{\oplus 2}$
obtained in Lemmas~\ref{Dd0}, \ref{HTT}, and \ref{HDDT0}.
The base terms $\mathrm{T}_0(n-1)$ and one copy of $\mathrm{T}(n-1)$
exactly reconstruct the cohomology of the group $\mathbf{D}_{n-1}$
via the decomposition~\eqref{DecD}. 

For $q < n-1$, the explicit sums in the theorem correspond directly to the remaining components. 
Applying Corollary~\ref{HTn} to the second copy of $\mathrm{T}(n-1)$ yields the sums over $i \geqslant 1$.
The residual terms from $\overline{\mathrm{T}}_0(n-1)$ (Lemma~\ref{HDDT0})
give the terms involving $\mathrm A_2(n-1)$ and $\mathrm A_2(n-2)$.
Finally, Lemma~\ref{HTT} applied to the doubled complex
$\bigl[\mathrm{T}^{(2)}(n-1)\bigr]^{\oplus 2}$ yields the last sum 
$\plusn_{i\geqslant 2} H^{q-i}\bigl(\mathrm A_2(n-i-1)\bigr)^{\oplus 2}$ (substituting $r=i$).

At the top degree $q=n-1$, we evaluate each of these components using Corollary~\ref{A2}. 
Since $H^m(\mathrm A_2(m))\cong H^{(m)}_0(\G[z])$, it is non-zero only for $m=0$, where $H^0(\mathrm A_2(0))\cong\mathbb{Z}$.
\begin{itemize}
    \item[--] The base $H^{n-1}(\mathbf{D}_{n-1})$ contributes $\mathbb{Z}\big/\chi(n)\mathbb{Z}$.
    \item[--] By Corollary~\ref{HTn}, the second copy of $H^{n-1}(\mathrm{T}(n-1))$
    also contributes exactly $\mathbb{Z}\big/\chi(n)\mathbb{Z}$.
    \item[--] In the residual part of $\overline{\mathrm{T}}_0(n-1)$,
    the terms $H^{n-1}(\mathrm A_2(n-1))$ and $H^{n-2}(\mathrm A_2(n-2))$
    vanish because $n \geqslant 5$ (thus $n-1>0$ and $n-2>0$).
    \item[--] The doubled complex $\bigl[\mathrm{T}^{(2)}(n-1)\bigr]^{\oplus 2}$
    contributes $\plusn_{r\geqslant 2} H^{n-1-r}(\mathrm{A}_2(n-r-1))^{\oplus 2}$.
    Letting $m=n-r-1$, we see this is non-zero only when $m=0$, which implies $r=n-1$.
    Since $n \geqslant 5$, the index $r=n-1 \geqslant 4\geqslant 2$ is included in the sum.
    This term yields $H^0(\mathrm{A}_2(0))^{\oplus 2}\cong\mathbb{Z}^{\oplus 2}$.
\end{itemize}
Summing these non-vanishing contributions gives exactly
$\mathbb{Z}^{\oplus 2}\plusn\bigl[\mathbb{Z}\big/\chi(n)\mathbb{Z}\bigr]^{\oplus 2}$,
completing the proof.
\ep

\subsection*{Calculation of $H^*(\mathrm{K}_2(n))$}
Define a decreasing filtration of the complex\/ $\mathrm{K}_2(n)$ by
\[
\mathrm{K}_2(n)=F_4(n)\supset\cdots\supset F_p(n)\supset F_{p+1}(n)
\supset\cdots\supset F_{n-1}(n)\supset\{0\},
\]
where
\[
F_p(n):=\plus_{l,r\geqslant2,\;p\leqslant l+r\leqslant n-1}
x^l\times\mathrm{A}(n-l-r-1)\times y^r.
\]
Since this filtration is finite, the associated spectral
sequence converges to the cohomology of $\mathrm K_2(n)$,
with $E_\infty$ identifying its associated graded group.

As follows from the definition of the differential of $\mathrm{K}_2(n)$,
the first page of this spectral sequence is the complex
\beq\label{SE1}
E^{p,j}_1(n)\cong\plus_{l\geqslant 2,\,r\geqslant 2,\,p=l+r}
x^l\times H^j\bigl(\mathrm{A}(n-p-1)\bigr)\times y^r,
\eeq
with differential $d_1:E^{p,j}_1(n)\to E^{p+1,j}_1(n)$ acting by
\[
\qquad d_1\bigl(x^l\times h\times y^r\bigr)=
\chi(l)\,x^{l+1}\times\overline{\mathrm{St}}^*(h)\times y^r
+(-1)^{l+\deg(h)}\chi(r)\,x^l\times \mathrm{St}^*(h)\times y^{r+1}.
\]

Set, for brevity,
\[
H^j_{l,r}(n):=x^l\times H^j\bigl(\mathrm{A}(n-l-r-1)\bigr)\times y^r.
\]
The differential $d_1$ maps $H^j_{l,r}(n)\to H^j_{l+1,r}(n)\plus H^j_{l,r+1}(n)$.

\emph{To compute the $E_2(n)$ page, observe the parities of $l$ and $r$.}
Because $\chi(m)$ evaluates to $2$ for even $m$ and $0$ for odd $m$,
the differential $d_1$ vanishes on components where both $l$ and $r$ are odd.
Consequently, the $E_1(n)$ page decomposes into a direct sum of independent $3$-term
subcomplexes for each pair of positive even integers $(2a,2b)$
with $2a+2b\leqslant n-1$:
\beq\label{3Complex}
0\longrightarrow H^j_{2a,2b}(n)\xrightarrow{\;d_1\;}
H^j_{2a+1,2b}(n)\plus H^j_{2a,2b+1}(n)\xrightarrow{\;d_1\;}H^j_{2a+1,2b+1}(n)\longrightarrow 0.
\eeq

The same parity argument gives a decomposition of the filtered
complex $\mathrm K_2(n)$ itself into subcomplexes whose
$E_1$-pages are the blocks \eqref{3Complex}.
Each subcomplex occupies at most three consecutive
filtration degrees.

\blr\label{LmD2Zero}
The spectral sequence associated with the filtration $F_p(n)$
of $\mathrm K_2(n)$ collapses at $E_2$.
\elr

\bp
The block decomposition implies that $d_r=0$ for
$r\geqslant3$. Within each block, the only potentially
nonzero $d_2$ goes from its even-even corner to its
odd-odd corner.

Fix a block and put $m=n-2a-2b-1$.
If $m<2$, its odd-odd corner is zero.
For $m\geqslant2$, surjectivity of $\mathrm{St}^*$ identifies the target
of this potential differential with
\[
E_2^{2a+2b+2,j-1}\big|_{\mathrm{block}}
\cong
H^{j-1}\bigl(\mathrm A(m-2)\bigr)
\big/
2H^{j-1}\bigl(\mathrm A(m-2)\bigr).
\]
In particular, this target is annihilated by $2$.

Suppose first that $j\geqslant3$.
The source is a subgroup of the finite squarefree-torsion
group $H^j(\mathrm A(m))$.
Its odd-primary components map trivially to the target,
so it remains to consider its $2$-primary component.
Represent a class in this component by
\[
u=x^{2a}\times c\times y^{2b},
\qquad \partial c=0.
\]
Since $2[c]=0$, choose $w\in\mathrm A^{j-1}(m)$ with
$\partial w=2c$, and set
\[
t=x^{2a}\times w\times y^{2b}.
\]
Write the total differential as
\[
\partial_{\widetilde D}=d_v+2h,
\]
where $d_v$ preserves the filtration degree and $h$ raises
it by one. The identity $\partial_{\widetilde D}^{\,2}=0$,
together with freeness of the cochain groups, gives
\[
d_vh+hd_v=0,\qquad h^2=0.
\]
Since $d_vt=2u$, the cochain $v=ht$ satisfies
\[
d_vv=-h(d_vt)=-2hu,
\qquad
2hv=2h^2t=0.
\]
Thus $u+v$ is a total cocycle lifting $u$, and $d_2[u]=0$.

It remains to consider $j\leqslant2$.
For $j=0$, the target has negative cohomological degree
and is zero. For $j=2$, the source is zero because
$H^2(\mathrm A(m))=0$.
For $j=1$, the source at the even-even corner vanishes
unless $m=2$: for $m\geqslant3$, the map
\[
\mathrm{St}^*:H^1(\mathrm A(m))\longrightarrow
H^1(\mathrm A(m-1))
\]
is an isomorphism between infinite cyclic groups.
When $m=2$, the source is generated by $[e(2)]$.
Both $\mathrm{St}$ and $\overline{\mathrm{St}}$ annihilate $e(2)$, so this
generator is already represented by a total cocycle.
Hence $d_2=0$ also in this case.
Therefore all differentials $d_r$, $r\geqslant2$, vanish.
\ep

\blr\label{LmExH}
For $m=n-2a-2b-1$, the complex \eqref{3Complex} is isomorphic to the complex
\beq\label{ExH}
0\longrightarrow H^j\bigl(\mathbf{B}_m\bigr)
\xrightarrow{(2\,\mathrm{St}^*,2\,\mathrm{St}^*)} H^j\bigl(\mathbf{B}_{m-1}\bigr)^{\oplus 2}
\xrightarrow{(2\,\mathrm{St}^*,-2\,\mathrm{St}^*)} H^j\bigl(\mathbf{B}_{m-2}\bigr)
\longrightarrow 0.
\eeq
\elr

\bp
The components of the complex \eqref{3Complex} naturally form a 
$2\times 2$ bicomplex corresponding to the coordinates $(l,r)\in\{2a,2a+1\}\times\{2b, 2b+1\}$.
The complex \eqref{3Complex} is exactly the total complex of this bicomplex.

These terms are canonically identified with the cohomology groups
$H^j(\mathbf{B}_m)$, $H^j(\mathbf{B}_{m-1})\oplus H^j(\mathbf{B}_{m-1})$, and $H^j(\mathbf{B}_{m-2})$.
Evaluating the differential $d_1$ (whose explicit formula is given right after \eqref{SE1})
on these components yields the boundary maps.
For the first map, the source has even indices $l=2a$ and $r=2b$,
so $\chi(l)=\chi(r)=2$. The differential maps a class $u \in H^j(\mathbf{B}_m)$ to:
\[
u\mapsto\bigl(2\,\overline{\mathrm{St}}^*(u),\,2(-1)^{2a+j}\,\mathrm{St}^*(u)\bigr)
=\bigl(2\,\overline{\mathrm{St}}^*(u),\,2(-1)^j\,\mathrm{St}^*(u)\bigr).
\]
For the second map, the source components have indices $(2a+1,2b)$ and $(2a,2b+1)$.
Since $\chi(\text{odd})=0$, one term in the differential vanishes in each case.
A pair $(v_1,v_2)$ from the direct sum maps to the final term as:
\[
(v_1,v_2)\mapsto 2(-1)^{2a+1+j}\,\mathrm{St}^*(v_1)+2\,\overline{\mathrm{St}}^*(v_2)=
-2(-1)^j\,\mathrm{St}^*(v_1)+ 2\,\overline{\mathrm{St}}^*(v_2).
\]

We make two changes of coordinates. First, we change coordinates in the second direct summand
by setting $v_2=(-1)^j v_2'$. In these coordinates,
the first map becomes
\[
u\mapsto\bigl(2\,\overline{\mathrm{St}}^*(u),
2\,\mathrm{St}^*(u)\bigr),
\]
and the second becomes
\[
(v_1,v_2')\mapsto
-2(-1)^j\,\mathrm{St}^*(v_1)
+2(-1)^j\,\overline{\mathrm{St}}^*(v_2').
\]
Second, we change coordinates in the final term
$H^j(\mathbf B_{m-2})$ by multiplication by
$(-1)^{j+1}$. The second map then takes the form
\[
(v_1,v_2')\mapsto
2\,\mathrm{St}^*(v_1)-2\,\overline{\mathrm{St}}^*(v_2').
\]
We next verify that $\overline{\mathrm{St}}^*=\mathrm{St}^*$.
With the orientation signs induced by reversal of the
ordered vertices, the operator $\overline{\mathrm{St}}=\tau\;\mathrm{St}\;\tau$
acts on the Fuchs basis in cohomological degree $j$ by
\[
\overline{\mathrm{St}}\,e(n_1,\ldots,n_a)=
\begin{cases}
(-1)^j e(n_2,\ldots,n_a),&n_1=1,\\
0,&n_1>1.
\end{cases}
\]
By the proof of Theorem~\ref{SeqS}, every cohomology
class is a sum of classes represented by shuffle cycles
$c*e(1_k)$, where
$c\in C_{2r-j}^{(2r)}(\G[z])$ is a cycle
and $k=m-2r$.
For $k\geqslant1$, the shuffle terms beginning with
$e(1)$ acquire the sign $(-1)^{2r-j}=(-1)^j$
when this initial letter is removed.
This sign cancels the sign in the formula for
$\overline{\mathrm{St}}$. Consequently,
\[
\overline{\mathrm{St}}(c*e(1_k))=\mathrm{St}(c*e(1_k))=c*e(1_{k-1}).
\]
For $k=0$, both operators vanish, since every letter
of $c$ has positive even weight.
Thus $\overline{\mathrm{St}}^*=\mathrm{St}^*$, and substituting
$\mathrm{St}^*$ for $\overline{\mathrm{St}}^*$ in the maps
yields exactly the complex \eqref{ExH}.
Here terms with a negative braid index are understood to be zero.
\ep

\emph{We denote the $i$-th homology group of the complex \eqref{ExH} by
$H_i^{[j]}(m)$, where $i=0,1,2$}.

For $j\in\mathbb Z$, define
\[
\mathcal D^j(m):=
H^{[j]}_0(m)\oplus H^{[j-1]}_1(m)\oplus H^{[j-2]}_2(m),
\]
where $H_i^{[a]}(m)=0$ for $a<0$.

In view of Lemma \ref{LmExH}, to complete the description of $E_\infty(n)$,
it remains to describe the set
\[
\mathrm{H}^{[j]}(m):=\bigl\{\mathrm{H}^{[j]}_0(m),\mathrm{H}^{[j]}_1(m),\mathrm{H}^{[j]}_2(m)\bigr\}
\]
for all non-negative integers $m$ and $j$. The following statement provides this description:
\smallskip

\blr\label{LmH}
There are the following isomorphisms:
\beq\label{H01}
\mathrm{H}^{[0]}(m)=
\bcs
\{\mathbb{Z},0,0\}&\text{if $m=0$},\\
\{0,\mathbb{Z}\oplus\mathbb{Z}_2,0\}&\text{if $m=1$},\\
\{0,\mathbb{Z}_2,\mathbb{Z}_2\}&\text{if $m\geqslant 2$},
\ecs
\qquad
\mathrm{H}^{[1]}(m)=
\bcs
\{0,0,0\}&\text{if $m=0,1$},\\
\{\mathbb{Z},0,0\}&\text{if $m=2$},\\
\{0,\mathbb{Z}\oplus\mathbb{Z}_2,0\}&\text{if $m=3$},\\
\{0,\mathbb{Z}_2,\mathbb{Z}_2\}&\text{if $m\geqslant 4$},
\ecs
\eeq
\beq\label{H2}
\mathrm{H}^{[2]}(m)=\{0,0,0\}.
\eeq
Furthermore, for $j\geqslant 3$, there are the following isomorphisms:
\beq\label{Hq3}
\begin{aligned}
\mathrm{H}^{[j]}_0(m)&\cong\bigl[H^j(\mathbf{B}_m)\big/H^j(\mathbf{B}_{m-1})\bigr]
\plus H^j(\mathbf{B}_{m-1})\otimes\mathbb{Z}_2,
\\
\mathrm{H}^{[j]}_1(m)&\cong\bigl[H^j(\mathbf{B}_{m-1})\big/H^j(\mathbf{B}_{m-2})\bigr]
\plus H^j(\mathbf{B}_{m-2})\otimes\mathbb{Z}_2\plus H^j(\mathbf{B}_{m-1})\otimes\mathbb{Z}_2,
\\
\mathrm{H}^{[j]}_2(m)&\cong H^j\bigl(\mathbf{B}_{m-2}\bigr)\otimes\mathbb{Z}_2.
\end{aligned}
\eeq
\elr

\bp
The isomorphisms \eqref{H01} easily follow from Theorem~\ref{SeqS} and the isomorphisms
\[
H^0(\mathbf{B}_m)\cong\mathbb{Z}\;\;\text{if $m\geqslant 0$},\quad
H^1(\mathbf{B}_m)\cong
\bcs
\mathbb{Z}&\text{if $m\geqslant 2$},\\
0&\text{if $m\leqslant 1$},
\ecs
\quad
H^j(\mathbf{B}_m)\cong 0\;\;\text{if $m\leqslant 2$ and $j\geqslant 2$},
\]
together with the fact that the homology groups of the $3$-term complex
\[
0\longrightarrow\mathbb{Z}
\xrightarrow{\alpha}\mathbb{Z}\oplus\mathbb{Z}\xrightarrow{\beta}\mathbb{Z}\longrightarrow 0,\qquad
\text{where\quad $\alpha(a)=(2a,2a),\,\beta(a,b)=2a-2b$},
\]
are $\{0,\mathbb{Z}_2,\mathbb{Z}_2\}$.

The identity \eqref{H2} follows from the vanishing of
$H^2(\mathbf{B}_m)$ for all $m\geqslant 0$, which forces the complex \eqref{ExH} to be trivial.

It remains to justify the formulas \eqref{Hq3} for $j\geqslant 3$.
If $j\geqslant m$, all groups involved in \eqref{Hq3} vanish for
degree reasons. Therefore we may assume $3\leqslant j<m$.

Since the groups $H^j(\mathbf{B}_m)$ are finite,
using Theorem~\ref{SeqS} and Corollary~\ref{A2}, we obtain a split exact sequence
of finite groups:
\[
0\longrightarrow H^j(\mathbf{B}_m)\big/H^j(\mathbf{B}_{m-1})
\longrightarrow H^j\bigl(\mathbf{B}_m\bigr)
\stackrel{\mathrm{St}^*}\longrightarrow H^j\bigl(\mathbf{B}_{m-1}\bigr)
\longrightarrow 0,\qquad 3\leqslant j<m.
\]
Since $\mathrm{St}^*$ is a surjection, we have:
\begin{eqnarray*}
\Ker\bigl(2\,\mathrm{St}^*:H^j(\mathbf{B}_m)\longrightarrow H^j(\mathbf{B}_{m-1})\bigr)&\cong&
 H^j(\mathbf{B}_m)\big/H^j(\mathbf{B}_{m-1})\plus
H^j\bigl(\mathbf{B}_{m-1}\bigr)\otimes\mathbb{Z}_2,\\
\mathrm{Im}\bigl(2\,\mathrm{St}^*:H^j(\mathbf{B}_m)
\longrightarrow H^j(\mathbf{B}_{m-1})\bigr)
&\cong& 2\,H^j(\mathbf{B}_{m-1}).
\end{eqnarray*}
These isomorphisms evaluated on the complex~\eqref{ExH} directly yield
$\mathrm{H}^{[j]}_0(m)$ and $\mathrm{H}^{[j]}_2(m)$, which are naturally isomorphic to
$\Ker(2\,\mathrm{St}^*)$ and $\mathrm{Coker}(2\,\mathrm{St}^*)$, respectively. 
For the middle homology $\mathrm{H}^{[j]}_1(m)$, the kernel of the map $(2\,\mathrm{St}^*, -2\,\mathrm{St}^*)$
consists of pairs $(x,y)$ such that $x-y\in \Ker(2\,\mathrm{St}^*)$, giving an isomorphism
with $H^j(\mathbf{B}_{m-1})\oplus \Ker(2\,\mathrm{St}^*)$. 
The image of $(2\,\mathrm{St}^*, 2\,\mathrm{St}^*)$ is the diagonal subgroup consisting of elements
$(2w,2w)$ where $w\in H^j(\mathbf{B}_{m-1})$. 
Taking the quotient yields
\[
\mathrm{H}^{[j]}_1(m)\cong\bigl[H^j(\mathbf{B}_{m-1})\big/2H^j(\mathbf{B}_{m-1})\bigr]
\plus\Ker\bigl(2\,\mathrm{St}^*:H^j(\mathbf{B}_{m-1})\to H^j(\mathbf{B}_{m-2})\bigr),
\]
which precisely expands to the three summands in $\mathrm{H}^{[j]}_1(m)$ via the established isomorphisms.
\ep
\btr\label{HK2}
For $n\geqslant5$, one has $H^q(\mathrm K_2(n))=0$ for
$q\leqslant 3$. For $q\geqslant 4$, there is an isomorphism
\[
H^q\bigl(\mathrm K_2(n)\bigr)
\cong
\bigoplus_{\substack{p\ {\rm even},\ 4\leqslant p\leqslant n-1\\
j+p=q}}
\mathcal D^j(n-p-1)^{\oplus(p/2-1)} .
\]
\etr

\bp
Since $\mathrm K_2(n)$ has no cochains in degrees below $4$,
we have $H^q(\mathrm K_2(n))=0$ for $q\leqslant3$.
Lemmas~\ref{LmD2Zero} and~\ref{LmH} determine
$E_\infty(n)=E_2(n)$ and hence all successive quotients of the induced
filtration on $H^q(\mathrm K_2(n))$. To obtain the asserted direct sum
decomposition of the cohomology itself, it remains to prove that this
filtration splits over $\mathbb Z$.

The block decomposition of the spectral sequence comes from a direct sum
decomposition of the filtered cochain complex $\mathrm K_2(n)$ itself.
Each block supported at
\[
(2a,2b),\quad(2a+1,2b),\quad(2a,2b+1),\quad(2a+1,2b+1)
\]
is preserved by the total differential $\partial_{\widetilde D}$.
Consequently, the cohomology is the direct sum of the cohomology groups of
these blocks. The remaining extension problem therefore concerns the
filtration within each individual block.

Let $X$ be one such block with lower-left corner $(2a, 2b)$.
The elements of $X$ carry the common factor $x^{2a}\times\cdots\times y^{2b}$,
shifting the cohomological degree by $p_0=2a+2b$.
Let $Y$ be the underlying unshifted complex, so that $X\cong Y[-p_0]$ and $H^q(X)\cong H^{q-p_0}(Y)$.
Setting $j=q-p_0$ and $m=n-p_0-1$, the complex $Y$ is the total complex of the anti-commutative square (bicomplex):
\[
\begin{array}{ccc}
\mathrm{A}(m) & \xrightarrow{2\,\mathrm{St}} & \mathrm{A}(m-1) \\
2\,\overline{\mathrm{St}}\downarrow \phantom{2\,\overline{\mathrm{St}}} &
& \phantom{-2\,\overline{\mathrm{St}}} \downarrow -2\,\overline{\mathrm{St}} \\
\mathrm{A}(m-1) & \xrightarrow{2\,\mathrm{St}} & \mathrm{A}(m-2)
\end{array}
\]
Thus, $Y$ can be viewed as the mapping cone of the chain map
$2\,\overline{\mathrm{St}}:\mathrm{Cn}(m)\to\mathrm{Cn}(m-1)$,
where $\mathrm{Cn}(k)=\mathrm{Cone}\bigl(2\,\mathrm{St}:\mathrm{A}(k)\to\mathrm{A}(k-1)\bigr)$
is the complex studied in Proposition~\ref{HTi00}.

It remains to show that the induced filtration on
$H^j(Y)$ splits over $\mathbb Z$, yielding
$H^j(Y)\cong\mathcal D^j(m)$.

First, consider an odd prime $\ell$.
Since multiplication by $2$ is an automorphism over $\mathbb{Z}_{(\ell)}$,
Theorem~\ref{SeqS} and the complex \eqref{ExH} give:
\[
H_0^{[s]}(m)_{(\ell)}\cong H^s(\mathrm A_2(m))_{(\ell)},\qquad
H_1^{[s]}(m)_{(\ell)}\cong H^s(\mathrm A_2(m-1))_{(\ell)},\qquad
H_2^{[s]}(m)_{(\ell)}=0.
\]
Since $H^*(\mathrm A_2(k))=0$ for odd $k$, the $\ell$-primary part of the
$E_\infty$-page of each block is concentrated in at most one filtration degree
(because $m$ and $m-1$ have opposite parities). Hence, there are absolutely
no extension problems for odd-primary components.

To analyze the $2$-primary extensions, we localize the complexes at the prime
$2$ by tensoring with $\mathbb{Z}_{(2)}$, and analyze two cases based on $m$.
For a complex $Z$, write
$Z_{(2)}:=Z\otimes_{\mathbb Z}\mathbb Z_{(2)}$.

\emph{Case 1: $m \geqslant 4$.}
For $k\geqslant3$, the cone exact sequence gives
\[
H^0(\mathrm{Cn}(k))=0,\qquad
H^1(\mathrm{Cn}(k))\cong H^2(\mathrm{Cn}(k))
\cong\mathbb Z_2.
\]
Together with Proposition~\ref{HTi00}, this shows that
\[
2\,H^i\bigl(\mathrm{Cn}(k)\otimes\mathbb Z_{(2)}\bigr)=0
\qquad\text{for every }i.
\]
Since $m\geqslant4$, this applies to both
$\mathrm{Cn}(m)$ and $\mathrm{Cn}(m-1)$.

Let $(u,v)\in Y_{(2)}$ be a cocycle of degree $j$.
The differential in the cone $Y=\mathrm{Cone}(2\,\overline{\mathrm{St}})$
is given by
$\partial_Y(u,v)=\bigl(\partial_{\mathrm{Cn}(m)} u,\partial_{\mathrm{Cn}(m-1)}v+(-1)^j\,2\,\overline{\mathrm{St}}(u)\bigr)$.
Since $u$ is a cycle in $\mathrm{Cn}(m)_{(2)}$ and its cohomology
group is a squarefree-torsion $2$-group, there exists a cochain
$w\in\mathrm{Cn}^{j-1}(m)_{(2)}$ such that $2u=\partial_{\mathrm{Cn}(m)} w$. 
Evaluating the boundary of $(w,0)$, we get $\partial_Y(w,0)=\bigl(2u,(-1)^{j-1}\,2\,\overline{\mathrm{St}}(w)\bigr)$. 
Subtracting this from $2(u,v)$ yields:
\[
2(u,v)-\partial_Y(w,0)
=\bigl(0,2v+(-1)^j\,2\,\overline{\mathrm{St}}(w)\bigr)
=(0,2z).
\]
Set $z=v+(-1)^j\,\overline{\mathrm{St}}(w)\in\mathrm{Cn}^{j-1}(m-1)_{(2)}$.
The cocycle condition for $(u,v)$ implies $\partial_{\mathrm{Cn}(m-1)}v+(-1)^j\,2\,\overline{\mathrm{St}}(u)=0$. Therefore, 
\[
\partial_{\mathrm{Cn}(m-1)}z=\partial_{\mathrm{Cn}(m-1)}v+(-1)^j\,\overline{\mathrm{St}}
(\partial_{\mathrm{Cn}(m)}w)=\partial_{\mathrm{Cn}(m-1)}v+(-1)^j\,2\,\overline{\mathrm{St}}(u)=0,
\]
showing that $z$ is a cycle in $\mathrm{Cn}(m-1)_{(2)}$.
Since $H^{j-1}(\mathrm{Cn}(m-1)_{(2)})$ is a squarefree-torsion $2$-group,
there exists $t$ such that $2z=\partial_{\mathrm{Cn}(m-1)} t$.
This gives $2[(u,v)]=[(0,2z)]=[(0,\partial_{\mathrm{Cn}(m-1)}t)]=0$.
Thus $H^j(Y_{(2)})$ is annihilated by $2$.
Its filtration is therefore a filtration of
$\mathbb F_2$-vector spaces and splits.

\emph{Case 2: $0\leqslant m\leqslant3$.}
In this case,
\[
Y^j=\mathrm A^j(m)\oplus
\mathrm A^{j-1}(m-1)^{\oplus2}\oplus
\mathrm A^{j-2}(m-2)=0
\qquad(j\geqslant3).
\]
For $m=0$, the complex consists of $\mathbb Z$ in degree $0$.
For $m=1$, it is
$\mathbb Z\xrightarrow{(2,2)}\mathbb Z^2$.
For $m=2$, it is
$\mathbb{Z}\xrightarrow{(0,2,2)}\mathbb Z^3\xrightarrow{(0,-2,2)}\mathbb{Z},
$
with the leftmost term in degree $0$.

The subscripts $L$, $R$, and $LR$ indicate the copies
of the Fuchs complexes at the corners
$(2a+1,2b)$, $(2a,2b+1)$, and $(2a+1,2b+1)$,
respectively; elements without subscripts belong
to the corner $(2a,2b)$.
For $m=3$, choose the ordered bases
\begin{gather*}
Y^0:\quad e(1,1,1),\qquad Y^1:\quad\bigl\{e(1,2),\,e(2,1),\,e(1,1)_L,\,e(1,1)_R\bigr\},\\
Y^2:\quad\bigl\{e(3),\,e(2)_L,\,e(2)_R,\,e(1)_{LR}\bigr\}.
\end{gather*}
The differential matrices are
\[
d^0=
\begin{pmatrix}0\\0\\2\\2\end{pmatrix},
\qquad
d^1=
\begin{pmatrix}
-1&-1&0&0\\
-2&0&0&0\\
0&-2&0&0\\
0&0&-2&2
\end{pmatrix}.
\]
These complexes give
\[
\begin{array}{c|ccc}
m&H^0(Y)&H^1(Y)&H^2(Y)\\ \hline
0&\mathbb{Z}&0&0\\
1&0&\mathbb{Z}\oplus\mathbb{Z}_2&0\\
2&0&\mathbb{Z}\oplus\mathbb Z_2&\mathbb{Z}_2\\
3&0&\mathbb{Z}_2&\mathbb{Z}\oplus\mathbb{Z}_2^{\oplus2}.
\end{array}
\]
Comparison with Lemma~\ref{LmH} yields
$H^j(Y)\cong\mathcal D^j(m)$ in every case.

In all cases, all extensions split, yielding $H^q(X)\cong H^{q-p_0}(Y)\cong\mathcal{D}^j(m)$. 
Since for a fixed even $p_0=p\geqslant 4$ there are exactly
$\frac{p}{2}-1$ isolated blocks summing to $p$,
each block $\mathcal{D}^j(n-p-1)$ appears with this exact multiplicity.
\ep

\subsection*{The group $H^*\bigl(\texorpdfstring{\widetilde{\mathbf{D}}_n}{\textasciitilde D_n}\bigr)$}
In view of the decomposition \eqref{K1K2}, we have the isomorphism
\beq\label{FinDD}
H^q(\widetilde{\mathbf{D}}_n)\cong
H^q\bigl(\mathrm{K}_1(n)\bigr)\oplus H^q\bigl(\mathrm{K}_2(n)\bigr).
\eeq
Consequently, Theorems \ref{HK1} and \ref{HK2}, combined with Lemma \ref{LmH},
yield a complete description of the additive cohomology structure
of the groups $\widetilde{\mathbf{D}}_n$.
\smallskip

\btr\label{StabS_DD}
For $n\geqslant6$, there is a canonical morphism of complexes
\[
\mathrm{St}_{\widetilde D}:\widetilde{\mathrm{D}}(n)
\longrightarrow
\widetilde{\mathrm{D}}(n-1).
\]
For any $q\geqslant1$, the induced homomorphism
$
\mathrm{St}^*_{\widetilde D}:
H^q(\widetilde{\mathbf D}_n)
\longrightarrow
H^q(\widetilde{\mathbf D}_{n-1})
$
is an isomorphism if $n\geqslant\max\{6,2q+1\}$.
Thus, the cohomology of the family $\widetilde{\mathbf D}$ stabilizes.
\etr

\bp
The isomorphism~\eqref{FinDD} shows that, in order to prove the theorem,
it is enough to establish the existence of canonical morphisms of complexes
\[
\mathrm{St}_{\mathrm{K}_m}:\mathrm{K}_m(n)\longrightarrow \mathrm{K}_m(n-1),
\qquad m\in\{1,2\},
\]
which induce isomorphisms
$
\mathrm{St}_{\mathrm{K}_m}^*:
H^q(\mathrm{K}_m(n))\longrightarrow H^q(\mathrm{K}_m(n-1))
$
if $n\geqslant 2q+1$.

For $m=1$, use the canonical decompositions
\eqref{K} and those established in the proofs of
Lemmas~\ref{Dd0} and~\ref{HDDT0}:
\[
\mathrm K_1(n)\cong
\mathrm T_0(n-1)\oplus
\mathrm T(n-1)^{\oplus2}\oplus
\mathrm T^{(0)}_0(n-1)\oplus
\mathrm T^{(2)}(n-1)^{\oplus2}.
\]
Define $\mathrm{St}_{\mathrm K_1}$ by the restrictions
of $\mathrm{St}_D$ on $\mathrm T_0(n-1)$ and on both
copies of $\mathrm T(n-1)$, and by zero on the
remaining summands.
This defines a morphism of complexes.

For $n\geqslant\max\{6,2q+1\}$,
Lemma~\ref{HTT}, the proof of Lemma~\ref{HDDT0},
and Corollary~\ref{A2} show that the additional
summands have zero cohomology in degree $q$
in both the source and the target.
On the remaining summands, the induced maps are
isomorphisms by the proof of Theorem~\ref{StabS_D}.
Hence $\mathrm{St}_{\mathrm K_1}^*$ is an
isomorphism in degree $q$.

Define the map $\mathrm{St}_{\mathrm{K}_2}$ by
\[
\mathrm{St}_{\mathrm{K}_2}(x^l\times c\times y^r)=
\bcs
x^l\times\mathrm{St}(c)\times y^r,&\text{if $l,r\geqslant 2$ and $l+r<n-1$},\\
0&\text{otherwise}.
\ecs
\]
Formula \eqref{BoundK2} shows that $\mathrm{St}_{\mathrm{K}_2}$ is a morphism of complexes.
It induces a morphism of spectral sequences whose first pages are
described by \eqref{SE1}.

For $q\leqslant3$, both cohomology groups of $\mathrm K_2$
vanish. Let $q\geqslant4$ and
$n\geqslant2q+1$.
For $p\geqslant4$, $j\geqslant0$, and
$p+j\in\{q-1,q,q+1\}$, we have
\[
n-p-1\geqslant2j+p-2\geqslant2j+2.
\]
Thus Theorem~\ref{StabSH}, together with the elementary
degree-zero case, shows that the induced map on $E_1$
is an isomorphism in these three total degrees.
It follows that the map on $E_2$ is an isomorphism
in total degree $q$.
By Lemma~\ref{LmD2Zero}, both spectral sequences
collapse at $E_2$.
The induced map on the associated graded groups
of $H^q$ is therefore an isomorphism.
Since the filtrations are finite, the map on $H^q$
itself is an isomorphism.
\ep

For each $j$, let $\mathcal D^j_\infty$ denote the stable value of
$\mathcal D^j(m)$ as $m\to\infty$.

\bcr\label{StableDD}
For any $q\geqslant 2$ and $n\geqslant 2q+1$,
the group $H^q\bigl(\widetilde{\mathbf{D}}_n\bigr)$ stabilizes.
The stable cohomology is given by the isomorphism:
\begin{multline*}
\qquad\qquad H^q(\widetilde{\mathbf D}_\infty)
\cong
H^q(\mathbf D_\infty)
\oplus \\
\left[
\plusn_{i\geqslant1}
H^{q-2i-1}(\mathbf B_\infty)\otimes\mathbb Z_2
\oplus
\plusn_{i\geqslant1}
H^{q-2i}(\mathbf B_\infty)[2]
\right] \\[1mm]
\oplus
\plusn_{\substack{s\ {\rm even},\ s\geqslant4\\ j+s=q}}
\bigl(\mathcal D^j_\infty\bigr)^{\oplus(s/2-1)}.\qquad\qquad
\end{multline*}
Thus, for any prime $p$, as well as for $p=0$, we have:
\[
\mathrm{G}_p(\widetilde{\mathbf{D}}_\infty;t)=
\bcs
\mathrm{G}_p(\mathbf{B}_\infty;t)&\text{if $p\neq 2$},\\[2mm]
\mathrm{G}_2(\mathbf B_\infty;t)+\left(\frac{2t^2}{1-t}+\frac{t^4}{(1-t)^2}\right)
\bigl(\mathrm{G}_2(\mathbf B_\infty;t)-1\bigr)&\text{if $p=2$}.
\ecs
\]
\ecr

\bp
By Lemma~\ref{LmH}, the groups $\mathcal D^j_\infty$ are
$2$-primary and satisfy
\[
\sum_{j\geqslant0}
\dim_{\mathbb F_2}\bigl(\mathcal D^j_\infty\otimes\mathbb F_2\bigr)t^j
=
(1+t)^2\bigl(\mathrm{G}_2(\mathbf B_\infty;t)-1\bigr).
\]
The stable decomposition follows from \eqref{FinDD}, Theorems~\ref{HK1}
and~\ref{HK2}, Lemma~\ref{LmH}, and Corollary~\ref{StableD}.
All summands outside $H^q(\mathbf D_\infty)$ are $2$-primary, so for
$p\neq2$ they do not contribute to $\mathrm{G}_p$. This gives
\[
\mathrm{G}_p(\widetilde{\mathbf D}_\infty;t)=\mathrm{G}_p(\mathbf B_\infty;t),
\qquad p\neq2.
\]
For $p=2$, write $B(t)=\mathrm{G}_2(\mathbf B_\infty;t)$. By
Corollary~\ref{StableD},
\[
\mathrm{G}_2(\mathbf{D}_\infty;t)
=
B(t)+\frac{t^2}{1-t}\bigl(B(t)-1\bigr).
\]
The second summand in the stable decomposition contributes another
$\frac{t^2}{1-t}\bigl(B(t)-1\bigr)$.

Since
$\sum_{k\geqslant2}(k-1)t^{2k}=\frac{t^4}{(1-t^2)^2}$,
the $K_2$-part contributes
\[
(1+t)^2\bigl(B(t)-1\bigr)
\sum_{k\geqslant2}(k-1)t^{2k}=
\frac{t^4}{(1-t)^2}\bigl(B(t)-1\bigr).
\]
Adding these three contributions gives the stated formula.
\ep

Theorem~\ref{HK1} and Lemmas~\ref{LmD2Zero}
and~\ref{LmH} imply the following statement:
\smallskip

\bcr\label{HDQ}
For $n\geqslant 5$, we have:
\[
\dim_{\mathbb{Q}}\;H^q\bigl(\widetilde{\mathbf{D}}_n;\mathbb{Q}\bigr)=
\bcs
1&\text{if $q=0, 1$},\\[1mm]
\frac{n+3}{2}&\text{if $q=n-2$ and $n$ is odd},\\[1mm]
\frac{n-2}{2}&\text{if $q=n-2$ and $n$ is even},\\[1mm]
\frac{n+5}{2}&\text{if $q=n-1$ and $n$ is odd},\\[1mm]
\frac{n}{2}&\text{if $q=n-1$ and $n$ is even},\\[1mm]
0&\text{otherwise}.
\ecs
\]
\ecr

\bp
According to the decomposition \eqref{FinDD}, we will evaluate the
rational dimensions for both components in the non-trivial degrees $q=n-2$ and $q=n-1$.

Step 1: \emph{Rational cohomology of $\mathrm{K}_1(n)$.}
By Theorem \ref{HK1}, the rational cohomology of $\mathrm{K}_1(n)$ is assembled from
$H^q(\mathbf{D}_{n-1};\mathbb{Q})$, the terms of the form $H^d(\mathrm{A}_2(m);\mathbb{Q})$, 
and the explicit top degree term $\mathbb{Z}^{\oplus 2}\plusn\bigl[\mathbb{Z}\big/\chi(n)\mathbb{Z}\bigr]^{\oplus 2}$. 
From Corollary \ref{A2}, $H^d(\mathrm{A}_2(m);\mathbb{Q})$ is non-zero (and isomorphic to $\mathbb{Q}$) 
if and only if $(m,d)=(0,0)$ or $(m,d)=(2,1)$.
Evaluating these constraints:
\begin{itemize}
    \item[--] \textit{For $q=n-2$}: 
    If $n$ is odd, we collect $1$ from $H^{n-2}(\mathbf{D}_{n-1};\mathbb{Q})$.
    The single sum 
    $$\plusn_{i\geqslant 1}H^{n-2-2i}(\mathrm{A}_2(n-2i-1))$$ contributes $\mathbb{Q}$ 
    when $(m,d)=(2,1)$, meaning $n-2i-1=2$ and $n-2-2i=1$, which holds for $2i=n-3$.
    The doubled sum $$\plusn_{i\geqslant 2}H^{n-2-i}(\mathrm{A}_2(n-i-1))^{\oplus 2}$$ contributes $\mathbb{Q}^{\oplus 2}$ 
    when $n-i-1=2$ and $n-2-i=1$, which holds for $i=n-3$.
    Thus, for $n$ odd, $\dim_{\mathbb{Q}}=1+1+2=4$. 
    If $n$ is even, both $\mathbf{D}_{n-1}$ and the single sum vanish (since $n-3$ is odd), 
    leaving only the doubled sum ($i=n-3$), which yields $\dim_{\mathbb{Q}}=2$.
    \item[--] \textit{For $q=n-1$}: 
    By Theorem \ref{HK1}, the rational dimension is entirely determined by the top degree term 
    $\mathbb{Z}^{\oplus 2} \oplus \bigl[\mathbb{Z}\big/\chi(n)\mathbb{Z}\bigr]^{\oplus 2} \otimes \mathbb{Q}$.
    If $n$ is odd, $\chi(n)=0$, so the second term yields $\mathbb{Q}^{\oplus 2}$, giving a total $\dim_{\mathbb{Q}}=2+2=4$. 
    If $n$ is even, $\chi(n)=2$, meaning the second term is pure torsion, leaving only the first term, so $\dim_{\mathbb{Q}}=2$.
\end{itemize}

Step 2: \emph{Rational cohomology of $\mathrm{K}_2(n)$.}
By Lemma~\ref{LmD2Zero}, the spectral sequence
collapses at $E_2$. Hence the rational dimension of
$H^q(\mathrm K_2(n))$ is the sum of the rational
dimensions of the $E_2$-terms of total degree $q$.
Tensoring with $\mathbb Q$ annihilates all torsion.
By Lemma \ref{LmH}, the only non-zero rational groups are:
\[
\left\{\mathrm{H}^{[0]}_0(0),\;\mathrm{H}^{[0]}_1(1),\;
\mathrm{H}^{[1]}_0(2),\;\mathrm{H}^{[1]}_1(3)\right\}\otimes\mathbb{Q}\cong\left\{\mathbb{Q},\mathbb{Q},\mathbb{Q},\mathbb{Q}\right\}.
\]
Each corresponding block contributes exactly $1$ to the rational dimension,
provided its assigned filtration $p=n-m-1$ is even (which enforces the parity of $n$),
and the multiplicity is $\frac{p}{2}-1$.
\begin{itemize}
    \item[--] \textit{Degree $q=n-2$}: 
    If $n$ is odd, only the block containing $\mathrm{H}^{[1]}_0(2)$
    contributes ($p=n-3$), giving a dimension of $\frac{n-5}{2}$. 
    If $n$ is even,
    the block containing $\mathrm{H}^{[1]}_1(3)$ ($p=n-4$)
    gives $\frac{n-6}{2}$.
    \item[--] \textit{Degree $q=n-1$}:
  If $n$ is odd, only the block containing
  $\mathrm H^{[0]}_0(0)$ contributes $(p=n-1)$, giving a dimension
  of $\frac{n-3}{2}$.
  If $n$ is even, only the block containing
  $\mathrm H^{[0]}_1(1)$ contributes $(p=n-2)$, giving a dimension
  of $\frac{n-4}{2}$.
\end{itemize}

Step 3: \emph{Total rational dimensions.}
Adding the contributions from $\mathrm{K}_1(n)$ and $\mathrm{K}_2(n)$ yields the final rational ranks:
\begin{itemize}
    \item[--] \textit{For $q=n-2$}: 
    If $n$ is odd, $4+\frac{n-5}{2}=\frac{n+3}{2}$. If $n$ is even, $2+\frac{n-6}{2}=\frac{n-2}{2}$.
    \item[--] \textit{For $q=n-1$}:
    If $n$ is odd, $4+\frac{n-3}{2}=\frac{n+5}{2}$. If $n$ is even, $2+\frac{n-4}{2}=\frac{n}{2}$.
\end{itemize}
Apart from $H^0\cong H^1\cong\mathbb Q$,
all remaining rational cohomology groups vanish.
\ep

\appendix

\setcounter{theorem}{0}
\renewcommand{\thetheorem}{A.\arabic{theorem}}
\renewcommand{\theHtheorem}{appendix.A.\arabic{theorem}}

\phantomsection
\section*{Appendix. The Normalized Bar Complex with Trivial Coefficients}
\label{AA}

The constructions recalled in this appendix are standard;
see \cite[Ch.~VIII, \S8; Ch.~X, \S\S7,12]{MacL}.
We give the explicit formulas to fix our notation
and sign conventions.

Let $k$ be a unital, associative, and commutative ring.
Below we consider the category $\mathcal A$ of unital, associative,
commutative graded $k$--algebras $A=\plusn_{i\geqslant 0} A_i$
with unit element $1_A\in A_0$, where the unit map identifies $k$ with $A_0$.
Each summand $A_i$ is assumed to be a~finitely generated $k$--module for $i\geqslant 0$.
We set $A_+:=\oplus_{i\geqslant 1}A_i$.
The identification $A_0\cong k$ defines the augmentation homomorphism
$\varepsilon_A:A\longrightarrow k$, given by projection onto $A_0$.
For $a\in A$, set
\[
\overline a:=a-\varepsilon_A(a)1_A\in A_+.
\]

An element $a\in A_i$ is said to be \textit{homogeneous of weight $w(a):=i$}.
The grading on the algebras in $\mathcal A$ is the weight grading.
In $\mathcal A$ is defined the tensor product $A\otimes B\in\mathcal{A}$
with multiplication
$(a_1\otimes b_1)(a_2\otimes b_2)=a_1 a_2\otimes b_1 b_2$.
The augmentation
\[
\varepsilon_{A\otimes B}:A\otimes B\longrightarrow k\qquad\text{is given by}\qquad
\varepsilon_{A \otimes B}(a \otimes b)=\varepsilon_A(a) \cdot \varepsilon_B(b).
\]
Thus,
\[
(A\otimes B)_+=
(A_+\otimes k1_B)\oplus(k1_A\otimes B_+)
\oplus(A_+\otimes B_+).
\]

\bdr\label{Hcomp}
The $k$--module
\[
C_q(A;k):=
\bcs
A_+^{\otimes q}&\text{if $q>0$},\\
k&\text{if $q=0$}
\ecs
\]
is called the (normalized) \textit{module of chains of degree $q$ of the algebra $A$}.

For $q>0$, elementary tensors we will write as
$\bigl[a_1|\dots|a_q\bigr]$, where $a_i\in A_+$.
If each $a_i$ is homogeneous, this is a \emph{homogeneous chain, of weight}
\[
w\bigl[a_1|\dots|a_q\bigr]:=\sum_{i=1}^q w(a_i).
\]
\edr

\bdr
Consider the sequence of homomorphisms
\[
C_*(A;k):\quad k\stackrel{b}\longleftarrow C_1(A;k)\stackrel{b}
\longleftarrow\cdots\stackrel{b}\longleftarrow C_{q-1}(A;k)
\stackrel{b}\longleftarrow C_q(A;k)\stackrel{b}\longleftarrow\cdots.
\]
The bar differential $b:C_q(A;k)\to C_{q-1}(A;k)$ is defined for $q>0$ by
\[
b\bigl(\bigl[a_1|a_2|\dots|a_q\bigr]\bigr)=
\sum_{i=1}^{q-1}(-1)^i\bigl[a_1|\dots|a_ia_{i+1}|\dots|a_q\bigr].
\]
In particular, $b:C_1(A;k)\to C_0(A;k)=k$ is zero.

This sequence forms a~complex, called the (normalized)
\textit{bar complex of the algebra $A$.}
Its homology will be called the \emph{homology of the $k$-algebra $A$}
(with trivial coefficients).
\[
H_*(A;k):=\plusn_{q\geqslant 0}H_q(A;k),
\]
where $H_q(A;k)$ is the \emph{homology group in degree $q$}.
For $h\in H_q(A;k)$, we write $\deg(h)=q$.
\edr

The module $C_*(A;k)$ decomposes as a direct sum of subcomplexes generated by homogeneous chains:
$C_*(A;k)=\plusn_{w\geqslant 0} C^{(w)}_*(A;k)$.
The homology of $C^{(w)}_*(A;k)$ is denoted by $H_*^{(w)}(A;k)$ and is called the
\textit{homology of weight $w$ of the algebra $A$}.

Let us now describe the homology of the tensor product of two
$k$--algebras in terms of the homology of the factors.
Denote by $S_r$ the symmetric group on the set $\{1,2,\dots,r\}$.

\bdr\label{Schuffle}
For integers $q_1,q_2\geqslant 0$, a~permutation $\sigma\in S_{q_1+q_2}$ is a~\textit{$(q_1,q_2)$--shuffle} if
\[
\sigma(1)<\sigma(2)<\dots<\sigma(q_1)\qquad\text{and}\qquad
\sigma(q_1+1)<\sigma(q_1+2)<\dots<\sigma(q_1+q_2).
\]
The set of all $(q_1,q_2)$--shuffles is denoted by $Sh(q_1,q_2)$.
\edr

For $x=\bigl[a_1|\dots|a_{q_1}\bigr]$ and $y=\bigl[b_1|\dots|b_{q_2}\bigr]$,
the \textit{shuffle product} of $x$ and $y$ is defined by
\[
x*y:=\sum_{\sigma\in Sh(q_1,q_2)}(-1)^\sigma\left[c_{\sigma^{-1}(1)}|\dots|c_{\sigma^{-1}(q_1+q_2)}\right]
\]
where $(-1)^\sigma$ denotes the sign of the permutation $\sigma$ and
$c_i:=a_i\otimes 1_B$,
$c_{q_1+j}:=1_A\otimes b_j$.
\bdr
The operation $*$ by linearity defines a homomorphism of graded modules
\[
C_*(A;k)\otimes C_*(B;k)\longrightarrow C_*(A\otimes B;k),
\]
called the \textit{shuffle product}.
\edr

\btr\label{EZ}
The operation $*$ satisfies the following identities:
\begin{enumerate}
  \item[\rm(a)] $b(x*y)=b(x)*y+(-1)^{\deg(x)}x*b(y)$, where $x\in C_{\deg(x)}(A;k),y\in C_*(B;k)$.
\smallskip
  \item[\rm(b)] If $x,y\in C_*(A;k)$, then, under the automorphism of
$A\otimes A$ induced by the flip map
\[
A\otimes A\longrightarrow A\otimes A,\qquad a\otimes b\mapsto b\otimes a,
\]
one has
$x*y=(-1)^{\deg(x)\deg(y)}y*x$.
  \smallskip
  \item[\rm(c)] Under the canonical associativity identification of tensor
products, one has
\[
(x*y)*z=x*(y*z).
\]
\end{enumerate}
\etr

\bdr\label{shprod}
For algebras $A$ and $B$, consider the graded module homomorphism
\[
\mathrm{EZ}:C_*(A;k)\otimes C_*(B;k)\to C_*(A\otimes B;k),
\]
defined by
\[
\mathrm{EZ}\left(\big[a_1|\dots|a_{q_1}\big]\otimes\big[b_1|\dots|b_{q_2}\big]\right)
:=
\big[a_1|\dots|a_{q_1}\big]*\big[b_1|\dots|b_{q_2}\big].
\]
This map is called the \textit{Eilenberg--Zilber homomorphism}.
\edr

\btr\label{ttth1}
The homomorphism $\mathrm{EZ}$ is a morphism of complexes.
If both $A$ and $H_*(A;k)$ are free $k$--modules,
then, for every $B\in\mathcal A$, the K\"unneth identification
\[
H_*\bigl(C_*(A;k)\otimes C_*(B;k)\bigr)
\cong H_*(A;k)\otimes H_*(B;k)
\]
gives an induced isomorphism
\[
\mathrm{EZ}_*:H_*(A;k)\otimes H_*(B;k)\longrightarrow H_*(A\otimes B;k).
\]
\etr

Since $A\in\mathcal A$ is a commutative algebra, the map
\[
\mu:A\otimes A\to A,\qquad\mu(a_1\otimes a_2)=a_1a_2
\]
is an algebra homomorphism.
By functoriality and the chain-map, associativity, and
graded-symmetry identities in Theorem~\ref{EZ}, we obtain:

\bcr\label{ttth2}
For $A\in\mathcal A$, the composition
\[
H_*(A;k)\otimes H_*(A;k)\stackrel{\mathrm{EZ}_*}{\longrightarrow}
H_*(A\otimes A;k)\stackrel{\mu_*}{\longrightarrow} H_*(A;k)
\]
endows $H_*(A;k)$ with the structure of a graded-commutative algebra.
\ecr

\bdr\label{dPon}
The graded-commutative algebra structure on $H_*(A;k)$ defined in
Corollary~\ref{ttth2} is called the \emph{Pontryagin algebra of $A$}.
\edr

\bdr\label{AW}
For algebras $A$ and $B$, consider the graded module homomorphism
\[
\mathrm{AW}:C_*(A\otimes B;k)\longrightarrow C_*(A;k)\otimes C_*(B;k),
\]
defined by
\[
\mathrm{AW}\bigl[a_1\otimes b_1|\cdots|a_q\otimes b_q\bigr]
:=\sum_{r=0}^{q}
\left(\prod_{i=1}^{r}\varepsilon_B(b_i)\right)
\left(\prod_{\,i=r+1}^{q}\varepsilon_A(a_i)\right)
\bigl[\overline a_1|\cdots|\overline a_r\bigr]\otimes
\bigl[\overline b_{r+1}|\cdots|\overline b_q\bigr].
\]
This map is called the \emph{Alexander–Whitney homomorphism}.

Empty products are understood to equal $1$, and an empty
chain denotes $1\in C_0(A;k)=k$ or $1\in C_0(B;k)=k$,
as appropriate.
A bracket containing a zero entry represents the zero chain.
\edr

\btr\label{AWH}
The homomorphism $\mathrm{AW}$ is a~morphism of complexes.
If both $A$ and $H_*(A;k)$ are free $k$--modules, then for any algebra $B$ the induced map
\[
\mathrm{AW}_*:H_*(A\otimes B;k)\longrightarrow H_*(A;k)\otimes H_*(B;k)
\]
is an isomorphism of $k$--modules, inverse to the isomorphism $\mathrm{EZ}_*$.
\etr

Recall that a \emph{graded Hopf algebra} over $k$ is a
nonnegatively graded unital associative algebra $H$,
equipped with a coassociative coproduct
$\Delta:H\to H\otimes H$, a counit $\varepsilon:H\to k$,
and an antipode $S:H\to H$, all preserving degrees.
The coproduct and counit are unital algebra
homomorphisms, where multiplication on $H\otimes H$
is defined by
\[
(a\otimes b)(c\otimes d)
=(-1)^{\deg(b)\deg(c)}ac\otimes bd
\]
for homogeneous elements.
The counit and antipode identities are
\[
(\varepsilon\otimes\mathrm{id})\Delta
=\mathrm{id}
=(\mathrm{id}\otimes\varepsilon)\Delta,
\qquad
\mu(S\otimes\mathrm{id})\Delta
=\eta\varepsilon
=\mu(\mathrm{id}\otimes S)\Delta,
\]
where $\mu:H\otimes H\to H$ and $\eta:k\to H$
denote multiplication and the unit map, respectively.
The Hopf algebra is called connected if $\eta$
identifies $k$ with $H_0$.
It is graded-commutative if
$ab=(-1)^{\deg(a)\deg(b)}ba$ for homogeneous $a,b$.
For $H_*(A;k)$, these signs refer to homological
degrees; the weight grading introduces no signs.

In the following statement, the Hopf algebra structure
on $A$ uses the ordinary tensor product multiplication,
without signs depending on weights, as stipulated above.

Theorem~\ref{AWH} and Corollary~\ref{ttth2}
yield the following statement:

\bcr\label{HopfHomology}
Suppose that $A\in\mathcal A$ is a commutative Hopf
algebra whose coproduct preserves weights and whose
counit is the augmentation $\varepsilon_A$.
Assume that both $A$ and $H_*(A;k)$ are free
$k$--modules. Then the composition
\[
H_*(A;k)
\xrightarrow{(\Delta_A)_*} H_*(A\otimes A;k)
\xrightarrow{\mathrm{AW}_*}
H_*(A;k)\otimes H_*(A;k)
\]
defines a coproduct which, together with the
Pontryagin product, makes $H_*(A;k)$ a connected
graded-commutative Hopf algebra.
The grading here is the homological grading.
\ecr

\bp
The Alexander--Whitney formula and the counit identities
show that the composite $\mathrm{AW}\circ C_*(\Delta_A)$
is given by
\[
[a_1|\cdots|a_q]\longmapsto
\sum_{r=0}^{q}
[a_1|\cdots|a_r]\otimes[a_{r+1}|\cdots|a_q].
\]
This map is a coassociative chain map compatible with
the shuffle product.
Under the hypotheses of the corollary,
Theorems~\ref{ttth1} and~\ref{AWH} identify
$H_*(A\otimes A;k)$ with
$H_*(A;k)\otimes H_*(A;k)$ via the mutually inverse
isomorphisms $\mathrm{EZ}_*$ and $\mathrm{AW}_*$.
Thus the induced product and coproduct on $H_*(A;k)$
are compatible, with unit and counit given by the
inclusion and projection of $H_0(A;k)=k$.
An antipode $S$ is defined recursively by $S(1)=1$ and,
for a homogeneous element $h$ of positive degree, by
\[
S(h)=-h-\sum S(h')h'',
\qquad
\Delta_*(h)=h\otimes1+1\otimes h+\sum h'\otimes h''.
\]
Both factors in each term of the latter sum have
degree strictly smaller than $\deg(h)$, so the
recursion is well defined.
Coassociativity implies, by induction on degree, that
$S$ satisfies both antipode identities.
Consequently, $H_*(A;k)$ is a graded-commutative
Hopf algebra.
\ep


\def\cprime{$'$}

\end{document}